\documentclass[12pt,leqno]{amsart}
\usepackage{amsfonts,amsthm,amssymb,amsmath,amscd}
\usepackage{mathrsfs}
\usepackage{url}

\def\Z{{\mathbb Z}}
\def\Qq{{\mathbb Q}}
\def\R{{\mathbb R}}
\def\C{{\mathbb C}}
\def\HH{{\mathfrak H}}

\def\sq{\hbox{\rlap{$\sqcap$}$\sqcup$}}
\def\qed{\ifmmode\sq\else{\unskip\nobreak\hfil
         \penalty50\hskip1em\null\nobreak\hfil\sq
         \parfillskip=0pt\finalhyphendemerits=0\endgraf}\fi}

\newtheorem{theorem}{Theorem}
\newtheorem{lemma}[theorem]{Lemma}
\newtheorem{prop}[theorem]{Proposition}
\newtheorem{cor}[theorem]{Corollary}
\newtheorem{conj}{Conjecture}
\newtheorem{df}[theorem]{Definition}

\numberwithin{theorem}{section}
\numberwithin{equation}{section}

\usepackage{hyperref}
\hypersetup{
 pdfborder={0 0 0},
 setpagesize=false,
 bookmarks=true,
 bookmarksdepth=tocdepth,
 bookmarksnumbered=true,
 colorlinks=true,  
 linkcolor = blue,  
 citecolor= cyan,    
 pdftitle={},
 pdfsubject={},
 pdfauthor={},
 pdfkeywords={}}

\usepackage{titletoc}

\allowdisplaybreaks[1]

\title[Isomorphism between Jacobi forms of index $D_{2n+1}$]{
Isomorphism between Jacobi forms of index $D_{2n+1}$ and
elliptic modular forms of level $2$
}
\author{Shuichi Hayashida}
\date{\today}

\subjclass[2020]{Primary 11F50; Secondary 11F37, 11F20, 11F46}
\keywords{Jacobi forms, half-integral weight, eta-function, Siegel modular forms}

\begin{document}

\begin{abstract}
This paper has three main objectives:
(i) To establish an isomorphism between Jacobi forms of index $D_{2n+1}$ (lattice index) and elliptic modular forms of level $2$.
(ii) To provide an explicit formula for the Fourier coefficients of Jacobi--Eisenstein series of index $D_{2n+1}$.
(iii) To construct a holomorphic modular form of weight $3/2$ and level $8$ (and $4$) 
from the Zagier--Eisenstein series $\mathscr{F}$ of weight $3/2$ and level $4$.
Moreover, we show that 
the four functions $E^*_2$, $\eta^3$, $\theta^3$ and $\mathscr{F}$ have essentially the same Hecke eigenvalue $1+p$ for any odd prime $p$,
where $E^*_2$ is the non-holomorphic Eisenstein series of weight $2$, 
$\eta$ is the Dedekind eta-function and $\theta$ is the usual theta function.
This fact arises as a special case of the isomorphism of (i).
\end{abstract}

\maketitle

\setcounter{tocdepth}{1}
\tableofcontents

\section{Introduction}

\subsection{}

The isomorphism between Jacobi forms
of integer index and a certain subspace of elliptic modular forms of specific level
is given in~\cite{SZ}.
The present paper has three primary objectives.
The first is to generalize this isomorphism from the case of Jacobi forms of integer index to those 
 of index $D_r$ ($r$ odd), 
 where $D_r$ denotes the root lattice.
 More precisely, we will show that
the space of Jacobi forms of weight $k+\frac{r+1}{2}$ and index $D_r$
corresponds to a certain subspace of elliptic modular forms of weight $2k$
and level $2$ as Hecke modules (Theorems~\ref{thm:main}, \ref{thm:J_decom} and \ref{thm:J_new}). 
This isomorphism was previously conjectured by A.~Mocanu~\cite{Mocanu}.

The second objective is to provide an explicit formula for the Fourier coefficients of the Jacobi--Eisenstein series of index $D_r$
($r$ odd).
This formula is expressed as a linear combination of the Fourier coefficients of Cohen--Eisenstein series (Theorem~\ref{thm:Jacobi_Eisenstein}).

The third objective is to construct a holomorphic modular form of weight $3/2$ and level $8$ (and $4$) 
from the Zagier--Eisenstein series $\mathscr{F}$ of weight $3/2$ and level $4$ (Theorem~\ref{thm:E_3_2_8}).
Moreover, we show that the four functions $\eta^3$, $\theta^3$, $\mathscr{F}$ and $E^*_2$ possess essentially 
the same Hecke eigenvalue $1+p$ for any odd prime $p$,
where $\eta$ is the Dedekind eta-function, $\theta$ is the usual theta function 
and $E^*_2$ is the non-holomorphic Eisenstein series of weight~$2$ (Theorems~\ref{thm:k_1} and \ref{thm:E_3_2_8}).

We also provide explicit formulas for the maps from Jacobi forms of index $D_r$ ($r$ odd) to elliptic modular forms of level $2$
(Theorems~\ref{thm:map_J_even} and \ref{thm:map_J_odd}).

The remainder of this section is devoted to a more precise description of these results.

\subsection{Isomorphisms}

We denote by $J_{k,D_r}$ the space of Jacobi forms of weight $k$ of index $D_r$ 
(see $\S$\ref{ss:JF_Dr} and $\S$\ref{ss:basic}
 for the definition).
Note that Jacobi forms of lattice index are isomorphic to those of matrix index (see $\S$\ref{ss:basic}).

Let $M_k(N)$ (resp. $M_k^{\text{new}}(N)$) be the space of elliptic modular forms (resp. newforms) of weight $k \in \mathbb{N}$ with respect to $\Gamma_0(N)$.
For $\epsilon \in \{+, - \}$, we define
\begin{equation*}
  M_k^{\text{new},\epsilon}(N)
  := 
  \{
    f \in M_k^{\text{new}}(N) \ : \ (f|_k\left( \begin{smallmatrix} 0 & -1 \\ N & 0 \end{smallmatrix} \right))(\tau) := f\!\left(\frac{-1}{N\tau}\right)  (N^{1/2} \tau)^{-k} = \epsilon i^{-k} f(\tau)
  \}.
\end{equation*}
Recall that if $f \in M_k^{\text{new},\epsilon}(N)$, then the completed $L$-function $L^*(f,s)$ of $f$ satisfies
the functional equation $L^*(f,s) = \epsilon L^*(f,k-s)$ (see \cite[p. 136]{SZ}).

It is shown in \cite[Main Thm.]{SZ} that the isomorphism
\begin{equation*}
  J_{k+1,m}^{\text{new}} \cong M_{2k}^{\text{new},-}(m)
\end{equation*}
holds for any natural numbers $m$ and $k$, 
where $J_{k+1,m}^{\text{new}}$ denotes the subspace of newforms of Jacobi forms of weight $k+1$ and index $m$.

In particular, the following isomorphism of Hecke modules holds:
\begin{equation}\label{isom:EZ_2}
  J_{k+1,D_1} \cong
   J_{k+1,2}  \cong   M_{2k}^{\text{new},-}(2) \oplus M_{2k}^{-}(1), 
\end{equation}
where $J_{k+1,2}$ denotes the space of Jacobi forms of weight $k+1$ and index $2$.

A.~Mocanu~\cite{Mocanu} posed the following conjecture,
which generalizes the isomorphism (\ref{isom:EZ_2}).
This conjecture is supported by numerical evidence from the Euler factors of Jacobi forms in $J_{k,D_r}$.
\begin{conj}[{\cite[Conj. 3.30]{Mocanu}}]\label{conj:mocanu}
For every $k \geq 2$, the following isomorphisms of Hecke modules hold:
\begin{align*}
  J_{k+2,D_3} &\cong M_{2k}^{\text{new},-}(2) \oplus M_{2k}^{+}(1), \\
  J_{k+3,D_5} &\cong M_{2k}^{\text{new},+}(2) \oplus M_{2k}^{-}(1), \\
  J_{k+4,D_7} &\cong M_{2k}^{\text{new},+}(2) \oplus M_{2k}^{+}(1).
\end{align*}
\end{conj}
It is known that for odd natural numbers $r$ and $r'$, if $r \equiv r' \!\! \mod 8$, then there is an isomorphism
$J_{k+\frac{r+1}{2},D_{r}} \cong J_{k+\frac{r'+1}{2},D_{r'}}$ as Hecke modules (see Corollary~\ref{cor:J_r_general}).

For $D \equiv 0,1 \!\! \mod 4$, let $\left( \frac{D}{\cdot} \right)$ denote the Kronecker symbol.
For $D \in \{-4,8,-8\}$, these symbols satisfy
\begin{equation*}
\left( \frac{-4}{r} \right) = (-1)^{\frac{r-1}{2}}, \quad
\left( \frac{8}{r} \right) = (-1)^{\frac{r^2-1}{8}}, \quad
\left( \frac{-8}{r} \right) =  \left( \frac{-4}{r} \right) \left( \frac{8}{r} \right)
\end{equation*}
for odd integer $r$, while
$\left( \frac{-4}{r} \right) = \left( \frac{8}{r} \right) = \left( \frac{-8}{r} \right) = 0$ if $r$ is even.
We define
\begin{equation*}
\epsilon_2 := - \left( \frac{-8}{r} \right) \quad \text{and} \quad \epsilon_1 := - \left( \frac{-4}{r} \right).
\end{equation*}
Then, the isomorphisms in Conjecture~\ref{conj:mocanu} can be unified and rewritten as
\begin{equation*}
   J_{k+\frac{r+1}{2}, D_r} \cong M_{2k}^{\text{new},\epsilon_2}(2) \oplus M_{2k}^{\epsilon_1}(1).
\end{equation*}

Our main theorem (Theorem~\ref{thm:main}) establishes that this isomorphism indeed holds for all odd $r$.
\begin{theorem}\label{thm:main}
Conjecture~\ref{conj:mocanu} is true.
\end{theorem}
Theorem~\ref{thm:main} is a direct consequence of Theorems~\ref{thm:J_decom} and \ref{thm:J_new}.
More specifically, Theorem~\ref{thm:J_decom} describes the decomposition of the space of Jacobi forms,
while Theorem~\ref{thm:J_new} identifies the newform part with the corresponding space of elliptic modular forms.

To prove Theorem~\ref{thm:main},  we first consider the composition of two maps:
\begin{equation*}
 M_{2k}(1) \xrightarrow{\text{Ikeda lift}} 
 \left\{ \begin{smallmatrix}
  \text{Siegel modular forms of} \\
  \text{weight $k+\frac{r+1}{2} \in 2\mathbb{N}$ and degree $r+1$} \end{smallmatrix}\right\}
  \xrightarrow{\text{FJ coeff.}} J_{k+\frac{r+1}{2},D_r},
\end{equation*}
where the first map is the Ikeda lift, 
while the second map extracts a Fourier--Jacobi coefficient.
We refer to the image of this composition as the space of  ``oldforms'' in $J_{k+\frac{r+1}{2},D_r}$.
We establish the following result: 
\begin{theorem}\label{thm:J_decom}
For $r \in \{1, 3, 5, 7\}$,
we have the decomposition
\begin{equation*}
  J_{k+\frac{r+1}{2},D_r} = J_{k+\frac{r+1}{2},D_r}^{\text{new}} \oplus J_{k+\frac{r+1}{2},D_r}^{\text{old}}
\end{equation*}
where the following isomorphisms of Hecke modules hold:
\begin{equation}\label{id:oldforms}
  \begin{aligned}
  J_{k+1,D_1}^{\text{old}} &  \cong &  J_{k+3,D_5}^{\text{old}} & \cong & 
   M_{2k}^{-}(1)
   &=&
  \begin{cases}
    \{ 0 \}, & \text{ if $k$ is even}, \\
    M_{2k}(1), & \text{ if $k$ is odd},
  \end{cases}
  \\
    J_{k+2,D_3}^{\text{old}} & \cong & J_{k+4,D_7}^{\text{old}} & \cong & 
     M_{2k}^{+}(1)
    &=&
    \begin{cases}
    M_{2k}(1), & \text{ if $k \geq 2$ is even}, \\
    \{ 0 \} , & \text{ if $k$ is odd}.
  \end{cases}
    \end{aligned}
\end{equation}
Here, $J_{k,D_r}^{\text{new}}$ denotes the orthogonal complement of $J_{k,D_r}^{\text{old}}$ in $J_{k,D_r}$ with respect to
the Petersson scalar product 
$($see \cite[p.~303]{Mocanu_JN}, \cite[p.~68]{ali} or \cite[p.~13]{Mocanu} for the definition of the Petersson scalar product$)$.
\end{theorem}
Theorem~\ref{thm:J_decom} will be proved in Section~\ref{sec:Ikeda_lift}.

Note that the dimension of the space of Jacobi--Eisenstein series is at most $1$ $($see Lemma~\ref{lem:dim_Eisenstein}$)$.
Furthermore, it is well known that the Jacobi--Eisenstein series is orthogonal to the space of Jacobi cusp forms 
with respect to the Petersson scalar product $($see~\cite[Thm. 4.1]{Mocanu_JN}$)$. 

If the weight $k + \frac{r+1}{2}$ is odd, the fact that $J_{k+\frac{r+1}{2},D_r}^{\text{old}} = \{ 0 \}$
follows directly from (\ref{id:oldforms}).

In the case $k=0$ of Theorem~\ref{thm:J_decom},
it is known that $J_{2,D_3} = \{ 0\}$ and $J_{4,D_7} = \C E_{4,D_7}$
for a certain Jacobi form $E_{4,D_7}$
 (see \cite[Thm.~3.29]{Mocanu}, which refers to results in~\cite{BS}).
Such a Jacobi form $E_{4,D_7}$ can also be obtained as a Fourier--Jacobi coefficient of the Siegel theta series
of weight $4$ and degree $8$.

We denote by $S_k(N)$ (resp. $S_k^{\text{new}}(N)$) the subspace of cusp forms in $M_k(N)$ (resp. $M_k^{\text{new}}(N)$).
For $k \geq 2$, the equality $M_{2k}^{\text{new}}(2) = S_{2k}^{\text{new}}(2)$ holds (see Lemma~\ref{lem:E_2k_new}).

Similarly, we denote by $J_{k,D_r}^{\text{cusp}}$ the space of Jacobi cusp forms in $J_{k,D_r}$ (see Section~\ref{sec:Jacobi_lattice} for the definition)
and define $J_{k,D_r}^{\text{cusp, new}} := J_{k,D_r}^{\text{cusp}} \cap J_{k,D_r}^{\text{new}}$.
In Lemma~\ref{lem:cusp_cusp_new}, 
we will show that $J_{k+\frac{r+1}{2},D_r}^{\text{new}} \subset J_{k+\frac{r+1}{2},D_r}^{\text{cusp}}$ for $k \geq 2$.
Consequently, we have the equality $J_{k+\frac{r+1}{2},D_r}^{\text{cusp, new}} = J_{k+\frac{r+1}{2},D_r}^{\text{new}}$ for $k \geq 2$.

To prove Theorem~\ref{thm:main}, we next consider the composition of the following two isomorphisms:
\begin{equation}\label{id:comp_new}
  J_{k+\frac{r+1}{2},D_r}^{\text{cusp, new}}
  \cong
  \left\{ 
  \begin{smallmatrix}
   \text{certain modular forms} \\
   \text{of weight $k+\frac12$}
  \end{smallmatrix}
  \right\}
  \cong
  S_{2k}^{\text{new},\epsilon_2}(2). 
\end{equation}
The second isomorphism in this composition is an analogue of the Shimura correspondence, 
which has been established in~\cite{UY} and \cite{YY}. 

Let $S_{k+\frac12}^+(8)$ be the Kohnen plus space of level $8$ (see $\S$\ref{ss:half_int_theta} 
for the definition).
This space admits the following decomposition:
$S_{k+\frac12}^+(8) = S_{k+\frac12}^{+,-1}(8) \oplus S_{k+\frac12}^{+,-5}(8)$ for odd $k$,
and $S_{k+\frac12}^+(8) = S_{k+\frac12}^{+,-3}(8) \oplus S_{k+\frac12}^{+,-7}(8)$ for even $k$
(see $\S$\ref{ss:half_int_theta}, for the notation).
Any form $g = \sum_n c_g(n) q^n$ in $S_{k+\frac12}^{+,-r}(8)$ is characterized by the condition that
\begin{equation*}
c_g(n) = 0 \mbox{ unless } n \equiv 0, 4, -r \!\! \mod 8.
\end{equation*}
The subspace $S_{k+\frac12}^{+,-r}(8)$ was introduced in~\cite[Prop. 4]{UY}. 
Let $S_{k+\frac12}^{\text{new},+,-r}(8)$ be the subspace of newforms in $S_{k+\frac12}^{+,-r}(8)$ 
(see $\S$\ref{ss:Shimura_even} for the definition of newforms in $S_{k+\frac12}^+(8)$, which follows \cite{UY}).
The Shimura correspondence $S_{k+\frac12}^{\text{new},+,-r}(8) \cong S_{2k}^{\text{new},\epsilon_2}(2)$ 
was essentially established in~\cite[Thm. 1]{UY}
(see also Theorem~\ref{thm:UY}).

On the other hand, 
the Shimura correspondence $\eta^{3r} M_{k+\frac{1-3r}{2}}(1) \cong S_{2k}^{\text{new},\epsilon_2}(2)$ was essentially established 
in~\cite[Thm.~2]{YY} (see Theorem~\ref{thm:YY}).
Here, $\eta$ denotes the Dedekind eta-function defined by $\eta(\tau) := q^{\frac{1}{24}} \prod_{n=1}^\infty (1-q^n)$ with $q = e^{2 \pi i \tau}$.
Note that on the space $\eta^{3r} M_{k+\frac{1-3r}{2}}(1)$, we employ twisted Hecke operators 
(see $\S$\ref{ss:Hecke_op} for their definition).

By combining these two types of compositions in (\ref{id:comp_new}), we obtain: 
\begin{theorem}\label{thm:J_new}
The following isomorphisms of Hecke modules hold:
\begin{itemize}
\item[(i)] If $k \geq 2$ is even, then 
\begin{equation*}
  \begin{aligned}
     J_{k+1,D_1}^{\text{cusp, new}}  \cong  J_{k+2,D_3}^{\text{cusp,new}} & \cong  S_{k+\frac12}^{\text{new},+,-3}(8)
      \cong  \eta^{21} M_{k-10}(1)  \cong  S_{2k}^{\text{new},-}(2), \\
     J_{k+3,D_5}^{\text{cusp, new}}  \cong  J_{k+4,D_7}^{\text{cusp, new}} & \cong  S_{k+\frac12}^{\text{new},+,-7}(8)
      \cong  \eta^{9} M_{k-4}(1)  \cong  S_{2k}^{\text{new},+}(2)     .
\end{aligned}
\end{equation*}
\item[(ii)] If $k \geq 3$ is odd, then
\begin{equation*}
  \begin{aligned}
     J_{k+1,D_1}^{\text{cusp, new}}  \cong  J_{k+2,D_3}^{\text{cusp, new}} & \cong  S_{k+\frac12}^{\text{new},+,-1}(8)
      \cong  \eta^{15} M_{k-7}(1)  \cong  S_{2k}^{\text{new},-}(2), \\
     J_{k+3,D_5}^{\text{cusp, new}}  \cong  J_{k+4,D_7}^{\text{cusp, new}} & \cong  S_{k+\frac12}^{\text{new},+,-5}(8)
      \cong  \eta^{3} M_{k-1}(1)  \cong  S_{2k}^{\text{new},+}(2) .
  \end{aligned}
\end{equation*}
\end{itemize}
\end{theorem}

Theorem~\ref{thm:J_new} follows from Theorems~\ref{thm:odd_weight_case} and \ref{thm:even_weight_case}.

\subsection{Fourier coefficients of Jacobi--Eisenstein series}
In  Theorem~\ref{thm:Jacobi_Eisenstein}, 
we provide an explicit formula for the Fourier coefficients of the Jacobi--Eisenstein series $E_{k+\frac{r+1}{2},D_r}$ in $J_{k+\frac{r+1}{2},D_r}$.
While such formulas have already been obtained in \cite[Thm. 5.1]{Mocanu_JN} for general lattice index $\underline{L}$, 
we provide a more precise formula for the case where $\underline{L} = D_r$ with odd $r$.

Let $\mathscr{H}_k$ denote the Cohen--Eisenstein series of weight $k+\frac12$ and level $4$ (see $\S$\ref{sec:Fourier_coeff_Jacobi_Eisenstein}).
From $\mathscr{H}_k$, we construct a modular form $\mathscr{H}^*_k$ of weight $k+\frac12$ and level $8$ 
by applying the $U_k(4)$-operator, which was introduced in~\cite{UY}.
Consequently, the Fourier coefficients of $\mathscr{H}^*_k$ are given as linear combinations of those of $\mathscr{H}_k$.
On the other hand, the Jacobi--Eisenstein series $E_{k+\frac{r+1}{2},D_r}$ corresponds to $\mathscr{H}^*_k$ 
under the isomorphism (see Proposition~\ref{prop:J_M_even}).
Thus, the Fourier coefficients of $E_{k+\frac{r+1}{2},D_r}$ can be expressed as linear combinations of 
the Fourier coefficients of $\mathscr{H}_k$ (see Theorem~\ref{thm:Jacobi_Eisenstein}).

\subsection{Modular forms of weight $3/2$ and level $8$}\label{ss:half_3_2_and_2}
We also establish Conjecture~\ref{conj:mocanu} for the case $k=1$ (Theorem~\ref{thm:k_1}).
In particular, we show that the four functions $\eta^3$, $\theta^3$, $\mathscr{F}$ and $E^*_2$
possess essentially the same Hecke eigenvalue $1+p$ for any odd prime $p$.
Here, we define the non-holomorphic Eisenstein series of weight $2$ as
$E^*_2(\tau) := 1 - 24 \sum_{n\geq1} \sigma(n) q^n - \frac{3}{\pi v} $, where $v = \text{Im}(\tau)$, 
and where $\sigma(n)$ denotes the divisor sum function.
Furthermore,
$\theta(\tau) := \sum_{n \in \Z} q^{n^2}$
is the usual theta function, 
and $\mathscr{F}$ denotes the Zagier--Eisenstein series of weight $3/2$ 
and level $4$ (see $\S$\ref{sec:Fourier_coeff_Jacobi_Eisenstein} for the definition).

Let $M_{k+\frac12}^{+,-r}(8)$ denote the subspace of the Kohnen plus space $M_{k+\frac12}^{+}(8)$
of weight $k+\frac12$ and level $8$
(see $\S$~\ref{ss:half_int_theta} for its definition).
Note that $S_{k+\frac12}^{+,-r}(8) = M_{k+\frac12}^{+,-r}(8) \cap S_{k+\frac12}^{+}(8)$.
We show that 
a holomorphic modular form $E_{3/2}^{(8)} \in M_{3/2}^{+,-5}(8) $ of weight $3/2$ and level $8$
can be constructed from $\theta^3$ as well as from $\mathscr{F}$ (see Theorem~\ref{thm:E_3_2_8}).

\begin{theorem}\label{thm:k_1}
As an analogue of Theorem~\ref{thm:J_new} for the case $k=1$, the following holds:
\begin{gather*}
  \dim J_{2,D_1} = \dim J_{3,D_3} = \dim M_{3/2}^{+,-1}(8) = \dim \left(\eta^{15} M_{-6}(1)\right) = \dim M_{2}^-(2) = 0, \\
  \dim J_{4,D_5}^{\text{new}} = \dim J_{5,D_7}^{\text{cusp,new}} 
  = \dim M_{3/2}^{+,-5}(8) 
  = \dim \left(\eta^3 M_0(1)\right)
  = \dim M_{2}^{+}(2) 
  = 1.
\end{gather*}
Thus, we have the following isomorphisms of Hecke modules:
\begin{equation*}
  J_{4,D_5}^{\text{new}} \cong J_{5,D_7}^{\text{cusp,new}} 
  \cong  M_{3/2}^{+,-5}(8) 
  \cong  \eta^3 M_0(1) 
  \cong  M_{2}^{+}(2).
\end{equation*}
Here, the spaces are spanned by the functions
$E_{4,D_5}$, $\psi_{5,D_7}$,  $E_{3/2}^{(8)}$, $\eta^3$, and $E_{2}^{(2)}$, respectively, 
such that $M_2^{+}(2) = \C E_2^{(2)}$ 
with $E_2^{(2)}(\tau) := 2 E^*_2(2\tau) - E^*_2(\tau)$. 
Note that $E_{4,D_5}$, $E_{3/2}^{(8)}$ and $E_2^{(2)}$ are neither Jacobi cusp forms nor (elliptic) cusp forms,
whereas $\psi_{5,D_7}$ is a Jacobi cusp form.
In particular, it follows that $J_{4,D_5}^{\text{new}} \, {\not \subset}\, J_{4,D_5}^{\text{cusp,new}}$.

Consequently, the five functions $E_{4,D_5}$, $\psi_{5,D_7}$, $E_{3/2}^{(8)}$, $\eta^3$ and $E_2^{(2)}$
have the same Hecke eigenvalue $1+p$ for any odd prime $p$.
\end{theorem}
Theorem~\ref{thm:k_1} will be proved in $\S$\ref{ss:proof_thm_k_1}.

The Jacobi form $E_{4,D_5} \in J_{4,D_5}$ in Theorem~\ref{thm:k_1} is uniquely determined up to a constant multiple.
We refer to \cite{BS} for the definition of $E_{4,D_5}$ (see also \cite[pp.~76--77]{Mocanu}).
An explicit formula for its Fourier coefficients is given in Theorem~\ref{thm:Jacobi_Eisenstein_k_1}.

The modular form 
$E_{3/2}^{(8)} \in M_{3/2}^{+,-5}(8)$ in Theorem~\ref{thm:k_1} is uniquely determined up to a constant multiple.
Since its constant term is non-zero, 
we normalize $E_{3/2}^{(8)}$ such that $E_{3/2}^{(8)}(\tau) = 1 + \mathcal{O}(q)$.
In Theorem~\ref{thm:E_3_2_8}, we show that
$E_{3/2}^{(8)}$ can be constructed from $\theta^3$ as well as from the Zagier--Eisenstein series $\mathscr{F}$ of weight $3/2$.
By virtue of Proposition~\ref{prop:J_M_even}, the Fourier coefficients of $E_{3/2}^{(8)}$ coincide with those of $E_{4,D_5}$
(up to a suitable normalization).

We now describe the Fourier coefficients of $E_{3/2}^{(8)}$.
For $m \in \mathbb{N}$ and $N \in \Z_{\geq 0}$, we let
\begin{equation*}
  r_m(N) := \#\left\{ (x_1,...,x_m) \in \Z^m \ : \ x_1^2 + \cdots + x_m^2 = N \right\}
\end{equation*}
denote the number of ways to represent $N$ as a sum of $m$ squares.
It is well known that 
for $m \leq 3$, we have $r_m(4N) =r_m(N)$ for any $N \in \mathbb{N}$.
Let $H(N)$ be the standard Hurwitz class number for the discriminant $-N$ (see \cite[p.~273]{cohen}
for the definition).

For a function $f$ on the upper half-plane $\HH$ satisfying $f(\tau+1) = f(\tau)$, 
we define the $U(4)$ operator by
\begin{equation*}
f|U(4) := \frac14 \sum_{a \!\!\!\!  \mod 4} f\!\left(\frac{\tau+a}{4}\right).
\end{equation*}
For a formal power series $f = \sum_{n \in \Z} a(n) q^n$, the operator $\wp_k$, defined by
\begin{equation*}
f|\wp_k := \sum_{(-1)^k n \equiv 0, 1 \!\! \mod 4} a(n) q^n , 
\end{equation*}
was introduced in \cite[p. 2]{UY}.
The operator $U_k(4) := U(4) \wp_k$ was also introduced in \cite[p. 2]{UY}.
Specifically, 
if $f(\tau) = \sum_{n=0}^\infty a(n) q^n$, then the action of $U_k(4)$ is given by
\begin{equation*}
 (f|U_k(4))(\tau) = \sum_{\substack{  n \geq 0 \\ (-1)^k n \equiv 0, 1 \!\!\!\! \mod 4}} a(4n) q^n.
\end{equation*}

\begin{theorem}\label{thm:E_3_2_8}
The modular form $E_{3/2}^{(8)} = 1 + \mathcal{O}(q) \in M_{3/2}^{+,-5}(8)$ admits the following expressions: 
\begin{equation*}
  E_{3/2}^{(8)}(\tau) = \sum_{\begin{smallmatrix} n \geq 0 \\ n \equiv 0, 3 \!\!\!\! \mod 4 \end{smallmatrix}} r_3(n) q^n
   =  \sum_{\begin{smallmatrix} n \geq 0 \\ n \equiv 0, 3 \!\!\!\! \mod 4 \end{smallmatrix}} r_3(4n) q^n
   =  (\theta^3|U_1(4))(\tau).
\end{equation*}
Furthermore, $E_{3/2}^{(8)}$ can be expressed as
\begin{equation*}
  E_{3/2}^{(8)}(\tau) 
  = 
  12 \sum_{\begin{smallmatrix} n \geq 0 \\ n \equiv 0, 3 \!\!\!\! \mod 4 \end{smallmatrix}} (H(4n) - 2 H(n)) q^n
   =  
  12 ( \mathscr{F}|U_1(4) - 2\mathscr{F} )(\tau),
\end{equation*}
where $\mathscr{F}$ is the non-holomorphic modular form of weight $3/2$ introduced in~\cite{zagier}
(see also \S\ref{ss:E_3_2_8} for this definition).

The form $E_{3/2}^{(8)}$ satisfies $E_{3/2}^{(8)} |U_1(4) = E_{3/2}^{(8)}$.

In particular, for any natural number $N$, we have
\begin{equation}\label{id:r3_HH}
 r_3(N) = 12(H(4N) - 2 H(N)),
\end{equation}
where we set $H(N) := 0$
if
$-N \equiv 2, 3 \!\! \mod 4$.
\end{theorem}
Theorem~\ref{thm:E_3_2_8} will be proved 
in~$\S$\ref{ss:E_3_2_8}.

Combining the identity (\ref{id:r3_HH}) with the fact that $\theta(\tau)^3 = \sum_{n=0}^\infty r_3(n) q^n$ yields the following corollary.
\begin{cor}
The holomorphic function
\begin{equation*}
  \sum_{n = 0}^\infty (H(4n) - 2 H(n)) q^n
\end{equation*}
is a modular form of weight $3/2$ and level $4$, which does not belong to the Kohnen plus space.
\end{cor}

The following formula follows from the identities~(\ref{id:r3_HH}) and (\ref{id:H_N_Dff}) in $\S$\ref{ss:cohen_type_level_8},
and it is also consistent with the results in \cite[p.~273--274]{cohen}.
\begin{cor}
For any $N \in \mathbb{N}$, we have
\begin{equation}\label{id:Cohen_rep}
 r_3(N)
 = 12 H(|D| f_1^2) \left(1 - \left( \frac{8}{D} \right) \right),
\end{equation}
where $-N = D f^2$ with a fundamental discriminant $D < 0$ and $f \in \frac12 \mathbb{N}$. 
Here, we write $f = 2^e f_1$ with an odd integer $f_1$
and an exponent $e \in \Z_{\geq -1}$.
Note that if $-N \equiv 2,3 \!\! \mod 4$, then $D \equiv 0 \mod 4$, $f \in \mathbb{N} - \frac12$ and $\left( \frac{8}{D} \right) = 0$.
\end{cor}
Note that the identity~(\ref{id:Cohen_rep}) was derived in \cite[Proposition 5.3.10]{cohen:computation}
for the specific case where $-N < -4$ is a fundamental discriminant.

\subsection{Application to arithmetic functions}
Since $\theta E_{3/2}^{(8)} \in M_2(8)$ and $M_2(8) = \C E_2^{(2)}(\tau) \oplus \C E_2^{(2)}(2\tau) \oplus \C E_2^{(2)}(4\tau)$,
we can determine the Fourier coefficients of $\theta E_{3/2}^{(8)}$ by comparing its initial Fourier coefficients with 
those of the basis of $M_2(8)$. 
This yields the following identity:
\begin{equation*}
  \theta(\tau) E_{3/2}^{(8)}(\tau) = \frac{1}{12} E_2^{(2)}(\tau) - \frac{1}{12} E_2^{(2)}(2\tau) + E_2^{(2)}(4 \tau).
\end{equation*}
This implies that 
for any $N \in \mathbb{N}$, the following identity holds: 
\begin{equation*}
  \frac12 \!\!\!\!\!\!
  \sum_{\begin{smallmatrix} s \in \Z \\ N - s^2 \equiv 0, 3 \!\! \mod 4\end{smallmatrix}}
  \!\!\!\!\!\!
   r_3(N - s^2)
  =
  \sigma(N) - 3\, \sigma\!\left( \frac{N}{2} \right) + 14\, \sigma\!\left( \frac{N}{4} \right) - 24\, \sigma\!\left( \frac{N}{8} \right),
\end{equation*}
where we set $\sigma(M) = 0$ if $M {\not \in} \mathbb{N}$.
In particular, if $N$ is odd, the identity simplifies to
\begin{equation*}
  \sigma(N) = 
   \frac12 
   \!\!\!\!\!\! 
   \sum_{\begin{smallmatrix} s \in \Z \\ N - s^2 \equiv 0, 3 \!\! \mod 4\end{smallmatrix}}
   \!\!\!\!\!\! 
   r_3(N - s^2).
\end{equation*}
By combining this with the well-known formula 
$\sum_{s \in \Z} r_3(N-s^2) = r_4(N) = 8 \sigma(N) - 32 \sigma(N/4)$, we obtain the following formulas:
\begin{equation*}
  \sum_{\substack{s \in \Z \\ s \equiv 1 \!\! \mod 2} } r_3(N-s^2) 
  =
  \begin{cases}
    2 \sigma(N), & \text{if } N \equiv 1 \mod 4, \\
    6 \sigma(N), & \text{if } N \equiv 3 \mod 4
  \end{cases}
\end{equation*}
and
\begin{equation*}
  \sum_{\substack{s \in \Z \\ s \equiv 0 \!\! \mod 2} } r_3(N-s^2) 
  =
  \begin{cases}
    6 \sigma(N), & \text{if } N \equiv 1 \mod 4, \\
    2 \sigma(N), & \text{if } N \equiv 3 \mod 4.
  \end{cases}
\end{equation*}
Note that these equations are equivalent to formulas for the number of representations 
of a given odd integer $N$ as
$N = a_1^2 + a_2^2 + a_3^2 + a_4^2$ , where $a_4$ is restricted to be odd and even, respectively.
Such formulas were originally derived by Liouville in the 19th century (see, e.g., \cite{williams}).

\subsection{Maps from Jacobi forms to elliptic modular forms}
We set $M_{2}^{\text{new},\epsilon_2}(2) := M_{2}^{\epsilon_2}(2)$.
In $\S$\ref{sec:S_d_0} we will provide explicit formulas for the maps from
$J_{k+\frac{r+1}{2},D_r}$ to $M_{2k}^{\text{new},\epsilon_2}(2) \oplus M_{2k}^{\epsilon_1}(1)$
that are compatible with the action of Hecke operators.
These maps are defined as follows: 
\begin{theorem}\label{thm:map_J_even}
Assume that $k + \frac{r+1}{2}$ is an even integer.
For $\phi \in J_{k+\frac{r+1}{2},D_r}$ 
we set
\begin{equation*}
  A\left( N \right)
  :=
  \begin{cases}
    c(n',r'), & \mbox{if $N = 8(n' - \beta(r'))$ and $N \equiv 0, 4 \mod 8$}, \\
    2 c(n',r'), & \mbox{if $N = 8(n' - \beta(r'))$ and $N \equiv -r \mod 8$},
  \end{cases} 
\end{equation*}
where $c(n',r')$ denotes the $(n',r')$-th Fourier coefficient of $\phi$ for $(n',r') \in \Z \times D_r^{\sharp}$
$($see $\S$\ref{ss:JF_Dr} for the notations of $\beta$ and $D_r^\sharp$$)$.

Note that $\sum_{N \equiv 0,4,-r \!\! \mod 8} A(N) q^N \in M_{k+\frac12}^{+,-r}(8)$.

Then, for any fundamental discriminant $(-1)^k d_0$ such that $d_0 > 0$ and 
$d_0 \equiv 0 \mod 4$,
the map $S_{d_0}$ defined by
\begin{equation*}
  S_{d_0}(\phi) :=
  \frac{A(0) L(1-k,\left( \frac{(-1)^k d_0}{ * }\right)) }{2 (1 + \left( \frac{8}{r} \right) 2^k) } 
  +
  \sum_{n=1}^\infty
  \left( 
  \sum_{d|n } \left( \frac{(-1)^k d_0}{d} \right) d^{k-1} A\!\left( \frac{n^2}{d^2} d_0 \right) 
  \right)
  q^n
\end{equation*}
is a linear map
\begin{equation*}
 S_{d_0} \ : \ J_{k+\frac{r+1}{2},D_r} \to M_{2k}^{\text{new},\epsilon_2}(2) \oplus M_{2k}^{\epsilon_1}(1).
\end{equation*}

Furthermore, the map $S_{d_0}$ sends 
\begin{itemize}
\item[$\cdot$]
$J_{k+\frac{r+1}{2},D_r}^{\text{cusp}}$ to $S_{2k}^{\text{new},\epsilon_2}(2) \oplus S_{2k}^{\epsilon_1}(1)$,
\item[$\cdot$]
$J_{k+\frac{r+1}{2},D_r}^{\text{new}}$ to $M_{2k}^{\text{new},\epsilon_2}(2)$, and
\item[$\cdot$]
$J_{k+\frac{r+1}{2},D_r}^{\text{old}}$ to $M_{2k}^{\epsilon_1}(1)$, 
\end{itemize}
and
commutes with the action of Hecke operators.
\end{theorem}

Note that if $k+\frac{r+1}{2}$ is odd, then
$J_{k+\frac{r+1}{2},D_r} = J_{k+\frac{r+1}{2},D_r}^{\text{cusp, new}}$ (See Theorem~\ref{thm:J_decom} and Lemma~\ref{lem:dim_Eisenstein}).
\begin{theorem}\label{thm:map_J_odd}
Assume that $k + \frac{r+1}{2}$ is an odd integer.
For $\phi \in J_{k+\frac{r+1}{2},D_r}^{\text{cusp, new}}$ and $N = 8 n' - r$,
we set $A(N) := c(n', v_1)$, where $c(n',r')$ denotes the $(n',r')$-th Fourier coefficient of $\phi$ for $(n',r') \in \Z \times D_r^{\sharp}$
$($see $\S$\ref{sec:decom} for the notation $v_1$$)$.

We note that $\sum_{n' \in \Z} c(n',v_1) q^{(8n'-r)/8} \in \eta^{24-3r} M_{k-12+\frac{3r+1}{2}}(1)$.

For any fundamental discriminant $(-1)^{k-1} d_0$ such that $d_0 > 0$ and $d_0 \equiv -r \mod 8$,
we define
\begin{equation*}
  S_{d_0}(\phi) :=
  \sum_{n=1}^\infty
    \left( 
  \left( \left( \frac{8}{r}\right) 2^{k-1}\right)^{e_2}
  \left( \frac{-4}{n_1} \right)
  \sum_{ d|n_1} \left( \frac{(-1)^k d_0 }{d} \right) d^{k-1}
   A\!\left( \frac{n_1^2}{d^2} d_0 \right) 
  \right)
  q^n,
\end{equation*}
where $n = 2^{e_2} n_1$ 
with $n_1$ being an odd integer and $e_2 \in \Z_{\geq 0}$. 
Then $S_{d_0}$ is a linear map 
\begin{equation*}
 S_{d_0} \ : \ J_{k+\frac{r+1}{2},D_r}^{\text{cusp, new}} \to S_{2k}^{\text{new},\epsilon_2}(2).
\end{equation*}
Moreover, the map $S_{d_0}$ commutes with the action of Hecke operators.
\end{theorem}

Theorems~\ref{thm:map_J_even} and \ref{thm:map_J_odd} will be proved in $\S$\ref{sec:S_d_0}.

We note that
maps similar to those in Theorems~\ref{thm:map_J_even} and \ref{thm:map_J_odd} for the case $D_1$ 
have been obtained 
in the context of Jacobi forms of matrix index in~\cite{bringmann}.

\vspace{1cm}

The present paper is organized as follows: 
In Section~\ref{sec:Jacobi_lattice}, we review
the definitions of Jacobi forms, modular forms of half-integral weight, 
and Hecke operators. 
In Section~\ref{sec:Ikeda_lift}, we prove Theorem~\ref{thm:J_decom} by means of the Ikeda lift and Fourier--Jacobi expansions. 
Section~\ref{sec:decom} is devoted to describing some basic properties of Jacobi forms of index $D_r$. 
In Sections~\ref{sec:JF_odd_weight} and \ref{sec:JF_even_weight}, we establish the isomorphisms 
between elliptic modular forms and Jacobi forms of index $D_r$ of odd and even weights, respectively.
In Section~\ref{sec:proof_main_k_1}, we provide the proofs of Theorems~\ref{thm:main} and \ref{thm:k_1}. 
Section~\ref{sec:Fourier_coeff_Jacobi_Eisenstein} presents an explicit formula 
for the Fourier coefficient of Jacobi--Eisenstein series in $J_{k,D_r}$ and gives the proof of Theorem~\ref{thm:E_3_2_8}.
Finally, Theorems~\ref{thm:map_J_even} and \ref{thm:map_J_odd} are proved in Section~\ref{sec:S_d_0}.

\vspace{1cm}

\noindent
\textbf{Acknowledgements: }
the author would like to express his sincere thanks to Tomoyoshi Ibukiyama, Nils-Peter Skoruppa and Haigang Zhou for their valuable comments.
This manuscript was prepared with the assistance of Gemini to improve linguistic readability after its initial drafting. 
This work is partially supported by JSPS KAKENHI Grant Number JP19K03419.

\section{Jacobi forms of lattice index}\label{sec:Jacobi_lattice}

\subsection{Basic properties}\label{ss:basic}

We denote by $\HH$ the Poincar\'e upper half space.
Let $e(x) := e^{2\pi i x}$, and we write $q := e(\tau)$ for $\tau \in \HH$.
For each $z\in \C$, we fix the argument to satisfy $-\pi < \arg(z) \leq \pi$ and set
$z^{m/2} := (z^{1/2})^m$ for $m \in \Z$.

Let $L$ be a free $\Z$-module of finite rank $n$, equipped with a $\Z$-valued symmetric bilinear form $\beta$.
Throughout this paper, we assume that $\beta$ is positive-definite.
Such a pair $\underline{L} = (L,\beta)$ is called an \textit{integral lattice} of rank $n$.
By abuse of notation, we also define $\beta(x) := \frac12 \beta(x,x)$.
In what follows, we occasionally denote the lattice structure by the quadratic form 
$x \mapsto \beta(x)$ instead of the bilinear form  $(x,y) \mapsto \beta(x,y)$.
If $\beta(x) \in \Z$ for all $x \in L$, then $\underline{L}$ is called an \textit{even} integral lattice.

For example, if $M$ is a positive-definite half-integral symmetric matrix of size $n$,
then
\begin{equation*}
  (\Z^n, (x,y)  \mapsto x (2M) {^t y})
\end{equation*}
is an even integral lattice.

The dual lattice of $L$ is defined as
\begin{equation*}
  L^\sharp := \left\{ x \in L \otimes_\Z \Qq \ : \ \beta(x,y) \in \Z \mbox{ for any $y \in L$} \right\}.
\end{equation*}

We fix a $\Z$-basis $\{e_i\}_i$ of $L$ and identify $L \otimes_\Z \C$ with $\C^n$ 
via the map $(x_1,...,x_n) \mapsto \sum_{i} x_i e_i$.

Let $k$ be an integer.

\begin{df}\label{df:jacobi_lattice}
 A holomorphic function $\phi(\tau,z)$ on $\HH \times \C^n$ is called a Jacobi form of weight $k$ and index $\underline{L}$ 
 if it satisfies the following three conditions: 
 \begin{itemize}
   \item[(i)] \quad For all $A = \left(\begin{smallmatrix} a & b \\ c & d \end{smallmatrix} \right) \in \mbox{SL}(2,\Z)$, 
   we have
     \begin{equation*}
       (\phi|_{k,\underline{L}} A)(\tau,z) :=
        (c\tau + d)^{-k}
          e\!\left(\frac{-c \beta(z)}{c\tau + d}\right)
              \phi\! \left(\frac{a\tau+b}{c\tau+d}, \frac{1}{c\tau+d} z \right)
        \ =\ \phi(\tau,z) .
     \end{equation*}
   \item[(ii)] \quad
   For all $\lambda$, $\mu$ $\in L$, we have 
   \begin{equation*}
    (\phi|_{k,\underline{L}}(\lambda,\mu))(\tau,z) := 
     \phi(\tau, z + \lambda \tau + \mu) e(\tau\beta(\lambda) + \beta(\lambda,z))
     \ =\
     \phi(\tau,z) .
   \end{equation*}
   \item[(iii)] \quad The function $\phi(\tau,z)$ has a Fourier expansion of the form
   \begin{equation*}
     \phi(\tau,z) = \sum_{n' \in \Z} \sum_{r' \in L^\sharp} c(n',r') e(n' \tau + \beta(r',z)),
   \end{equation*}
   where $c(n',r') = 0$ unless $n' - \beta(r') \geq 0 $.
 \end{itemize}
 The complex numbers $c(n',r')$ are called the $(n',r')$-th Fourier coefficients of $\phi$.
 
 If $c(n',r') = 0$ whenever $n' - \beta(r') = 0 $, then $\phi$ is called a Jacobi cusp form.
\end{df}
 We denote by $J_{k,\underline{L}}$ $($resp. $J^{\text{cusp}}_{k,\underline{L}}$$)$ the space of all Jacobi forms 
 $($resp. Jacobi cusp forms$)$ of weight $k$ and index $\underline{L}$.

In Definition~\ref{df:jacobi_lattice}, if the lattice $\underline{L} = (L,\beta)$ is given by  
$
  (\Z^n, x \mapsto \beta(x) = x M {^t x})
$
with a positive-definite half-integral symmetric matrix $M$ of size $n$,
then $\phi$ is called a Jacobi form of weight $k$ and index $M$ (see \cite{Zi}).
We denote the space of all such Jacobi forms by $J_{k,M}$.
 
 Let $\underline{L} = (L,\beta)$ be an integral lattice with a $\Z$-basis $\{e_1,...,e_n\}$.
 The matrix $2M = (\beta(e_i,e_j))_{i,j}$ is called the {\it Gram matrix} of $\beta$ with respect to $\{e_1,...,e_n\}$.
  Under this basis, we have the identifications 
  $L \cong \Z^n$, $L^\sharp \cong (2M)^{-1}\Z^n$, and $L^\sharp/L \cong \Z^n / (2M \Z^n)$.
 In particular, if $\underline{L}$ is an even integral lattice, then $J_{k,\underline{L}} \cong J_{k,M}$ (see \cite[$\S$1]{BS}).

\subsection{Discriminant modules}
For an even integral lattice  $\underline{L} = (L,\beta)$, \textit{the discriminant module} $D_{\underline{L}}$  
is defined as the pair 
\begin{equation*}
  D_{\underline{L}} :=
  \left( L^\sharp / L, x + L \mapsto \beta(x) + \Z \right)
\end{equation*}
(see \cite[$\S$ 2]{BS}).
Note that $L^\sharp / L$ is a finite abelian group and the map $x + L \mapsto \beta(x) + \Z \in \Qq/\Z$
is the associated quadratic form.

Two discriminant modules $D_{\underline{L_i}} = (L_i^\sharp/ L_i, x + L_i \rightarrow \beta_i(x) + \Z)$ $(i=1,2)$ are said to be isomorphic,
if there exists a group isomorphism  $j : L_1^\sharp/L_1 \rightarrow  L_2^\sharp/L_2$ that preserves the quadratic forms, i.e., 
$\beta_2(j(x)) + \Z =  \beta_1(x) + \Z$ for all $x + L_1 \in L_1^\sharp/L_1$.
 
\subsection{Theta functions}

For $x \in L^\sharp$, we define the theta function associated with $\underline{L}$ and $x \in L^\sharp$ by
\begin{equation*}
   \theta_{\underline{L},x}(\tau,z) := \sum_{y \in x +  L} e(\tau \beta(y) + \beta(y,z))
   \quad ((\tau,z) \in \HH \times \C^n).
\end{equation*}

The following transformation formulas are well known.
\begin{lemma}\label{lem:theta_transformation}
 For $r_i \in L^\sharp$, the following identities hold:
 \begin{equation}
   \notag
   \theta_{\underline{L},r_i}(\tau + 1,z) = e(\beta(r_i)) \theta_{\underline{L},r_i}(\tau,z)
\end{equation}   
 and
\begin{align}
   \label{id:theta_inverse}
  &
   \hspace{-1cm}
   \theta_{\underline{L},r_i}(-\tau^{-1}, \tau^{-1}z) \\
   \notag
   &=
   \left( \frac{\tau}{i} \right)^{n/2} e(\tau^{-1} \beta(z))
   \frac{1}{\sqrt{|L^\sharp / L |}}  \sum_{r_j \in L^\sharp / L} e(-\beta(r_i, r_j)) \theta_{\underline{L},r_j}(\tau,z)
 \end{align}
  $($See \cite[p. 100]{ara}, \cite[Theorem 2.3.3]{ali}, \cite[Cor. 3.34]{boy}, \cite[p. 16]{BS}, etc.$)$.
\end{lemma}

The following Lemma follows from \cite[Thm. 2.2]{BS}.
\begin{lemma}\label{lem:theta_decom}
Let  $\phi \in J_{k,\underline{L}}$ be a Jacobi form. Then $\phi$ admits the following decomposition:
\begin{equation*}
   \phi(\tau,z) = \sum_{x \in L^\sharp \slash L} h_x(\tau) \theta_{\underline{L},x}(\tau,z)
  \mbox{ with } h_x(\tau)   =  \sum_{n' \in \Z} c(n',x) e^{2 \pi i (n' - \beta(x)) \tau }.   
\end{equation*}
Since Definition~\ref{df:jacobi_lattice} (ii) and (iii) imply $c(n',x) = c(n'+ \beta(x,y) + \beta(y), x + y)$ for any $y \in L$, 
the coefficient $c(n',x)$ depends only on $x$ modulo $L$
and $n' - \beta(x)$. 
In particular, the above decomposition of $\phi$ is well-defined.
\end{lemma}

\begin{theorem}[{\cite[Thm. 2.5]{BS}}]\label{thm:bs_isom}
 Let $\underline{L}_1$ and $\underline{L}_2$ be two positive-definite even lattices of rank $n_1$ and $n_2$, respectively.
 Assume that $j : D_{\underline{L_1}} 
  \xrightarrow{\cong} 
 D_{\underline{L_2}}$ is an isomorphism of the discriminant modules. 
Then the map
\begin{equation*}
  I_j \ : \ J_{k+\lfloor \frac{n_1}{2}\rfloor, \underline{L_1}} \rightarrow J_{k+\lfloor\frac{n_2}{2}\rfloor, \underline{L_2}}
\end{equation*}
defined by
\begin{equation*}
     \sum_{x \in L_1^\sharp \slash L_1} h_x(\tau) \theta_{\underline{L_1},x}(\tau,z)
     \mapsto
     \sum_{x \in L_1^\sharp \slash L_1} h_x(\tau) \theta_{\underline{L_2},j(x)}(\tau,z')
\end{equation*}
is an isomorphism of $\C$-vector spaces.
\end{theorem}
As shown in~\cite[Thm.~4.2.4]{ali}, 
this isomorphism is also compatible with the action of Hecke operators (see Theorem~\ref{thm:ali_hecke} below).

\subsection{Modular forms of half-integral weight
}\label{ss:half_int_theta}
We denote by $\mathfrak{G}$ the group consisting of all pairs $(A, \omega(\tau))$,
where $A = \left(\begin{smallmatrix} a & b \\ c & d \end{smallmatrix} \right)$ is an element
of the connected component $\mathrm{GL}^+(2,\R)$ of $\mathrm{GL}(2,\R)$,
and $\omega$ is a holomorphic function on the upper half-plane $\HH$ satisfying 
\begin{equation*}
  |\omega(\tau)|^2 = (\det A)^{-\frac12} |c\tau + d|.
\end{equation*}
The group operation on $\mathfrak{G}$ is defined by
$(A_1,\omega_1(\tau)) \cdot (A_2, \omega_2(\tau) ) := (A_1 A_2, \omega_1(A_2 \tau) \omega_2(\tau))$.

We define the usual theta function by 
\begin{equation*}
  \theta(\tau) := \sum_{m \in \Z} e(m^2 \tau).
\end{equation*}
There exists an injective homomorphism $\Gamma_0(4) \rightarrow \mathfrak{G}$ given by
$  A  \mapsto (A, j(A, \tau)) $, where
\begin{equation*}
  j(A,\tau) := \frac{\theta(A \tau)}{\theta(\tau)} 
  \quad \text{for } A = \left(\begin{smallmatrix} a & b \\ c & d \end{smallmatrix} \right) \in \Gamma_0(4).
\end{equation*}
It is well known that this factor is explicitly given by 
\begin{equation*}
  j(A,\tau) = \biggl(\frac{c}{d}\biggr) \left(\frac{-4}{d}\right)^{-1/2} (c\tau + d)^{1/2},
\end{equation*}
where $\left(\frac{c}{d}\right)$ denotes the Kronecker symbol
(cf. \cite[p.~194]{Miyake}).

Let $N$ be a positive integer such that $4|N$.
For $k \in  \Z$, 
we denote by $M_{k+\frac12}(N)$ the vector space of holomorphic functions $f$ on $\HH$,
satisfying the following conditions:
\begin{enumerate}
\item[$\cdot$]
  For any $A \in \Gamma_0(N)$, 
  \begin{equation*}
   (f|_{k+\frac12}A)(\tau) := j(A,\tau)^{-(2k+1)} f(A\tau)  =  f(\tau)
\end{equation*}
 \item[$\cdot$]
   $f^{2}$ is an elliptic modular form of weight $2k+1$
  and level $N$ with character $\left( \frac{-4}{d} \right)$.
\end{enumerate}
An element $f \in M_{k+\frac12}(N)$ is called a modular form of weight $k+\frac12$ and level $N$.
Such a function $f$ is a cusp form if $f^2$ is a cusp form.
We denote the subspace of cusp forms by $S_{k+\frac12}(N)$.

We define the Kohnen plus-space of level $8$ as follows:
\begin{equation*}
  M_{k+\frac12}^{+}(8)
  :=
  \left\{ 
    g = \sum_n c_g(n) q^n \in M_{k+\frac12}(8)
    \ : \
    c_g(n) = 0 \mbox{ unless }  n \equiv 0, (-1)^k  \!\! \mod 4
  \right\}.
\end{equation*}

For $k \in \Z$ and $r \in \{1,3,5,7\}$ satisfying $(-1)^k \equiv -r \mod 4$, we define
the subspace $M_{k+\frac12}^{+, -r}(8)$ by
\begin{equation*}
  M_{k+\frac12}^{+, -r}(8)
  :=
  \left\{ 
    g = \sum_n c_g(n) q^n \in M_{k+\frac12}^{+}(8)
    \ : \
    c_g(n) = 0 \mbox{ unless } n \equiv 0,4, -r  \!\! \mod 8
  \right\}.
\end{equation*}
For example, it holds that $\theta \in M_{1/2}^{+,-7}(8)$.

Note that $  M_{k+\frac12}^{+, -r}(8) \subset   M_{k+\frac12}^{+}(8)$ by definition.
It follows from \cite[Prop. 4]{UY} that
the plus-space admits the following direct sum decomposition: 
\begin{equation*}
  M_{k+\frac12}^{+}(8) =
   \begin{cases}
   M_{k+\frac12}^{+,-1}(8) \oplus M_{k+\frac12}^{+,-5}(8), & \text{ if $k$ is odd}, \\
   M_{k+\frac12}^{+,-3}(8) \oplus M_{k+\frac12}^{+,-7}(8), & \text{ if $k$ is even}.    
   \end{cases}
\end{equation*}

Similarly, we define the spaces of cusp forms by $S_{k+\frac12}^{+}(8) := M_{k+\frac12}^{+}(8) \cap S_{k+\frac12}(8)$
and $S_{k+\frac12}^{+,-r}(8) := M_{k+\frac12}^{+,-r}(8) \cap S_{k+\frac12}(8)$.
Then we have
\begin{equation}\label{id:S_+_r}
  S_{k+\frac12}^{+}(8) = 
   \begin{cases}
   S_{k+\frac12}^{+,-1}(8) \oplus S_{k+\frac12}^{+,-5}(8), & \text{ if $k$ is odd}, \\
   S_{k+\frac12}^{+,-3}(8) \oplus S_{k+\frac12}^{+,-7}(8), & \text{ if $k$ is even}.    
   \end{cases}
\end{equation}

\subsection{Modular forms of half-integral weight of $\eta$-type}\label{ss:eta_type}
We set
\begin{equation*}
 \mathrm{Mp}(2,\Z) := \left\{ (A, \omega(\tau))  \ : \ A = \begin{pmatrix} a & b \\ c & d \end{pmatrix} \in \mbox{SL}(2,\Z),
 \omega(\tau)^2 =  c\tau + d \right\},
\end{equation*}
which we regard as a subgroup of $\mathfrak{G}$.
Following \cite{BS}, we note that $\mbox{Mp}(2,\Z)$ is a two-fold central extension of $\mbox{SL}(2,\Z)$:
\begin{equation*}
  1 \rightarrow \{ \pm 1\} \rightarrow \mathrm{Mp}(2,\Z) \rightarrow \mathrm{SL}(2,\Z) \rightarrow 1.
\end{equation*}
The group $\mbox{Mp}(2,\Z)$ is generated by the two elements
$\tilde{T} := \left(\left(\begin{smallmatrix} 1 & 1 \\ 0 & 1 \end{smallmatrix} \right),1 \right)$
and $\tilde{S} := \left(\left(\begin{smallmatrix} 0 & -1 \\ 1 & 0 \end{smallmatrix} \right), \sqrt{\tau} \right)$, 
which satisfy
 $\tilde{S}^2 = (\tilde{S} \tilde{T})^3 = \left(\left(\begin{smallmatrix} -1 & 0 \\ 0 & -1 \end{smallmatrix} \right), \sqrt{-1} \right)$.
Moreover, $\tilde{S}$ has order $8$ in $\mathrm{Mp}(2,\Z)$.

It is known that the group of linear characters of $\mbox{Mp}(2,\Z)$ is the cyclic group of order $24$ generated by $\varepsilon$,
where
\begin{equation*}
  \varepsilon((A,\omega))
  :=
  \frac{\eta(A\tau)}{\eta(\tau) \omega(\tau)} \quad 
  \text{for }
  (A,\omega) \in \mathrm{Mp}(2,\Z).
\end{equation*}
This definition of $\varepsilon$ is independ of the choice of $\tau \in \HH$.
In particular, 
we have $\varepsilon(\tilde{T}) = e(1/24)$ and $\varepsilon(\tilde{S}) = \varepsilon(\tilde{T})^{-3} = e(-1/8)$.
For further details on the group of linear characters of $\mathrm{Mp}(2,\Z)$, we refer the reader to \cite[Prop. 1.1]{BS}.

For $k \in \frac12 \Z$ and $s \in \Z$ ($0 \leq s \leq 23$),
we denote by $M_{k+\frac{s}{2}}(1,\varepsilon^s)$ the vector space of holomorphic functions $f$ on $\HH$ satisfying
the following conditions:
\begin{enumerate}
\item[$\cdot$]
$f(A\tau) = \varepsilon((A,\omega))^s \omega(\tau)^{2k+s} f(\tau)$
for any $(A,\omega) \in \mbox{Mp}(2,\Z)$.
\item[$\cdot$]
$f^{24}$ is an elliptic modular form of weight $24k+12s$ and level $1$.
\end{enumerate}
Since $\xi = \left( \left(\begin{smallmatrix} 1 & 0 \\ 0 & 1 \end{smallmatrix}\right),-1 \right)\in \mathrm{Mp}(2,\Z)$
and $\varepsilon(\xi) = -1$,
it follows that 
if a non-zero function $f$ exists, then $2k+2s  \equiv 0 \mod 2$.
Consequently, $\dim M_{k+\frac{s}{2}}(1,\varepsilon^s) = 0$ unless $k \in \Z$.

\begin{lemma}\label{lem:eta_M}
For $0 \leq s \leq 23$ and $k \in \Z$, we have
\begin{equation*}
   M_{k+\frac{s}{2}}(1,\varepsilon^s) = \eta^s M_k(1).
\end{equation*}
\end{lemma}
\begin{proof}
 Let $\Delta := \eta^{24}$ be the Ramanujan $\Delta$-function.
 From the transformation formula and the $q$-expansion of $f \in M_{k+s/2}(1,\varepsilon^s)$, 
 it follows that $\eta^{24-s} f$ is a modular form of weight $k+12$ and level $1$ which vanishes at the cusp.
 Thus, $\eta^{24-s} M_{k+s/2}(1,\varepsilon^s) \subset S_{k+12}(1) =  \Delta M_k(1)$,
 which implies $M_{k+s/2}(1,\varepsilon^s) \subset \eta^{s} M_k(1)$.
 The opposite inclusion $\eta^{s} M_k(1) \subset M_{k+s/2}(1,\varepsilon^s)$ is obvious.
\end{proof}

Following~\cite[p.~6]{YY}, we refer to $f \in M_{k+\frac{s}{2}}(1,\varepsilon^s)$ 
as a modular form of weight $k+\frac{s}{2}$ of $\eta$-type.
In this paper, we focus exclusively on the cases $s \in \{ 3,9,15, 21\}$.

\subsection{Hecke operators}\label{ss:Hecke_op}

Assume that $k \in \Z$, and let $f \in M_{k+\frac12}(4)$. 
Let $f(\tau) = \sum_{n=0}^\infty c(n) q^n$ be the Fourier expansion of $f$.
For any odd prime $p$, the action of the Hecke operator $T(p^2)$ is defined by
\begin{equation*}
  (f|_{k+\frac12}T(p^2))(\tau) := 
  \sum_{n=1}^\infty 
  \left\{ c(p^2 n) + \left( \frac{(-1)^k n }{p}\right) p^{k-1} c(n) + p^{2k-1} c(n/p^2) \right\}
   q^n,
\end{equation*}
where we set $c(n/p^2) = 0$ if $n/p^2 $ is not an integer.

Let $f \in \eta^r M_{k+\frac12-\frac{r}{2}}(1)$ for $r \in \{3, 9, 15, 21\}$, 
which is a modular form of weight $k+1/2$ of $\eta$-type. 
The Fourier expansion of $f$ is given by $f(\tau) = \sum_{n=1}^\infty c(n) q^{n/8}$. 
Note that $c(n) = 0$ unless $n \equiv r/3 \mod 8$.
For any odd prime $p$, we introduce the \textit{twisted} Hecke operator $\tilde{T}(p^2)$ defined by
\begin{equation*}
  (f|_{k+\frac12}\tilde{T}(p^2))(\tau) := 
    \left( \frac{-4}{p} \right)
  \sum_{n=1}^\infty 
  \left\{ c(p^2 n) + \left( \frac{(-1)^k n }{p}\right) p^{k-1} c(n) + p^{2k-1} c(n/p^2) \right\}
   q^{n/8},
\end{equation*}
where we set $c(n/p^2) = 0$ if $n/p^2 $ is not an integer.
It should be noted that our definition of $\tilde{T}(p^2)$ differs from the definition of $T_{p^2}$ in \cite[Prop. 11]{YY} 
by the factor $\left( \frac{-4}{p} \right)$.
To establish the isomorphism between $\eta^r M_{k+\frac{1-r}{2}}(1)$ and a certain space of Jacobi forms
of odd weight as Hecke modules
(see Prop.~\ref{prop:J_M_odd}), 
we employ the twisted operator $\tilde{T}(p^2)$ for $\eta^r M_{k+\frac{1-r}{2}}(1)$
throughout this paper.

We introduce the Hecke operators $T^J(p)$ acting on the space of Jacobi forms of index $\underline{L}$.
The following lemma is a direct consequence of \cite[$\S$4]{Murase}.
\begin{lemma}\label{lem:hecke_fourier}
Let $\phi \in J_{k,\underline{L}}$ be a Jacobi form with the Fourier expansion
 \begin{equation*}
   \phi(\tau,z) = \sum_{n' \in \Z} \sum_{r' \in L^\sharp} c(n',r') e(n'\tau + \beta(r',z)) .
 \end{equation*}
 Let $m$ be the rank of $L$.
 For any prime $p$ such that $p {\not |} \,  (2 |L^\sharp/L|)$,
 we define the action of the Hecke operator $T^J(p)$ by 
\begin{equation*}
  (\phi |_{k,\underline{L}} T^J(p))(\tau,z)
  := \sum_{n' \in \Z} \sum_{r' \in L^\sharp} c^{*}(n',r') e(n'\tau + \beta(r',z)),
\end{equation*}
where the Fourier coefficients $c^{*}(n',r')$ are given by 
\begin{align*}
  c^{*}(n',r') 
 :=
  &\quad c(p^{2}n', p r') \\
  &
  + p^{k-m-2}
  c(n',r')
  \begin{cases}
    \left( \left( \frac{-1}{p} \right) p \right)^{\frac{m}{2}} \left( \frac{[L^\sharp/L]}{p} \right) 
    \delta(n',r'),
    &\text{if $m \equiv 0 \!\! \mod 2$},\\
    \left( \left( \frac{-1}{p} \right) p \right)^{\frac{m+1}{2}} \left( \frac{2 [L^\sharp/L]}{p} \right)
    \left( \frac{n'- \beta(r')}{p} \right),
    &\text{if $m \equiv 1 \!\! \mod 2$},
  \end{cases} \\
  &
  + p^{2k-m-2}
  \sum_{\lambda \in L/ p L }c\!\left(\frac{1}{p^{2}}(n' - \beta(r',\lambda) + \beta(\lambda)), \frac{1}{p}(r' -  \lambda)\right).
\end{align*}
Here, $\lambda$ runs through a set of representatives of $ L/ p L$. 
Note that there exists only one $\lambda \in L/ p L$ such that $\frac{1}{p}(r'-\lambda) \in L^\sharp$.
We set $c(n',r') = 0$ if $(n',r')\, {\not \in}\, \Z \times L^\sharp$. 
Thus, the above summation $\displaystyle{\sum_{\lambda \in L/ p L }}$ consists of at most one term.
Furthermore, we define 
 \begin{equation*}
  \delta(n',r') :=
  \begin{cases}
      p-1, & \text{if $n' - \beta(r')  \equiv 0 \mod p $},\\
        -1, &\text{otherwise}.
  \end{cases}
  \end{equation*}
  Under this definition, we have $\phi|T^J(p) \in J_{k,\underline{L}}$.
  Moreover, if $\phi$ is a Jacobi cusp form, then $\phi|T^J(p)$ is also a Jacobi cusp form.
\end{lemma}
In this paper, we call $\phi$ a Hecke eigenform if it is an eigenfunction of $T^J(p)$
for all primes $p$ such that $p \, {\not | }\, (2|L^\sharp/L|)$.

Note that when $m$ is odd, then the definition of $T^J(p)$ coincides with 
the operator $T(p)$ defined in~\cite[Thm. 2.6.1]{ali}.

\begin{theorem}[{\cite[Thm.~4.2.4]{ali}}]\label{thm:ali_hecke}
 Assume that $n_1 \equiv n_2 \equiv 1 \mod 2$.
 Then the isomorphism 
 \begin{equation*}
  I_j \ : \ J_{k+\lfloor \frac{n_1}{2}\rfloor, \underline{L_1}} \xrightarrow{\cong} J_{k+\lfloor\frac{n_2}{2}\rfloor, \underline{L_2}}
  \end{equation*}
  defined in Theorem~\ref{thm:bs_isom} is compatible with the action of Hecke operators $T^J(p)$:
  \begin{equation*}
    I_j(\phi|_{k+\lfloor \frac{n_1}{2}\rfloor}T^J(p)) =
    I_j(\phi)|_{k+\lfloor \frac{n_2}{2}\rfloor} T^J(p)
  \end{equation*}
  for any $\phi \in J_{k+\lfloor \frac{n_1}{2}\rfloor, \underline{L_1}}$
  and any prime $p$ such that $p\, {\not |}\, (2|L_1^\sharp/L_1|)$.
\end{theorem}

\subsection{Jacobi forms of index $D_r$}\label{ss:JF_Dr}
Following \cite[$\S$ 3.3]{Mocanu}, we define the root lattice $D_r$ of rank $r$ as the set
\begin{equation*}
  D_r := \left\{ (x_1, ..., x_r) \in \Z^r \ : \ x_1 + \cdots + x_r \in 2 \Z \right\}
\end{equation*}
equipped with the Euclidean bilinear form:
\begin{equation*}
  \beta((x_1,...,x_r),(y_1,...,y_r)) :=
  x_1 y_1 + \cdots + x_r y_r.
\end{equation*}
The dual lattice of $D_r$ is given by
\begin{equation*}
  D_r^\sharp = \left\{ (x_1,...,x_r) \in \Z^r \cup {\left(\frac12 + \Z\right)}^r \ \right\}.
\end{equation*}
Let $\{e_i\}_i$ be the standard basis of $\Z^r$.
A set of representatives for the discriminant module $D_r^{\sharp}/D_r$ can be chosen as follows:
\begin{equation*}
  D_r^{\sharp}/ D_r =
  \begin{cases}
  \left\{ 
  0, e_r, \frac{e_1 + \cdots + e_r}{2},
  \frac{e_1 + \cdots + e_{r-1} - e_r}{2}
  \right\}, & \text{if $r > 1$}, \\
  \left\{ 
  0, e_1, \frac12 e_1, -\frac12 e_1
  \right\}, & \text{if $r = 1$}.
  \end{cases}
\end{equation*}

Assume that $r$ is an odd integer.
Then we have $D_r^{\sharp}/ D_r \cong \Z / 4 \Z$ as $\Z$-module.
The discriminant module $D_{D_r}$ of $D_r$ is given by
\begin{equation*}
  D_{D_r} \cong
  \left(  \Z/4\Z, x + 4 \Z \mapsto \frac{rx^2}{8} + \Z\right).
\end{equation*}
This implies that the discriminant module $D_{D_r}$ is determined solely by $r \!\! \mod 8$.
By virtue of Theorems~\ref{thm:bs_isom} and~\ref{thm:ali_hecke},
we obtain the following corollary:
\begin{cor}[{\cite[Thm. 2.5]{BS}}, {\cite[Thm. 4.2.4]{ali}}]\label{cor:J_r_general}
If $r_1$ and $r_2$ are odd integers such that
$r_1 \equiv r_2 \mod 8$, then
\begin{equation*}
  J_{k+\frac{r_1+1}{2},D_{r_1}} \cong J_{k+\frac{r_2+1}{2}, D_{r_2}}
\end{equation*}
as Hecke modules.
\end{cor}
Consequently, it is sufficient to consider 
the cases $r \in \{1,3,5, 7\}$.

Note that the space of Jacobi forms of lattice index $D_r$ is isomorphic to the space of 
Jacobi forms of matrix index $M_r$:
\begin{equation}\label{id:D_M}
  J_{k,D_r}
   \cong 
  J_{k, M_r},
\end{equation}
where the matrices $M_r$ are given by
$M_1 = 2$,
$M_3 = \begin{pmatrix} 1 & u & u \\ u & 1 & u \\ u & u & 1 \end{pmatrix}$,
$M_5 = \left( \begin{smallmatrix} 1 & u & 0 & 0 & u \\ 
   u & 1 & u & 0 & 0  \\ 0 & u & 1 &  u & 0 \\
   0 & 0 & u & 1 & u \\ u & 0 & 0 & u & 1  \end{smallmatrix} \right)$,
and
$M_7 = \left(\begin{smallmatrix} 1 & u & 0 & 0 & 0 & 0 & u \\ 
   u & 1 & u & 0 & 0 & 0 & 0   \\ 0 & u & 1 &  u & 0 & 0 & 0\\
   0 & 0 & u & 1 & u & 0 & 0 \\ 0 & 0 & 0 & u & 1 & u & 0 \\
   0 & 0 & 0 & 0 & u & 1 & u \\ u & 0 & 0 & 0 & 0 & u & 1 \end{smallmatrix} \right)$ 
with $u = 1/2$. Here $2M_r$ denotes the Gram matrix of $D_r$
with respect to the basis $\{2 e_1\}$ for $r=1$, and $\{e_1+e_2, e_2+e_3,...,e_{r-1}+e_r, e_r + e_1 \}$
for $r= \{ 3, 5, 7\}$.
One can easily verity that $\det (2 M_r) = 4$ for all $r \in \{ 1,3,5, 7\}$.

\section{Ikeda lifting}\label{sec:Ikeda_lift}

We denote by $M_k(\text{Sp}_n(\Z))$ the space of Siegel modular forms
of weight $k$ and degree $n$. 
Here $\text{Sp}_n(\Z)$ denotes the Siegel modular group of degree $n$ (consisting of $2n \times 2n$ matrices).
The subspace consisting of all Siegel cusp forms in $M_k(\text{Sp}_n(\Z))$ 
is denote by $S_k(\text{Sp}_n(\Z))$.

\begin{theorem}[Duke--Imamoglu~\cite{BK}, Ikeda~\cite{Ik}]
Let $k$ and $n$ be positive integers such that $k+n$ is even.
Then there exists an injective linear map
\begin{equation*}
  I_n \ : \   S_{2k}(1)  \to  S_{k+n}(\mathrm{Sp}_{2n}(\Z)),
\end{equation*}
which maps Hecke eigenforms to Hecke eigenforms.
In the case $n=1$, the map $I_1$ coincides with the Saito--Kurokawa lift.
For further details, we refer the reader refer to \cite{Ik}.
\end{theorem}

We denote by $L_n^*$ the set of all half-integral symmetric matrices of size $n$.
Let $\HH_{n}$ be the Siegel upper half space of size $n$.
A matrix $M \in L_n^*$ is called \textit{maximal},
if it satisfies the following condition: 
if $A^{-1} M {^t A^{-1}} \in L_n^*$ for some $A \in \text{GL}_n(\R) \cap M_n(\Z)$,
then $A \in \text{GL}_n(\Z)$.
Here $M_n(\Z)$ denotes the set of all $n \times n$ matrices with entries in $\Z$.

\begin{lemma}\label{lem:DDI_FJ}
 Let $I_{n}(f) \in S_{k+n}(\mbox{\rm{Sp}}_{2n}(\Z))$ be the Ikeda lift of $f \in S_{2k}(1)$.
 We consider the Fourier--Jacobi expansion of $I_n(f)$:
 \begin{equation*}
   I_n(f) \left( \begin{pmatrix} \tau & z \\ {^t z} & \omega \end{pmatrix} \right)
   = \sum_{M \in L_{2n-1}^*} \phi_{M}(\tau,z) e(M\omega),
 \end{equation*}
 where  $\tau \in \HH$, $z \in \C^{2n-1}$, and $\omega \in \HH_{2n-1}$.
 Then each $\phi_M$ is a Jacobi cusp form of weight $k+n$ and index $M$.
 Moreover, if $f$ is a Hecke eigenform and $M$ is maximal, 
 then $\phi_M$ is also a Hecke eigenform.
\end{lemma}
\begin{proof}
 It is well known that $\phi_M$ is a Jacobi cusp form of weight $k+n$ of index $M$ for any $M \in L_{2n-1}^*$
 (see \cite[Introduction]{Zi}).
 The fact that $\phi_M$ is a Hecke eigenform when $M$ is maximal follows from the following observations:
 \begin{itemize}
 \item[(i)] Jacobi--Eisenstein series is a Hecke eigenform (see~\cite[Thm 3.3.18]{ali}). 
 \item[(ii)] The $M$-th Fourier--Jacobi coefficient of Siegel--Eisenstein series
 is a Jacobi--Eisenstein series if $M$ is maximal (see~\cite[Satz 7]{Bo}). 
 \item[(iii)]  The Ikeda lift $I_n(f)$ inherits certain relations among the Fourier coefficients of Siegel--Eisenstein series.
 Since the $M$-th Fourier--Jacobi coefficient of Siegel--Eisenstein series is a Hecke eigenform,
 we conclude that $\phi_M$ is also a Hecke eigenform.
 \end{itemize}
 We omit the further details (see, for example, \cite{FJlift}).
\end{proof}
 
 \begin{lemma}\label{lem:cusp_to_jacobi_cusp}
  Assume $2n = r+1$.
  If $2M_r$ is the Gram matrix of $D_r$ for $r \in \{1,3,5,7\}$, 
  and if $f$ is a Hecke eigenform,
  then the Fourier--Jacobi coefficient $\phi_{M_r}$ in Lemma~\ref{lem:DDI_FJ} is not identically zero.
  In particular, $\phi_{M_r}$ is a non-zero Hecke eigenform in $J^{\text{cusp}}_{k+\frac{r+1}{2},M_r}$.
 \end{lemma}
 \begin{proof}
 Since there exists $u_1, u_2 \in D_r^{\sharp}$ such that $\beta(u_1) = 0$ and $\beta(u_2) = \frac12$,
 the set $\{ 8(n' - \beta(u)) \, : \, n' \in \Z, u \in D_r^{\sharp}\}$ contains all positive integers divisible by $4$.
 It is shown in \cite[Page 260]{Ko} that there exists a fundamental discriminant $D \equiv 0 \mod 4$ such that $c_h(|D|) \neq 0$,
 where $h \in S_{k+n+1/2}^+(4)$ is the modular form of half-integral weight corresponding to $f$ via the Shimura correspondence, 
 and $c_h(|D|)$ denotes its $|D|$-th Fourier coefficient.
 By the explicit construction of the Fourier coefficients of the Ikeda lift (cf.~\cite[p.~642]{Ik}), 
 this non-vanishing property implies that $\phi_{M_r}$ is not identically zero.
 The result then follows from Lemma~\ref{lem:DDI_FJ}.
 \end{proof}
 
 Therefore, by virtue of the isomorphism (\ref{id:D_M}), we obtain the following proposition.
 \begin{prop}\label{prop:J_old}
 Assume that $k \equiv \frac{r+1}{2} \mod 2$
 and $k \geq 2$.
 Then there exists an injective linear map $\mathcal{I}$ defined by the composition of the Ikeda lift and the Fourier--Jacobi expansion:
 \begin{equation*}
  \mathcal{I} \ : \
  M_{2k}^{\epsilon_1}(1) \to J_{k+\frac{r+1}{2},D_r}. 
 \end{equation*}
 The map $\mathcal{I}$ sends
 Eisenstein series to Jacobi--Eisenstein series
 and $S_{2k}^{\epsilon_1}(1)$ to $J_{k+\frac{r+1}{2},D_r}^{\text{cusp}}$.
 Moreover, it commutes with the action of the Hecke operators.
 \end{prop}
 \begin{proof}
  Recall that $\epsilon_1 = - \left(\frac{-4}{r }\right)$.
  The condition $k \equiv \frac{r+1}{2} \!\! \mod 2$ is equivalent to
  $(-1)^k = (-1)^{\frac{r+1}{2}} = \epsilon_1$.
  Thus, under the assumptions of this proposition, we have $M_{2k}(1) = M_{2k}^{\epsilon_1}(1)$.
  
  An Eisenstein series in $M_{2k}^{\epsilon_1}(1)$ is mapped to a Siegel--Eisenstein series via the Ikeda lift.
  Since the $M$-th Fourier--Jacobi coefficient of a Siegel--Eisenstein series is a Jacobi--Eisenstein series 
  whenever $M$ is maximal, 
  the map $\mathcal{I}$ sends Eisenstein series in $M_{2k}^{\epsilon_1}(1)$ to Jacobi--Eisenstein series
  in $J_{k+\frac{r+1}{2},D_r}$.
  
  The reminder of the proof follows from Lemma~\ref{lem:cusp_to_jacobi_cusp}.
 \end{proof}
 
We are now ready to prove Theorem~\ref{thm:J_decom}.
\begin{proof}[Proof of Theorem~\ref{thm:J_decom}]
By definition, 
the subspace $J_{k+\frac{r+1}{2},D_r}^{\text{old}}$ is the image of the map $\mathcal{I}$ given in Proposition~\ref{prop:J_old}, 
and $J_{k+\frac{r+1}{2},D_r}^{\text{new}}$ is defined as the orthogonal complement of $J_{k+\frac{r+1}{2},D_r}^{\text{old}}$ 
in $J_{k+\frac{r+1}{2},D_r}$ with respect to the Petersson scalar product.
Consequently, we have the direct sum decomposition 
$J_{k+\frac{r+1}{2},D_r} = J_{k+\frac{r+1}{2},D_r}^{\text{new}} \oplus J_{k+\frac{r+1}{2},D_r}^{\text{old}}$.
The isomorphism $J_{k+\frac{r+1}{2},D_r}^{\text{old}} \cong M_{2k}^{\epsilon_1}(1)$ as Hecke modules 
then follows directly from Proposition~\ref{prop:J_old}.
\end{proof}

\section{Theta decomposition of Jacobi forms of index $D_r$}\label{sec:decom}
 The structure of the lattice $\underline{D_r} = (D_r, \beta)$ for $r \in \{1,3,5, 7\}$
 has been explained in~$\S$\ref{ss:JF_Dr}.
We define a set of representatives $\{v_0,v_2,v_1,v_3\}$ for the discriminant module $D_r^\sharp/D_r$ as follows:
\begin{itemize}
\item[$\cdot$] For $r = 1$, we set $v_j := \frac{j}{2} e_1 $ for $j \in \{0,1,2,3\}$. 
\item[$\cdot$] For $r \in \{ 3,5,7\}$, we set 
\begin{equation*}
 v_0 := 0, \quad v_2 := e_r, \quad v_1 := \frac{e_1 + \cdots + e_r}{2}, \quad v_3 := \frac{e_1 + \cdots + e_{r-1} - e_r}{2} . 
\end{equation*}
\end{itemize}
The values of the quadratic form on these representatives are given by
\begin{equation*}
  \beta(v_0) = 0, \quad \beta(v_2) = \frac12, \quad \beta(v_1) = \beta(v_3) = \frac{r}{8}.
\end{equation*}
Note that the map $v_j + D_r \mapsto j + 4 \Z$ induces an isomorphism $D_r^\sharp / D_r \cong \Z/(4\Z)$
as $\Z$-modules.
Under this identification, the associated quadratic form on $D_r^\sharp/ D_r$ is given by
\begin{equation*}
  v_j + D_r  \mapsto 
  \beta(v_j) + \Z =  \frac{j^2 r}{8} + \Z,
\end{equation*}
and the associated bilinear form satisfies $\beta(v_i, v_j) + \Z = \frac{i j r}{4} + \Z$.

For simplicity, we write
\begin{equation*}
 \theta_{r,j} := \theta_{D_r, v_j}.
\end{equation*}
By virtue of Lemma~\ref{lem:theta_transformation}, the following transformation formulas hold:
 \begin{equation}\label{id:theta_1}
   \theta_{r,j}(\tau + 1,z) = e(j^2 r/8) \theta_{r,j}(\tau,z)
\end{equation} 
 and
\begin{equation}\label{id:theta_2}
   \theta_{r,j}(-\tau^{-1}, \tau^{-1}z) 
   =
   \frac12 
   \tau^{r/2} e(-r/8)
   e(\tau^{-1} \beta(z))
    \sum_{i = 0}^3 e(-i j r / 4) \theta_{r,i}(\tau,z).
 \end{equation}
Any Jacobi form $\phi \in J_{k,D_r}$ admits a unique decomposition
\begin{equation*}
  \phi(\tau,z) =
  \sum_{j=0}^3 h_j(\tau) \theta_{r,j}(\tau,z),
\end{equation*}
where the functions $h_j(\tau)$ are uniquely determined by $\phi$ (see Lemma~\ref{lem:theta_decom}).

Since $2 v_0, 2 v_2, v_1 + v_3 \in D_r$,
the symmetric properties of the theta functions $\theta_{r,j}$ imply that
$  \theta_{r,j}(\tau,-z) = \theta_{r,j}(\tau,z)$ for  $j \in \{ 0,2\}$,
and
$
  \theta_{r,1}(\tau,-z) =  \theta_{r,3}(\tau,z).
$
Using the parity property of the Jacobi forms, 
$\phi(\tau,-z) = (-1)^k \phi(\tau,z)$, 
we obtain the follwing constraints on the functions $h_j(\tau)$: 
\begin{equation*}
 \begin{cases}
   h_1(\tau) = h_3(\tau), & \mbox{ if $k$ is even}, \\
  h_0(\tau) = h_2(\tau) = 0 \mbox{ and }  
  h_1(\tau) = - h_3(\tau), & \mbox{ if $k$ is odd}.
 \end{cases}
\end{equation*}
Consequently, the Jacobi form $\phi$ can be decomposed as follows: 
\begin{align*}
  &  \phi(\tau,z) \\
   & = 
  \begin{cases}
  h_0(\tau) \theta_{r,0}(\tau,z)
  +
  h_2(\tau) \theta_{r,2}(\tau,z)
  +
  h_1(\tau) (\theta_{r,1}(\tau,z) + \theta_{r,1}(\tau,-z)), & \text{ if $k$ is even}, \\
  h_1(\tau) (\theta_{r,1}(\tau,z) - \theta_{r,1}(\tau,-z)),
  & \text{ if $k$ is odd}.
  \end{cases}
\end{align*}
Note that the constant term of $\phi$ coincides with the constant term of $h_0$.

\begin{lemma}\label{lem:dim_Eisenstein}
Let the notation be as above.
A Jacobi form $\phi \in J_{k,D_r}$ is a Jacobi cusp form if and only if
the constant term of $h_0(\tau)$ is zero.

Moreover, we have $\dim J_{k,D_r} - \dim J_{k,D_r}^{\text{cusp}} \leq 1$.
In particular, if $k$ is odd, then $J_{k,D_r} = J_{k,D_r}^{\text{cusp}}$.
\end{lemma}
\begin{proof}
Let $(n',r') \in \Z \times D_r^\sharp$.
If $n' - \beta(r') = 0$, then $\beta(r') \in \Z$, which implies $r' \in D_r$.

Let $c(n',r')$ denote the $(n',r')$-th Fourier coefficient of $\phi$.
Recall that for any $(n'_i,r'_i) \in \Z \times D_r^\sharp$ ($i = 1,2$),
if $n'_1-\beta(r'_1) = n'_2-\beta(r'_2)$ and $r'_1 \equiv r'_2 \mod D_r$,
then $c(n'_1,r'_1) = c(n'_2,r'_2)$.
It follows that if $c(0,v_0) = 0$, then $c(n', r') = 0$ for all $(n',r')$ satisfying 
$n' - \beta(r') = 0$, and thus $\phi$ is a Jacobi cusp form.

Since $c(0,v_0)$ is the constant term of $h_0$, 
$\phi$ is a Jacobi cusp form if and only if the constant term of $h_0$ is zero.
This condition imposes at most one linear constraint, which implies
the inequality $\dim J_{k,D_r} - \dim J_{k,D_r}^{\text{cusp}} \leq 1$.

If $k$ is odd, we have already seen that $h_0(\tau) = 0$ for any $\phi \in J_{k,D_r}$.
Thus, the constant term of $h_0$ is necessarily zero, leading to the equality 
$J_{k,D_r} = J_{k,D_r}^{\text{cusp}}$.
\end{proof}

It is known that if $k+\frac{r+1}{2} \equiv 0 \!\! \mod 2$ and $k \geq 2$,
then there exists the Jacobi--Eisenstein series $E_{k+\frac{r+1}{2},D_r} \in J_{k+\frac{r+1}{2},D_r}$
(see \cite[Thm. 4.1]{Mocanu_JN}).
In this case, the codimension of the space of cusp forms is exactly one: 
$\dim J_{k+\frac{r+1}{2},D_r} - \dim J_{k+\frac{r+1}{2},D_r}^{\text{cusp}} = 1$.

\begin{lemma}\label{lem:cusp_cusp_new}
 For $k \geq 2$, we have $J_{k+\frac{r+1}{2},D_r}^{\text{new}} \subset J_{k+\frac{r+1}{2},D_r}^{\text{cusp}}$.
 In particular, the equality $J_{k+\frac{r+1}{2},D_r}^{\text{cusp, new}} = J_{k+\frac{r+1}{2},D_r}^{\text{new}}$ 
 holds for $k \geq 2$.
\end{lemma}
\begin{proof}
If $k+\frac{r+1}{2}$ is odd, the previous result $J_{k+\frac{r+1}{2},D_r} = J_{k+\frac{r+1}{2},D_r}^{\text{cusp}}$
(see Lemma~\ref{lem:dim_Eisenstein}) 
immediately implies the inclusion $J_{k+\frac{r+1}{2},D_r}^{\text{new}} \subset J_{k+\frac{r+1}{2},D_r}^{\text{cusp}}$.

Next, assume that $k+\frac{r+1}{2}$ is even.
By Proposition~\ref{prop:J_old}, the space $\in J_{k+\frac{r+1}{2},D_r}^{\text{old}}$ 
contains the Jacobi--Eisenstein series $E_{k+\frac{r+1}{2},D_r}$.
Let $\phi \in J_{k+\frac{r+1}{2},D_r}^{\text{new}}$. 
If the constant term of $\phi$ were non-zero, then
$\phi$ could not be orthogonal to $E_{k+\frac{r+1}{2},D_r}$ with respect to the Petersson scalar product.
This would contradict the assumption that $\phi \in J_{k+\frac{r+1}{2},D_r}^{\text{new}}$.
Therefore, the constant term of $\phi$ must be zero, which implies $\phi \in J_{k+\frac{r+1}{2},D_r}^{\text{cusp}}$.
\end{proof}

\begin{lemma}\label{lem:E_2k_new}
  For $k \geq 2$, we have $M_{2k}^{\text{new}}(2) \subset S_{2k}(2)$.
  In particular, $S_{2k}^{\text{new}}(2) = M_{2k}^{\text{new}}(2)$ holds for $k \geq 2$.
\end{lemma}
\begin{proof}
 It is well known that $\dim M_{2k}(2) - \dim S_{2k}(2) = 2$ for $k \geq 2$.
 Let $E_{2k}^{(1)}$ be the Eisenstein series of weight $2k$ with respect to $\text{SL}(2,\Z)$.
 The functions $E_{2k}^{(1)}(\tau)$ and $E_{2k}^{(1)}(2\tau)$ are linearly independent 
 and span the subspace of Eisenstein series in $M_{2k}(2)$.
 
 Consider a linear combination $f(\tau) := u E_{2k}^{(1)}(\tau) + v E_{2k}^{(1)}(2\tau)$ for $u, v \in \C$.
 If $f$ is a cusp form, its constant terms at all cusps must vanish.
 In particular, the vanishing of the constant term at the cusp $\infty$ implies $u+v = 0$, 
 or $v = -u$.
 For any odd prime $p$, the $p$-th Fourier coefficient of such an $f$ is proportional to $(1+p^{2k-1}) u$.
 Since $f$ is a cusp form, it must satisfy the Hecke bound, which implies $(1+p^{2k-1}) u = \mathcal{O}(p^k)$.
 This forces $u = v = 0$ since $2k-1 > k$ for $k \geq 2$.
 
 Thus, we have the decomposition $M_{2k}(2) = \left(\C E_{2k}^{(1)}(\tau) \oplus \C E_{2k}^{(1)}(2\tau)\right) \oplus S_{2k}(2)$.
 Since the subspace $\C E_{2k}^{(1)}(\tau) \oplus \C E_{2k}^{(1)}(2\tau)$ consists of oldforms,
 the space of newforms $M_{2k}^{\text{new}}(2)$ is orthogonal to this subspace with respect to the Petersson scalar product.
 We therefore conclude that $M_{2k}^{\text{new}}(2) \subset S_{2k}(2)$.
\end{proof}

\section{Jacobi forms of odd weight}\label{sec:JF_odd_weight}
In Sections \ref{sec:JF_odd_weight} and \ref{sec:JF_even_weight},
we treat the cases where the weight $k+\frac{r+1}{2}$ of Jacobi forms is odd and even, respectively.

In this section, we assume that $k + \frac{r+1}{2}$ is an odd integer.
The main goal of this section is to establish the following:
\begin{theorem}\label{thm:odd_weight_case}
For $r \in \{1,3,5,7\}$ and an odd integer $k + \frac{r+1}{2}$ with $k \geq 2$, we have
\begin{equation*}
  J_{k+\frac{r+1}{2},D_r}^{\text{cusp}} \cong \eta^{24-3r} M_{k-12+\frac{3r+1}{2}}(1) \ \cong \ S_{2k}^{\text{new}, \epsilon_2}(2)
\end{equation*}
as Hecke modules,
where $\epsilon_2 = -\left( \frac{-8}{r} \right)$, which equals $ -1$ if $r \in \{1, 3\}$ and $1$ if  $r \in \{  5, 7\}$.
We note that the twisted Hecke operators $\tilde{T}(p^2)$ act on $\eta^{24-3r} M_{k-12+\frac{3r+1}{2}}(1)$
$($see $\S$\ref{ss:Hecke_op}$)$.
\end{theorem}
Theorem~\ref{thm:odd_weight_case} follows from 
Proposition~\ref{prop:J_M_odd} and 
Corollary~\ref{cor:odd_half_integral} (Theorem~\ref{thm:YY})
below.

For $\phi \in J_{k+\frac{r+1}{2},D_r}^{\text{cusp}}$,
we write 
\begin{equation}\label{id:phi_odd}
\phi(\tau,z) = h_1(\tau) (\theta_{r,1}(\tau,z) - \theta_{r,1}(\tau,-z))
\end{equation}
(see $\S$\ref{sec:decom} for the notation).
\begin{prop}\label{prop:J_M_odd}
 For $r \in \{1,3,5,7\}$, let $k+\frac{r+1}{2}$ be an odd integer. 
 Then the map defined by
 \begin{equation*}
 \mathcal{J}_{r,k}^{\text{odd}} \ : \   \phi(\tau,z) \mapsto h_1(\tau)
 \end{equation*}
 induces an isomorphism of Hecke modules
 \begin{equation*}
 \mathcal{J}_{r,k}^{\text{odd}} \ : \ J_{k+\frac{r+1}{2},D_r}^{\text{cusp}} 
\cong
  \eta^{3(8-r)} M_{k-12+\frac{3r + 1}{2}}(1) .
\end{equation*}
Namely, for any odd prime $p$, we have
\begin{equation}\label{id:J_odd_hecke}
  \mathcal{J}_{r,k}^{\text{odd}}(\phi | T^J(p)) = \left(\mathcal{J}_{r,k}^{\text{odd}}(\phi)\right) | \tilde{T}(p^2) .
\end{equation}
\end{prop}
\begin{proof}
 It follows from Lemma~\ref{lem:eta_M} that $\eta^{3(8-r)} M_{k-12+\frac{3r + 1}{2}}(1) = M_{k+\frac12}(1,\varepsilon^{24-3r})$.
 
 First, we show that $\mathcal{J}_{r,k}^{\text{odd}}(\phi) (= h_1)$ belongs to $M_{k+\frac12}(1,\varepsilon^{24-3r})$.
 By using the identities (\ref{id:phi_odd}), (\ref{id:theta_1}), (\ref{id:theta_2}) and the transformation formulas for $\phi$,
 we obtain
  \begin{equation*}
   h_1(\tau+1) = e(-r/8) h_1(\tau)
  \quad \text{ and } \quad
    h_1(-\tau^{-1}) = 
    \tau^{k+1/2} e(3r/8) h_1(\tau).
  \end{equation*}
  Since 
  $\varepsilon(\tilde{T}) = e(1/24)$ and $\varepsilon(\tilde{S}) = e(-1/8)$ (see $\S$\ref{ss:eta_type} for the notation), 
  these identities can be rewritten as
  \begin{align}
    h_1(\tau+1) &= 
    \varepsilon( \tilde{T} )^{24-3r} h_1(\tau), 
    \label{id:h_1_1}
    \\
    h_1(-\tau^{-1}) &= 
    \varepsilon(\tilde{S})^{24-3r} \sqrt{\tau}^{2k+1} h_1(\tau).
    \label{id:h_1_2}
  \end{align}
  Since $\tilde{T}$ and $\tilde{S}$ generates $\mbox{Mp}(2,\Z)$, 
  the fact that $h_1 \in M_{k+\frac12}(1,\varepsilon^{24 - 3 r})$ follows from (\ref{id:h_1_1}) and (\ref{id:h_1_2}).
  It is clear that the map $\mathcal{J}_{r,k}^{\text{odd}}$ is injective.
  
  Next, we shall show that $\mathcal{J}_{r,k}^{\text{odd}}$ is surjective.
  Let $h_1 \in M_{k+\frac12}(1,\varepsilon^{24-3r})$
  and define $\phi(\tau,z) = h_1(\tau) (\theta_{r,1}(\tau,z) - \theta_{r,1}(\tau,-z))$.
  From the transformation formulas for $h_1$ and $\theta_{r,j}$ ($j=0,1,2,3$)
  (see (\ref{id:h_1_1}), (\ref{id:h_1_2}), (\ref{id:theta_1}) and (\ref{id:theta_2})), 
   it follows that $\phi$ satisfies the transformation laws 
  (i) and (ii) in Definition~\ref{df:jacobi_lattice}
  for weight $k+\frac{r+1}{2}$ and index $D_r$.
  Since $\theta_{r,1}$ satisfies the condition  (iii) in Definition~\ref{df:jacobi_lattice},
  $\phi$ also satisfies this condition.
  Therefore, we have $\phi \in J_{k+\frac{r+1}{2},D_r}$. 
  By Lemma~\ref{lem:dim_Eisenstein}, we have
  $J_{k+\frac{r+1}{2},D_r} = J_{k+\frac{r+1}{2},D_r}^{\text{cusp}}$, which implies that
  $\mathcal{J}_{r,k}^{\text{odd}}$ is surjective.
  
  Finally, we prove the identity~(\ref{id:J_odd_hecke}).
  Consider the Fourier expansion of $\phi \in J_{k+\frac{r+1}{2},D_r}^{\text{cusp}}$:
  \begin{equation*}
    \phi(\tau,z) =  \sum_{n' \in \Z} \sum_{r' \in D_r^\sharp} C(n',r') e(n'\tau + \beta(r',z)) .
  \end{equation*}
  Then, $h_1(\tau) = (\mathcal{J}_{r,k}^{\text{odd}}(\phi))(\tau)$ is given by
  \begin{equation*}
    h_1(\tau) = \sum_{ n' \in \Z } C(n',v_1) e\!\left(\left(n'-\frac{r}{8}\right)\tau\right),
  \end{equation*}
  where $v_1 = \frac{e_1 + \cdots + e_r}{2}$ is defined in $\S$\ref{sec:decom}.
  Note that $\beta(v_1) = r/8$.
  
  For $n' \in \Z$, we set $A(8n' - r) := C(n',v_1)$. Then $h_1(\tau) = \sum_{m\in \Z} A(m) q^{m/8}$.

  Let $p$ be an odd prime.
  By the definition of $\tilde{T}(p^2)$ (see~$\S$\ref{ss:Hecke_op}), we have
  \begin{equation*}
    (h_1|_{k+\frac12} \tilde{T}(p^2))(\tau) = \sum_{m \in \Z} A^*(m) q^{m/8},
  \end{equation*}
  where
  \begin{equation}\label{id:A^*_Hecke}
    A^*(m) = \left( \frac{-4}{p} \right)
    \left\{ 
    A(p^2 m) + \left( \frac{(-1)^k m }{p}\right) p^{k-1} A(m) + p^{2k-1} A(m/p^2) 
    \right\}.
  \end{equation}
  Since $A(m) = 0$ unless  $m \equiv -r \mod 8$, it follows that $A^*(m)$ also vanishes unless $m \equiv -r \mod 8$.
  
  From the definition of $T^J(p)$ (see Lemma~\ref{lem:hecke_fourier}), 
  using the relations $\frac{r+1}{2} \equiv k +1 \mod 2$ and $[D_r^\sharp : D_r ] = 4$,
  we obtain
  \begin{equation*}
  (\phi |_{k+\frac12,D_r} T^J(p))(\tau,z)
  = \sum_{n' \in \Z} \sum_{r' \in D_r^\sharp} C^{*}(n',r') e(n'\tau + \beta(r',z)),
  \end{equation*}
  where the Fourier coefficients $C^{*}(n',r')$ are given by
  \begin{align}
   C^{*}(n',r') 
  &=
    C(p^{2}n', p r') 
    + p^{k-1}
       \left( \frac{(-1)^{k+1}}{p} \right)  
      \left( \frac{8(n'- \beta(r'))}{p} \right) 
          C(n',r') 
  \label{id:C^*}          
          \\
    & \quad
    + p^{2k-1}
    \sum_{\lambda \in D_r/ p D_r }C\!\left(\frac{1}{p^{2}}(n' - \beta(r',\lambda) + \beta(\lambda)), \frac{1}{p}(r' -  \lambda)\right).
  \notag
  \end{align}  

  Since $\phi \in J_{k+\frac{r+1}{2},D_r}^{\text{cusp}}$ and the weight $k+\frac{r+1}{2}$ is odd,
  it follows that $\phi(\tau,-z) = - \phi(\tau,z)$.
  This implies the relation $C(n',-v_1) = -C(n',v_1)$.
  We now show that $C^{*}(n',v_1) = A^*(8 n' - r)$ for all $n' \in \Z$.
  
  Note that $p v_1 \equiv \left( \frac{-4}{p} \right) v_1 \mod D_r$. We also recall that $\beta(v_1) = r/8$.

  Since $C(n',r')$ depends only on the pair ($n' - \beta(r')$, $r' \! \! \mod D_r$), we have
  \begin{align*}
   C(p^2 n', p v_1) &=
   C\!\left(p^2 n' + \frac{1-p^2}{8}r, \left( \frac{-4}{p} \right) v_1\right) 
   = \left( \frac{-4}{p} \right) C\!\left(p^2 n' + \frac{1-p^2}{8}r,  v_1\right) \\
  & =
  \left( \frac{-4}{p} \right)  A(p^2(8n' - r)).
  \end{align*}
  Similarly, if $\frac{1}{p}(v_1 - \mu) \in D_r^{\sharp}$ for some $\mu \in D_r$, 
  then 
  \begin{equation*}
   \frac{1}{p}(v_1 - \mu) \equiv \frac{p^2}{p}(v_1 - \mu)  \equiv  p v_1  \equiv  
    \left( \frac{-4}{p}\right) v_1 \mod D_r.
  \end{equation*}
  Furthermore, there exists a unique $\mu \in D_r/ p D_r$ such that $\frac{1}{p} (v_1 - \mu)  \in D_r^{\sharp}$. 
  For this unique $\mu$, we obtain
  \begin{align*}
   &
   \hspace{-1cm}
   \sum_{\lambda \in D_r/ p D_r }
   C\!\left(\frac{1}{p^{2}}(n' - \beta(v_1,\lambda) + \beta(\lambda)), \frac{1}{p}(v_1 -  \lambda)\right) \\
   &= 
   C\!\left(\frac{1}{p^{2}}(n' - \beta(v_1,\mu) + \beta(\mu)), \frac{1}{p}(v_1 -  \mu)\right) 
   \ = \
   C\!\left(\frac{8 n' - r + p^2 r }{8 p^2}, \left( \frac{-4}{p} \right) v_1\right) \\
   &=
   \left( \frac{-4}{p} \right) C\!\left(\frac{8 n' - r + p^2 r }{8 p^2},  v_1\right) 
   \ = \
   \left( \frac{-4}{p} \right) A\!\left(\frac{1}{p^2}(8n' - r)\right).
  \end{align*}
  Consequently, the identity~(\ref{id:C^*}) for $r' = v_1$ yields
  \begin{align*}
    C^{*}(n',v_1) 
    &= \left( \frac{-4}{p} \right) \left\{ A(p^2 (8n' - r)) \right. \\
   &
    \left. \qquad
        + p^{k-1}
       \left( \frac{(-1)^{k}}{p} \right)  
      \left( \frac{8n'- r}{p} \right) 
          A(8n' - r)
        +  p^{2k-1} A\!\left(\frac{1}{p^2}(8n' - r)\right)
    \right\} \\
   &=
   A^*(8n' - r).
  \end{align*}  

  Since we have
  \begin{equation*}
  \mathcal{J}_{r,k}^{\text{odd}}(\phi | T^J(p))(\tau) = \sum_{n' \in \Z} C^*(n',v_1) q^{(8n'-r)/8} 
  \end{equation*}
   and 
  \begin{equation*}
   ((\mathcal{J}_{r,k}^{\text{odd}}\phi) |\tilde{T}(p^2))(\tau) = \sum_{m \in \Z} A^*(m) q^{m/8} ,
  \end{equation*}
  the identity $C^{*}(n',v_1) = A^*(8n' - r)$ implies that 
  $\mathcal{J}_{r,k}^{\text{odd}}(\phi | T^J(p)) = \left(\mathcal{J}_{r,k}^{\text{odd}}(\phi)\right)| \tilde{T}(p^2)$.
   This completes the proof.
\end{proof}

\begin{theorem}[{\cite[Thm. 2]{YY}}]\label{thm:YY}
For $r \in \{1,3,5, 7\}$ and $s \in 2\Z$, we have the following isomorphism of Hecke modules:
 \begin{equation*}
   \eta^{3r} M_{s}(1) \cong
   \begin{cases}
   S_{3r+2s-1}^{\text{new},+}(2), & \text{if } r \in \{ 1,3\}, \\
   S_{3r+2s-1}^{\text{new},-}(2), & \text{if } r \in \{5,7\}.
   \end{cases}
 \end{equation*}
 Here, the Hecke action on the left-hand side is given by the twisted Hecke operators $\tilde{T}(p^2)$ of half-integral weight 
 (see $\S$\ref{ss:Hecke_op}),
 while the right-hand side is equipped with the usual Hecke operators of integral weight.
\end{theorem}
Note that the notation $S_{3r+2s-1}^{\text{new}}\!\left(2, - \left(\frac{8}{r} \right) \right)$
in \cite[Thm.2]{YY} is equivalent to
the notation $S_{3r+2s-1}^{\text{new},\epsilon}(2)$ in this paper, where
$\epsilon =  -\left(\frac{8}{r} \right) i^{3r+2s-1} =  \left( \frac{-8}{r}\right)$. 
Specifically, $\epsilon = 1$ if $r \in \{ 1,3\}$, and $\epsilon = -1$ if $r \in \{ 5,7 \}$.
While the usual Hecke operator $T(p^2)$ are used for $ \eta^{3r} M_{s}(1)$ in  \cite[Thm.2]{YY},
we employ the twisted Hecke operator $\tilde{T}(p^2)$. 
Consequently, the twisting factor $\otimes \left( \frac{-4}{ \cdot } \right)$ present in the result
of \cite[Thm.2]{YY} is omitted here.

It should also be noted that when $r = 1$ and $s = 0$, 
the right-hand side $S_2^{\text{new}}(2)$ of the isomorphism above should be replaced by $M_2^{\text{new}}(2) = \C E_2^{(2)}$.
Here, $E_2^{(2)}$ is the modular form of weight $2$ and level $2$ defined in Theorem~\ref{thm:k_1} 
(see $\S$\ref{ss:proof_thm_k_1} for details).

By substituting $(8-r,k-12+\frac{3r+1}{2})$ for $(r,s)$ in the isomorphism of 
Theorem~\ref{thm:YY}, we obtain the following corollary:
\begin{cor}\label{cor:odd_half_integral}
For $r \in \{1,3,5, 7\}$ and an odd integer $k + \frac{r+1}{2}$ with $k \geq 2$, we have
 \begin{equation*}
      \eta^{3(8-r)} M_{
   k-12+\frac{3r+1}{2}
  }(1)
   \cong 
   \begin{cases}
   S_{2k}^{\text{new},-}(2), & \text{if } r \in \{ 1, 3\}, \\
   S_{2k}^{\text{new},+}(2), & \text{if } r \in \{ 5, 7\},
   \end{cases}
 \end{equation*}
 as Hecke modules.
\end{cor}

We can now prove Theorem~\ref{thm:odd_weight_case}.
\begin{proof}[Proof of Theorem~\ref{thm:odd_weight_case}]
Theorem~\ref{thm:odd_weight_case} follows immediately from 
Proposition~\ref{prop:J_M_odd}
and Corollary~\ref{cor:odd_half_integral} (Theorem~\ref{thm:YY}).
\end{proof}

\section{Jacobi forms of even weight}\label{sec:JF_even_weight}
In this section, we assume that $r \in \{1,3,5,7 \}$ and that the weight $k + \frac{r+1}{2}$ is an even integer.
Let the notation be as in\ $\S$\ref{sec:decom}.
The main objective of this section is to prove Theorem~\ref{thm:even_weight_case}.
\subsection{Jacobi forms of even weight of index $D_r$}\label{ss:Jacobi_forms_even_wt_main}
The vector space $S_{k+\frac12}^{+,-r}(8)$ was defined in~$\S$\ref{ss:half_int_theta}
(see $\S$\ref{ss:half_int_theta} for the notation).

\begin{theorem}\label{thm:even_weight_case}
Let $k+\frac{r+1}{2}$ be an even integer with $k \geq 0$. 
Then, we have the following isomorphisms of Hecke modules:
\begin{equation}\label{id:J_S_S_even}
  J_{k+\frac{r+1}{2},D_r}^{\text{cusp, new}} \cong S_{k+\frac12}^{\text{new},+,-r}(8)  \cong  S_{2k}^{\text{new}, \epsilon_2}(2),
\end{equation}
where $\epsilon_2 = -\left( \frac{-8}{r} \right)$, which equals $-1$  if $r \in \{1,3 \}$ and $1$ if $r \in \{ 5, 7 \}$.
Here, $S_{k+\frac12}^{\text{new},+,-r}(8)$ is the subspace of all newforms in $S_{k+\frac12}^{+,-r}(8)$,
which will be formally defined before Theorem~\ref{thm:UY}.
\end{theorem}
The first isomorphism in (\ref{id:J_S_S_even}) follows from Corollary~\ref{cor:J_S_half}
of Proposition~\ref{prop:J_M_even}, 
while the second isomorphism is essentially obtained from \cite[Thm. 1]{UY} (see Theorem~\ref{thm:UY} below).

\subsection{Jacobi forms and modular forms of half-integral weight}
\begin{lemma}\label{lem:g08}
  The group $\Gamma_0(8)$ is generated by the following two types of matrices:
  \begin{equation*}
  u(s) := \begin{pmatrix} 1 & 0 \\ 8s & 1\end{pmatrix} \ (s \in \Z), \quad
  v(a) \ := \ \begin{pmatrix} a & 1\\ a^2-1 & a \end{pmatrix} \ (a \in \{\pm 1, \pm 3\}).
  \end{equation*}
\end{lemma}
\begin{proof}
Let $\Gamma'$ denote the group generated by the matrices above.
A direct calculation shows that
$u(s) v(a) u(s) = v(a+8s)$.
This implies that  $v(a) \in \Gamma'$ for any odd integer $a$.
Note that $ \left(\begin{smallmatrix} - 1 & 0 \\ 0 & - 1\end{smallmatrix} \right) = v(1) v(-1) \in \Gamma'$ 
and $v(a)^{-1} = - v(-a) \in \Gamma'$.

Take $A = \begin{pmatrix} x & y \\ 8z & w \end{pmatrix} \in \Gamma_0(8)$.
We can find an odd integer $a$ such that $|ay + w| < |y|$.
Since $ v(a) A =  \begin{pmatrix} * & a y + w \\ * &  * \end{pmatrix}$, 
by repeatedly multiplying $A$ by such matrices $v(a)$, we can reduce the absolute value of 
the upper-right entry.
The reduction allows us to assume, with loss of generality, that $y = 0$.
In this case, $A$ must be of the form $\pm \left(\begin{smallmatrix} 1 & 0 \\ 8z & 1\end{smallmatrix} \right)$, 
which belongs to $\Gamma'$.
Thus, we conclude that $\Gamma' = \Gamma_0(8)$.
\end{proof}

Recall from $\S$\ref{sec:decom} that any $\phi \in J_{k+\frac{r+1}{2}, D_r}$ has a  decomposition 
of the form:
\begin{equation}\label{id:phi_theta_decom_even}
  \phi(\tau,z) = \sum_{j=0}^3 h_j(\tau) \theta_{r,j}(\tau,z), 
\end{equation}
where the functions $h_j(\tau)$ are given by
\begin{equation*}
  h_j(\tau) = \sum_{n' \in \Z} C(n',v_j) e((n' - \beta(v_j))\tau).
\end{equation*}
Here, $C(n',v_j)$ denotes the $(n',v_j)$-th Fourier coefficient of $\phi$ (see Lemma~\ref{lem:theta_decom}).
We then define the map $\mathcal{J}_{r,k}^{\text{even}}$ by 
\begin{equation}\label{id:J_r_k_even_def}
  (\mathcal{J}_{r,k}^{\text{even}}(\phi))(\tau) := \sum_{j = 0}^3 h_j(8\tau).
\end{equation}

Recall that the vector space $M_{k+\frac12}^{+,-r}(8)$ was defined in~$\S$\ref{ss:half_int_theta}.
\begin{prop}\label{prop:J_M_even}
Assume that $r \in \{ 1, 3, 5, 7\}$ and that
$k+\frac{r+1}{2}$ is an even integer with $k \geq 0$.
Then, the map $\mathcal{J}_{r,k}^{\text{even}}$
 induces an isomorphism of Hecke modules  
 \begin{equation*}
  \mathcal{J}_{r,k}^{\text{even}} \ : \ J_{k+\frac{r+1}{2},D_r} 
   \xrightarrow{\cong}  
 M_{k+\frac12}^{+, -r}(8).
 \end{equation*}
 Namely, for any $\phi \in J_{k+\frac{r+1}{2}, D_r}$ and for any odd prime $p$, we have
 \begin{equation*}
 \mathcal{J}_{r,k}^{\text{even}}(\phi|T^J(p)) = 
   (\mathcal{J}_{r,k}^{\text{even}}(\phi))|T(p^2) .
    \end{equation*}
  Furthermore, the restriction of $\mathcal{J}_{r,k}^{\text{even}}$ to $J_{k+\frac{r+1}{2},D_r}^{\text{cusp}}$
  yields the isomorphism
 \begin{equation*}
  \mathcal{J}_{r,k}^{\text{even}} \ : \ J_{k+\frac{r+1}{2},D_r}^{\text{cusp}} 
  \xrightarrow{\cong}  
 S_{k+\frac12}^{+, -r}(8) .
 \end{equation*}  
\end{prop}

Proposition~\ref{prop:J_M_even} follows from the following three lemmas, 
which establish that $\mathcal{J}_{r,k}^{\text{even}}$ is injective  (Lemma~\ref{lem:J_even_inj}),
surjective (Lemma~\ref{lem:J_even_surj}),
and compatible with the action of Hecke operators (Lemma~\ref{lem:J_even_hecke}). 
In what follows, we provide the proofs of these lemmas.

\begin{lemma}\label{lem:J_even_inj}
 For any $\phi \in J_{k+\frac{r+1}{2},D_r}$, we have $\mathcal{J}_{r,k}^{\text{even}}(\phi) \in M_{k+\frac12}^{+, -r}(8)$.
 Moreover, the map $\mathcal{J}_{r,k}^{\text{even}} \, : \, J_{k+\frac{r+1}{2},D_r} \rightarrow M_{k+\frac12}^{+, -r}(8)$ is injective.
 Furthermore, if $\phi \in J_{k+\frac{r+1}{2},D_r}^{\text{cusp}}$, then $\mathcal{J}_{r,k}^{\text{even}}(\phi) \in S_{k+\frac12}^{+,-r}(8)$.
\end{lemma}
\begin{proof}
 For $\phi \in J_{k+\frac{r+1}{2},D_r}$, let $g = \mathcal{J}_{r,k}^{\text{even}}(\phi)$.
 Since $\beta(v_0) = 0$, $\beta(v_2) = 1/2$ and $\beta(v_1) = \beta(v_3) = r/8$,
 and noting that $h_1 = h_3$, the injectivity of $\mathcal{J}_{r,k}^{\text{even}}$ follows directly from the Fourier expansion of $g$.

We now show that $g \in M_{k+\frac12}^{+,-r}(8)$.
By Lemma~\ref{lem:g08}, it suffices to verify the transformation laws of $g$ for
 the generators $u(s)$ and $v(a)$ of $\Gamma_0(8)$.
To this end, we consider the transformation property of $g(\tau/8) = \sum_{j = 0}^3 h_j(\tau) $ under the actions of 
 \begin{equation*}
 u'(s)  := \begin{pmatrix} 8 & 0 \\ 0 & 1 \end{pmatrix} u(s) \begin{pmatrix} 8 & 0 \\ 0 & 1 \end{pmatrix}^{-1}  = \begin{pmatrix} 1 & 0 \\  s & 1\end{pmatrix} \quad
 (s \in \Z)
 \end{equation*}
 and 
\begin{equation*}
v'(a)  :=  \begin{pmatrix} 8 & 0 \\  0 & 1\end{pmatrix} v(a)
 \begin{pmatrix} 8 & 0 \\  0 & 1\end{pmatrix}^{-1}
 =  \begin{pmatrix} a & 8 \\ (a^2-1)/8 & a \end{pmatrix} \quad (a \in \{\pm 1, \pm 3\}).
\end{equation*}

Note the identity
$  u'(s) = \begin{pmatrix} 0 & -1 \\ 1 & 0 \end{pmatrix} 
  \begin{pmatrix} 1 & -s \\ 0 & 1 \end{pmatrix} 
  \begin{pmatrix} 0 & 1 \\ -1 & 0 \end{pmatrix} $.
By virtue of the identities (\ref{id:theta_2}) and (\ref{id:phi_theta_decom_even}), and the transformation formula for $\phi$, we have
\begin{equation*}
  h_j(-\tau^{-1}) = \frac12 \tau^{k + \frac12} e(r/8) \sum_{t = 0}^3  e(j t r /4) h_t(\tau).
\end{equation*}
For any $s \in \Z$, it follows that
\begin{align*}
  h_j(-(\tau-s)^{-1}) &= \frac12 \left(\tau-s\right)^{k + \frac12} e(r/8) \sum_{t=0}^3 e(j t r /4) h_t(\tau-s)\\
  &= \frac12 \left(\tau-s\right)^{k + \frac12} e(r/8) \sum_{t=0}^3 e(j t r /4) e(t^2 r s /8) h_t(\tau).
\end{align*}
Then, applying (\ref{id:theta_2}) again, we obtain
\begin{align*}
  h_j(-(-\tau^{-1}-s)^{-1}) 
  &= \frac12 \left(-\tau^{-1}-s\right)^{k + \frac12} e(r/8) \sum_{t=0}^3 e(j t r /4) e(t^2 r s /8) h_t(-\tau^{-1}) \\
  &=
  \frac14 \left( -\tau^{-1}-s\right)^{k + \frac12} 
   \tau ^{k + \frac12} e(r/4)
  \sum_{t=0}^3 e(j t r /4)    e(t^2 r s /8) \\
  & \quad
  \times
  \sum_{u=0}^3
   e(t u r / 4) 
   h_u(\tau).
\end{align*}
Since 
$\left(-\tau^{-1}-s\right)^{\frac{1}{2}} \tau^{\frac{1}{2}} = (s\tau + 1)^{\frac12} e(1/4)$
and since $2k + 1 \equiv -r \mod 4$, we have
$\left( -\tau^{-1}-s\right)^{k + \frac12} 
   \tau ^{k + \frac12} = (s\tau + 1)^{k+\frac12} e(-r/4)$.
Using the fact that $\sum_{j=0}^3 e(j t r /4) = 4 \delta_{t, 0}$, 
where $\delta_{t,0}$
denotes the Kronecker delta,
we have
\begin{align*}
  \sum_{j=0}^3 h_j(\tau (s\tau + 1)^{-1}) 
  &=
  \sum_{j=0}^3 h_j(-(-\tau^{-1}-s)^{-1}) \\
  &=
  (s \tau + 1)^{k + \frac12}
  \sum_{u=0}^3 h_u(\tau).
\end{align*}
Therefore, 
$g\!\left(\dfrac{\tau}{8(s\tau+1)} \right) = (s\tau+1)^{k+\frac12} g\!\left(\dfrac{\tau}{8} \right)$, which implies 
 $g|_{k+\frac12} u(s) = g$ for any $s \in \Z$.

For $a \in \left\{ \pm 1 \right\}$, we observe that
$ v'(a) = 
  \begin{pmatrix} a & 8 \\ 0 & a \end{pmatrix}$.
  Since $a = a^{-1}$ for $a  = \pm 1$, it follows from the Fourier expansion of $h_j$ that
 $\sum_{j=0}^3 h_j(\tau + 8 a^{-1}) = \sum_{j=0}^3 h_j(\tau)$. 
Consequently, we obtain $g|_{k+\frac12} v(a) = g$ for $a \in \left\{ \pm 1 \right\}$.

For $a \in \left\{ \pm 3\right\}$, we utilize the decomposition
$
  v'(a) = 
  \begin{pmatrix} 1 & a \\ 0 & 1 \end{pmatrix}
  \begin{pmatrix} 0 & -1 \\ 1 & 0 \end{pmatrix}
  \begin{pmatrix} 1 & a \\ 0 & 1 \end{pmatrix}
$.
Then, we have
\begin{equation*}
  \sum_{j=0}^3 h_j(\tau + a) = \sum_{j=0}^3 e(-a j^2 r /8) h_j(\tau)
\end{equation*}
and
\begin{equation*}
  \sum_{j=0}^3 h_j(-\tau^{-1} + a) = \sum_{j=0}^3 e(-a j^2 r/8) h_j(-\tau^{-1}) \\
  = \frac12 \tau^{k + \frac12} e(r/8)
  \sum_{j=0}^3 e(-a j^2 r/8) \sum_{t} e(j t r/4) h_t(\tau) .
\end{equation*}
It follows that
\begin{align*}
  & \hspace{-1cm} \sum_{j=0}^3 h_j\!\left(-\frac{1}{\tau+a} + a\right) \\
  &= \frac12 \left( \tau+a \right)^{k + \frac12} e\!\left( \frac{r}{8} \right)
  \sum_{j=0}^3 e\!\left(\frac{-a j^2 r}{8} \right) \sum_{t=0}^3 e\!\left(\frac{j t r }{4}\right) h_t(\tau+a) \\
  &= \frac12 \left( \tau+a \right)^{k + \frac12} e\!\left(\frac{r}{8}\right)
  \sum_{j=0}^3 e\!\left(\frac{-a j^2 r}{8} \right) \sum_{t=0}^3 e\!\left(\frac{j t r}{4}\right) e\!\left(\frac{-a t^2 r}{8}\right) h_t(\tau) .  
\end{align*}
A straight-forward calculation shows that
\begin{equation*}
   \sum_{j=0}^3 e\!\left(\frac{-a j^2 r}{8} \right)  e\!\left( \frac{j t r}{4} \right)
   e\!\left( \frac{-a t^2 r}{8} \right)
   =
   2\, e\!\left( \frac{-a r}{8} \right).
\end{equation*}
Since $(a-1)(2k+1+r) \equiv 0 \mod 8$, we obtain
\begin{align*}
  \sum_{j=0}^3 h_j((a\tau + (a^2-1))(\tau+a)^{-1}) 
  &=
    \sum_{j=0}^3 h_j(-(\tau+a)^{-1} + a)  \\
  &= \left( \tau+a \right)^{k + \frac12} e((1-a)r/8)
   \sum_{t=0}^3  h_t(\tau) \\
  &= \left( \tau+a \right)^{k + \frac12} e((a-1)(2k+1)/8)
   \sum_{t=0}^3  h_t(\tau).
\end{align*}
Thus, 
\begin{equation*}
  g\! \left( \frac{a\tau+(a^2-1)}{8(\tau+a)} \right)
  =
  \left( \tau+a \right)^{k + \frac12} e((a-1)(2k+1)/8)
  g\! \left(\frac{\tau}{8}\right).
\end{equation*}
Since the automorphy factor for $v(a)$ is $j(v(a), \tau) = (8\tau+a)^{\frac12} e((a-1)/8)$ for $a \in \left\{ \pm 3 \right\}$,
we conclude that
$  g|_{k+\frac12} v(a) = g$.

Therefore, $g\in M_{k+\frac12}(8)$.
 Since $\beta(v_0) = 0$, $\beta(v_2) = 1/2$ and $\beta(v_1) = \beta(v_3) = r/8$,
the construction of $g$ implies that $g \in M_{k+\frac12}^{+,-r}(8)$.

Next, assume that $\phi$ is a Jacobi cusp form.
Since the constant term of $\phi$ vanishes,
the constant term of $h_0$ is also zero. 
By the transformation formulas for $\{h_j\}_j$,
it follows that the function $ |\mbox{Im}(\tau)^{-k/2-1/4} g(\tau)|$ is bounded on the upper half-plane $\HH$.
This shows that $g$ is a cusp form, and thus $g \in S_{k+\frac12}^{+,-r}(8)$.
\end{proof}

\begin{lemma}\label{lem:J_even_surj}
  Let $g \in M_{k+\frac12}^{+, -r}(8)$. We write the Fourier expansion of $g$ as
  \begin{equation*}
    g(\tau) = \sum_{\begin{smallmatrix} n \in \Z \\ n \equiv 0,4,-r \mod 8 \end{smallmatrix} } c_g(n) q^n 
                 =   \sum_{m = 0, 4, -r} f_m(8 \tau),
  \end{equation*}
  where $\displaystyle{f_m(\tau) := \sum_{\begin{smallmatrix} n \in \Z \\ n \equiv m \mod 8 \end{smallmatrix}} c_g(n) q^{n/8}}$.
  Define the function $J(g)$ by 
  \begin{equation*}
  (J(g))(\tau,z) := \sum_{j = 0}^3 h_j(\tau) \theta_{r,j}(\tau,z),
  \end{equation*}
  where the components $h_j(\tau)$ are set as
  \begin{equation*}
     h_j(\tau) := \begin{cases} 
       f_{2 j}(\tau), & \text{if } j \in \{ 0,2 \}, \\
       \frac12 f_{-r}(\tau), & \text{if } j \in \{ 1, 3\}.
     \end{cases}
  \end{equation*}
    Then $J(g)$ belongs to $J_{k+\frac{r+1}{2},D_r}$.
    Moreover, the map $J$ is the inverse of $\mathcal{J}_{r,k}^{\text{even}}$.
    Furthermore, 
    if $g \in S_{k+\frac12}^{+, -r}(8)$, then $J(g) \in J_{k+\frac{r+1}{2},D_r}^{\text{cusp}}$.
\end{lemma}
\begin{proof}
For simplicity, we write $\phi = J(g)$.
The transformation laws in  Definition~\ref{df:jacobi_lattice}(ii) and (iii) for $\phi$
follow directly from its definition.
Thus, it suffices to verify the transformation law in Definition~\ref{df:jacobi_lattice} (i).
Since the invariance $\phi(\tau+1,z) = \phi(\tau,z)$ follows from the definition of $\phi$ and 
the transformation formulas for $\theta_{r,j}$ under $\left(\begin{smallmatrix} 1 & 1 \\ 0 & 1 \end{smallmatrix}\right)$, 
we only need to show the transformation law with respect to $\left(\begin{smallmatrix} 0 & -1 \\ 1 & 0 \end{smallmatrix}\right)$:
\begin{equation}\label{id:J_inverse}
  \phi(-\tau^{-1}, \tau^{-1} z) =
  \tau^{k+\frac{r+1}{2}} e(\tau^{-1} \beta(z)) \phi(\tau,z).
\end{equation}
To establish (\ref{id:J_inverse}),
it is sufficient to show the following relations for the components $f_m$:
\begin{align}
  \label{id:f_0_inverse}
  f_0(-\tau^{-1}) &=
  \frac12
  \sqrt{\tau}^{2k+1} e(r/8)
   \left\{ f_0(\tau) + f_4(\tau) + f_{-r}(\tau)\right\}, \\
  \label{id:f_4_inverse}
  f_4(-\tau^{-1}) &=
  \frac12
  \sqrt{\tau}^{2k+1} e(r/8)
   \left\{ f_0(\tau) + f_4(\tau) - f_{-r}(\tau)\right\}, \\  
  \label{id:f_r_inverse}
  f_{-r}(-\tau^{-1}) &=
  \sqrt{\tau}^{2k+1} e(r/8)
   \left\{ f_0(\tau) -  f_4(\tau) \right\}.  
\end{align}

Since we have the relation
\begin{equation*}
 \frac14
  \sum_{\begin{smallmatrix} a\!\!\! \mod 8 \\ (a,2) = 2 \end{smallmatrix}} g\! \left( \frac{\tau + a}{8} \right)
  = f_0(\tau) + f_4(\tau),
\end{equation*}
it follows that
\begin{align*}
  f_0(-\tau^{-1}) &=
   \frac18 \sum_{a \!\!\! \mod 8} g\! \left( \frac{-\tau^{-1}+a}{8} \right) \\
  & = 
   \frac12 \left( f_0(-\tau^{-1}) + f_4(-\tau^{-1})\right) +
   \frac18 \sum_{\begin{smallmatrix} a \!\!\! \mod 8 \\ (a,2) = 1 \end{smallmatrix}} g\! \left( \frac{a\tau -1}{8\tau} \right).
\end{align*}
Consequently, the difference $f_0(-\tau^{-1}) - f_4(-\tau^{-1})$ can be expressed as
\begin{align*}
    f_0(-\tau^{-1}) - f_4(-\tau^{-1})
    &=
   \frac14 \sum_{\begin{smallmatrix} a \!\!\! \mod 8 \\ (a,2) = 1 \end{smallmatrix}} g\! \left( \frac{a\tau -1}{8\tau} \right) \\
   & = 
   \frac14 \sum_{a = \pm 1, \pm 3} g\! \left( \begin{pmatrix} a & (a^2-1)/8 \\ 8 & a \end{pmatrix} 
   \begin{pmatrix} 1 & - a \\ 0 & 8 \end{pmatrix} \cdot 
   \tau \right) \\
   &=
   \frac14 \sum_{a=\pm 1, \pm 3} j\!\left( \begin{pmatrix} a & (a^2-1)/8 \\ 8 & a \end{pmatrix}, 
   \begin{pmatrix} 1 & - a \\ 0 & 8 \end{pmatrix} \cdot 
   \tau
    \right)^{2k+1}
   g\! \left( \frac{\tau - a}{8} \right) .
\end{align*}
By using the formula for the automorphy factor
\begin{equation*}
j\! \left( \begin{pmatrix} a & (a^2-1)/8 \\ 8 & a \end{pmatrix}, \begin{pmatrix} 1 & -a \\ 0 & 8 \end{pmatrix} \tau \right)
 = \sqrt{\left( \frac{-4}{a} \right)}^{-1} \left( \frac{8}{a} \right) \sqrt{\tau} = e\!\!\left(\frac{a-1}{8}\right) \sqrt{\tau},
\end{equation*}
which holds for any odd integer $a$, we obtain
\begin{align*}
    f_0(-\tau^{-1}) - f_4(-\tau^{-1})
    &=
    \frac14 \sqrt{\tau}^{2k+1} \sum_{a=\pm 1, \pm 3} e\!\!\left( \frac{(a-1)(2k+1)}{8} \right)  g\!\!\left( \frac{\tau-a}{8} \right) \\
    &=
    \frac14 \sqrt{\tau}^{2k+1} \sum_{m = 0,4,-r}  \sum_{a=\pm 1, \pm 3} e\!\!\left( \frac{(a-1)(2k+1) - a r_i}{8} \right)
    f_m(\tau).
\end{align*}
Under the condition $2k+1+r \equiv 0 \mod 4$, this identity simplifies to
\begin{equation}\label{id:f0-f4}
    f_0(-\tau^{-1}) - f_4(-\tau^{-1})
    =
     \sqrt{\tau}^{2k+1} e(r/8) f_{-r}(\tau).
\end{equation}
Since $\sqrt{-\tau^{-1}}^{-1} = - i \sqrt{\tau}$, substituting $- \tau^{-1}$ for $\tau$ in $(\ref{id:f0-f4})$ 
yields the identity $(\ref{id:f_r_inverse})$.

Since $-\frac{1}{\tau+1} = -1 + \frac{\tau}{\tau+1}$, substituting $\tau+1$ for $\tau$ in $(\ref{id:f0-f4})$ yields
\begin{equation}\label{id:f0f4+}
  f_0\!\left(\frac{\tau}{\tau+1}\right) + f_4\!\left(\frac{\tau}{\tau+1}\right)
  =
  \sqrt{\tau+1}^{2k+1} f_{-r}(\tau).
\end{equation}
Similarly, using the relation $- \frac{1}{-\tau^{-1}-1} = \frac{\tau}{\tau+1}$, the substitution $\tau \rightarrow -\tau^{-1} - 1$ in $(\ref{id:f0-f4})$ gives
\begin{equation}\label{id:f0f4-}
  f_0\!\left(\frac{\tau}{\tau+1}\right) - f_4\!\left(\frac{\tau}{\tau+1}\right)
  =
  \sqrt{-\tau^{-1}-1}^{2k+1} e\!\left( \frac14 r \right)  f_{-r}(-\tau^{-1}).
\end{equation}
By adding $(\ref{id:f0f4+})$ and $(\ref{id:f0f4-})$, we obtain
\begin{equation*}
  f_0\!\left(\frac{\tau}{\tau+1}\right)
  = \frac12 \left\{ \sqrt{\tau+1}^{2k+1} f_{-r}(\tau) +  \sqrt{-\tau^{-1}-1}^{2k+1} e\!\left( \frac14 r \right)  f_{-r}(-\tau^{-1}) \right\}.
\end{equation*}
Noting that $\sqrt{-\tau^{-1}-1} \sqrt{\tau} = e(1/4)\sqrt{\tau+1}$, this can be rewritten as
\begin{equation*}
  f_0\!\left(\frac{\tau}{\tau+1}\right)
  = \frac12 \sqrt{\tau+1}^{2k+1} \left\{  f_{-r}(\tau) +  \sqrt{\tau}^{-2k-1} e\!\left( \frac14 (2k+1+r) \right)  f_{-r}(-\tau^{-1}) \right\}.
\end{equation*}
By virtue of $(\ref{id:f_r_inverse})$ and the condition $2k+1+r \equiv 0 \mod 4$, we arrive at
\begin{equation*}
  f_0\!\left(\frac{\tau}{\tau+1}\right)
  = \frac12 \sqrt{\tau+1}^{2k+1} \left\{  f_{-r}(\tau) +   e\!\left( \frac{r}{8} \right)  (f_0(\tau) - f_4(\tau)) \right\}.
\end{equation*}
Replacing $\tau$ with $\tau - 1 $ in this identity yields (\ref{id:f_0_inverse}).
Finally, the identities $(\ref{id:f_0_inverse})$ and $(\ref{id:f0-f4})$ together provide (\ref{id:f_4_inverse}).

The identity
\begin{equation*}
  \begin{pmatrix}
   h_0(-\tau^{-1}) \\
   h_2(-\tau^{-1}) \\
   h_1(-\tau^{-1}) \\
   h_3(-\tau^{-1})
  \end{pmatrix}
  =
    \frac12 \sqrt{\tau}^{2k+1} e(r/8)
  \begin{pmatrix}
    1 & 1 &  1 & 1 \\
    1 & 1 & -1 & -1 \\
    1 & -1 & e(r/4) & -e(r/4) \\ 
    1 & -1 & -e(r/4) & e(r/4) 
  \end{pmatrix}
    \begin{pmatrix}
   h_0(\tau) \\
   h_2(\tau) \\
   h_1(\tau) \\
   h_3(\tau)
  \end{pmatrix} 
\end{equation*}
follows directly from the identities (\ref{id:f_0_inverse}), (\ref{id:f_4_inverse}) and (\ref{id:f_r_inverse}), along with the fact that 
$h_1 = h_3 = \frac12 f_{-r}$.

Combing this identity with the transformation law for $\{\theta_{r,j}\}$ (see (\ref{id:theta_inverse})),
we obtain (\ref{id:J_inverse}).
This establishes the transformation law in Definition~\ref{df:jacobi_lattice} (i) for $\phi$.
We therefore conclude that $J(g) = \phi \in J_{k+\frac{r+1}{2},D_r}$.

If $g \in S_{k+\frac12}^{+, -r}(8)$, the constant term of $f_0$ vanishes, 
which implies that the constant term of $h_0$ is also zero.
It then follows that $\phi$ is a Jacobi cusp form (see Lemma~\ref{lem:dim_Eisenstein}), 
so $J(g) \in J_{k+\frac{r+1}{2},D_r}^{\text{cusp}}$.

Finally, from the construction of $J$, it is straightforward that $J$ is the inverse of $\mathcal{J}_{r,k}^{\text{even}}$.
\end{proof}

\begin{lemma}\label{lem:J_even_hecke}
 The map $\mathcal{J}_{r,k}^{\text{even}}$ is compatible with the action of Hecke operators.
 That is, for any $\phi \in J_{k+\frac{r+1}{2}, D_r}$ and for any odd prime $p$, we have
 \begin{equation*}
 \mathcal{J}_{r,k}^{\text{even}}(\phi|T^J(p)) = 
   (\mathcal{J}_{r,k}^{\text{even}}(\phi))|T(p^2).
 \end{equation*}
\end{lemma}
\begin{proof}
The proof of this lemma is analogous to that for $\mathcal{J}_{r,k}^{\text{odd}}$ (see the proof of Proposition~\ref{prop:J_M_odd}). 
It suffices to establish the identity between the Fourier coefficients of $\mathcal{J}_{r,k}^{\text{even}}(\phi|T^J(p))$ and 
those of $(\mathcal{J}_{r,k}^{\text{even}}(\phi))|T(p^2)$.

Consider the Fourier expansion of $\phi \in J_{k+\frac{r+1}{2},D_r}$:
\begin{equation*}
    \phi(\tau,z) =  \sum_{n' \in \Z} \sum_{r' \in D_r^\sharp} C(n',r') e(n'\tau + \beta(r',z)) .
\end{equation*}
Since the weight $k+\frac{r+1}{2}$ is even, $\phi$ satisfies $\phi(\tau,-z) = \phi(\tau,z)$,
which implies that the relation $C(n',-r') = C(n',r')$ holds
for all $(n',r') \in \Z \times D_r^\sharp$.

  From the definition of $T^J(p)$ (see Lemma~\ref{lem:hecke_fourier}), 
  we have
  \begin{equation*}
  (\phi |_{k+\frac12,D_r} T^J(p))(\tau,z)
  = \sum_{n' \in \Z} \sum_{r' \in D_r^\sharp} C^{**}(n',r') e(n'\tau + \beta(r',z)),
  \end{equation*}
  where the Fourier coefficients $C^{**}(n',r')$ are given by
  \begin{align}
  \label{id:C^**}
   C^{**}(n',r') 
  &=
    C(p^{2}n', p r') \\
  \notag
  & \quad
    + p^{k-1}
       \left( \frac{(-1)^{k}}{p} \right)  
      \left( \frac{8(n'- \beta(r'))}{p} \right) 
          C(n',r') \\
  \notag
    & \quad
    + p^{2k-1}
    \sum_{\lambda \in D_r/ p D_r }C\!\left(\frac{1}{p^{2}}(n' - \beta(r',\lambda) + \beta(\lambda)), \frac{1}{p}(r' -  \lambda)\right).
  \end{align}  
  Note that the quadratic character in the second term on the right-hand side differs from that in (\ref{id:C^*}).
  
  For simplicity, let $g = \mathcal{J}_{r,k}^{\text{even}}(\phi)$, and write its Fourier expansion as
  $ g(\tau) = \sum_{m = 0}^{\infty} A(m) q^{m}$.
  Then, the Fourier expansion of $g|_{k+\frac12}T(p^2)$ is given by
  \begin{equation*}
  (g|_{k+\frac12}T(p^2))(\tau) =  \sum_{m = 0}^\infty A^{**}(m) q^{m},
  \end{equation*}
  where the Fourier coefficients $A^{**}(m)$ are given by
  \begin{equation}\label{id:A^**}
    A^{**}(m) =     
    A(p^2 m) + \left( \frac{(-1)^k m }{p}\right) p^{k-1} A(m) + p^{2k-1} A(m/p^2).
  \end{equation}
  From the definition of $\mathcal{J}_{r,k}^{\text{even}}$, the relation between $C(n',r')$ and $A(m)$ is given by
  \begin{equation*}
  C(n',r') = 
  \begin{cases} 
    A(8(n' - \beta(r'))), & \mbox{if $r' \equiv v_0, v_2 \mod D_r$}, \\
    \frac12 A(8(n' - \beta(r'))), & \mbox{if $r' \equiv v_1, v_3 \mod D_r$},
    \end{cases}
  \end{equation*}
  for any $(n',r') \in \Z \times D_r^{\sharp}$.
  We also note that $C(n',v_1) = C(n',v_3)$.

  By comparing the right-hand sides of (\ref{id:C^**}) and (\ref{id:A^**}),
  we find that
  \begin{equation*}
  C^{**}(n',v_j) = \begin{cases}
   A^{**}(8(n'-\beta(v_j))), & \text{if } j \in \{ 0, 2\} , \\
    \frac12 A^{**}(8(n' - \beta(v_j))), & \text{if } j \in \{1, 3\}.
   \end{cases}
  \end{equation*}
  Summing over $j \in \{0,1,2,3\}$, we conclude that
  $ \mathcal{J}_{r,k}^{\text{even}}(\phi|T^J(p)) = 
   (\mathcal{J}_{r,k}^{\text{even}}(\phi))|T(p^2) $.
\end{proof}

We are now ready to complete the proof of Proposition~\ref{prop:J_M_even}.

\begin{proof}[Proof of Proposition~\ref{prop:J_M_even}]
The assertion of 
Proposition~\ref{prop:J_M_even} follows from Lemma~\ref{lem:J_even_inj} (the injectivity of $\mathcal{J}_{r,k}^{\text{even}}$), 
Lemma~\ref{lem:J_even_surj} (the surjectivity of $\mathcal{J}_{r,k}^{\text{even}}$), and
Lemma~\ref{lem:J_even_hecke} (the compatibility of $\mathcal{J}_{r,k}^{\text{even}}$ with Hecke operators).
\end{proof}

\subsection{Shimura correspondence and newforms}\label{ss:Shimura_even}
The theory of newforms in $S_{k+\frac12}^+(8)$ was introduced in \cite[Introduction]{UY}.
First, we recall the definition of the operators $U(d)$ and $U_k(4) = U(4) \wp_k$:
\begin{equation*}
  (\sum_{n \in \Z} a_n q^n)|U(d) = \sum_{n \in \Z} a_{dn} q^n, \qquad
  (\sum_{n\in \Z } a_n q^n )| \wp_k \ =\ \!\!\!\! \sum_{\begin{smallmatrix} n \in \Z \\ (-1)^k n \equiv 0,1 \mod 4 \end{smallmatrix}} 
  \!\!\!\!\! a_n q^n.
\end{equation*}
defined for any formal power $q$-series. 
It was shown in~\cite[Thm. 1 (1), Prop. 6]{UY} that the image 
$S_{k+\frac12}^{+}(4)|U_k(4) := \left\{ f|U_k(4) \, : \, f \in S_{k+\frac12}^+(4) \right\}$ is a subspace of $S_{k+\frac12}^{+}(8)$. 
We define the space of oldforms by
\begin{equation*}
  S_{k+\frac12}^{\text{old},+}(8) := S_{k+\frac12}^+(4) + S_{k+\frac12}^{+}(4)|U_k(4).
\end{equation*}
The space of newforms $S_{k+\frac12}^{\text{new},+}(8)$ is then defined as the orthogonal complement of $S_{k+\frac12}^{\text{old},+}(8)$
in $S_{k+\frac12}^{+}(8)$.
It is known from \cite[Thm. 1 (1)]{UY} that we have the following direct sum decompositions:
\begin{align*}
  S_{k+\frac12}^{\text{old},+}(8) &=  S_{k+\frac12}^+(4) \oplus S_{k+\frac12}^{+}(4)|U_k(4), \\
  S_{k+\frac12}^{+}(8) &=  S_{k+\frac12}^{\text{new},+}(8) \oplus S_{k+\frac12}^{\text{old},+}(8).
\end{align*}
Finally, we set 
\begin{equation*}
 S_{k+\frac12}^{\text{new},+,-r}(8) := 
 S_{k+\frac12}^{\text{new},+}(8) \cap  S_{k+\frac12}^{+,-r}(8).
\end{equation*}

The following theorem essentially follows from \cite[Thm. 1, Prop. 4, Cor. 2]{UY}.
\begin{theorem}[{\cite{UY}}]\label{thm:UY}
For any $k \in \Z$ and $r \in \{1, 3, 5, 7\}$, we have an isomorphism of Hecke modules:
\begin{equation*}
 S_{k+\frac12}^{\text{new},+,-r}(8) \ \cong \ S_{2k}^{\text{new}, \epsilon_2}(2),
\end{equation*}
where $\epsilon_2 = - \left( \frac{-8}{r} \right)$.
\end{theorem}
\begin{proof}
According to \cite[Thm. 1 (3)]{UY}, there exists an isomorphism 
\begin{equation*}
  J_k \ : \ S_{k+\frac12}^{\text{new},+}(8) \xrightarrow{\cong}  S_{2k}^{\text{new}}(2) 
\end{equation*}
as Hecke modules. 
This means that,  
\begin{equation*}
  J_k(g|T(p^2)) = J_k(g)|T(p) \quad \text{and} \quad J_k(g|U_k(4)) = J_k(g)|U(2)
\end{equation*}
for any $g \in S_{k+\frac12}^{\text{new},+}(8)$ and any odd prime $p$,
where $T(p)$ denotes the usual Hecke operator acting on $S_{2k}(2)$.

According to~\cite[Cor. 2, Prop. 4]{UY}, it is known that the identity
\begin{equation*}
  2^{1-k}  g|U_k(4) = -\epsilon g
\end{equation*}
is equivalent to the condition that
\begin{equation*}
  a_g(n) = 0 \quad \text{whenever} \quad \left( \frac{8}{(-1)^k n} \right) = - \epsilon,
\end{equation*}
where $a_g(n)$ denotes the $n$-th Fourier coefficient of $g$.

Assume that $g \in S_{k+\frac12}^{\text{new},+,-r}(8)$ is a Hecke eigenform.
To complete the proof, it suffices to show that $J_k(g) \in S_{2k}^{\text{new},\epsilon_2}(2)$.
Since  $g \in S_{k+\frac12}^{\text{new},+,-r}(8)$, 
its $n$-th Fourier coefficient $a_g(n)$ vanishes unless $n \equiv 0,4, -r \!\! \mod 8$.
This implies that
$g|U_k(4) = 2^{k-1} \left(\frac{8}{(-1)^k m} \right) g$, where 
$m$ is chosen such that $m \equiv 4 - r \!\! \mod 8$.
Specifically, we can take
\begin{equation*}
  m = \begin{cases}
    -5, & \text{if } r = 1, \\
    -7, & \text{if } r = 3, \\
    -1, & \text{if } r = 5, \\
    -3, & \text{if } r = 7.
  \end{cases}
\end{equation*}
Recalling that $k + \frac{r+1}{2} $ is even,
we have $(-1)^k = (-1)^{(r+1)/2}$. A direct calculation of the Kronecker symbol then yields
$g|U_k(4) = - \left(\frac{8}{r} \right) 2^{k-1}  g$.

Let $f = J_k(g)$ and consider its Fourier expansion $f = \sum a_n q^n$.
Then we have $f|U(2) = a_2 f$.
Since $J_k(g|U_k(4)) = J_k(g)|U(2)$, it follows that $a_2 = - \left(\frac{8}{r} \right) 2^{k-1} $.
It is known that
\begin{equation*}
  f|_{2k}\left( \begin{smallmatrix} 0 & -1 \\ N & 0 \end{smallmatrix} \right)
  =
  c f
\end{equation*}
with $c = -2^{1-k} a_2$ (see, e.g., \cite[Introduction]{zhang}).
Substituting the expression for $a_2$, we obtain
\begin{equation*}
  c = \left( \frac{8}{r} \right)  = (-1)^k (-1)^{\frac{r+1}{2}}  \left( \frac{8}{r} \right)  =  - (-1)^k \left( \frac{-8}{r}\right).
\end{equation*}
Recalling the definition $\epsilon_2 = - \left( \frac{-8}{r}\right)$, we have
$c = \epsilon_2 i^{-2k}$, as required.
We therefore conclude that $J_k(g) = f \in S_{2k}^{\text{new},\epsilon_2}(2)$.
\end{proof}

The following corollary follows from Proposition~\ref{prop:J_M_even} and Theorem~\ref{thm:UY}. 
Recall the assumptions $r \in \{1,3,5,7\}$ and $k \geq 0$.
\begin{cor}\label{cor:J_S_half}
There is an isomorphism of Hecke modules
\begin{equation*}
  J_{k+\frac{r+1}{2},D_r}^{\text{cusp, new}} \cong S_{k+\frac12}^{\text{new},+,-r}(8).
\end{equation*}
\end{cor}
\begin{proof}
  Let
 \begin{equation*}
   S_{k+\frac12}^{(1),-r} := \mathcal{J}_{r,k}^{\text{even}}(J_{k+\frac{r+1}{2},D_r}^{\text{cusp, new}})
   \quad \text{and} \quad
   S_{k+\frac12}^{(0),-r} := \mathcal{J}_{r,k}^{\text{even}}(J_{k+\frac{r+1}{2},D_r}^{\text{cusp, old}}).
 \end{equation*}
 By Proposition~\ref{prop:J_M_even}, the space $S_{k+\frac12}^{+,-r}(8)$ admits the direct sum decomposition
  $S_{k+\frac12}^{+,-r}(8) = S_{k+\frac12}^{(1),-r} \oplus S_{k+\frac12}^{(0),-r}$.
 Thus, it suffices to show that $S_{k+\frac12}^{(1),-r} = S_{k+\frac12}^{\text{new},+,-r}(8)$.

We set $S_{k+\frac12}^{\text{old},+,-r}(8) := S_{k+\frac12}^{\text{old},+}(8) \cap S_{k+\frac12}^{+,-r}(8)$.
From the direct sum decomposition $S_{k+\frac12}^{+}(8) = S_{k+\frac12}^{\text{new},+}(8) \oplus S_{k+\frac12}^{\text{old},+}(8)$,
it follows that $S_{k+\frac12}^{+,-r}(8) = S_{k+\frac12}^{\text{new}, +,-r}(8) \oplus S_{k+\frac12}^{\text{old},+,-r}(8)$.

By Theorem~\ref{thm:UY} and the strong multiplicity one theorem (see \cite[Cor. 4.6.30]{Miyake}), we obtain 
the inclusion 
$S_{k+\frac12}^{(0),-r} \subset S_{k+\frac12}^{\text{old},+,-r}(8)$. 
This inclusion can be shown as follows.
Suppose that $\phi \in J_{k+\frac{r+1}{2},D_r}^{\text{cusp, old}}$ is a Hecke eigenform. 
Then, there exists a Hecke eigenform
$f_1 \in S_{2k}^{\epsilon_1}(1)$ such that the eigenvalues of $\phi$ are determined by those of $f_1$.
If we assume that $\mathcal{J}_{r,k}^{\text{even}}(\phi) \in S_{k+\frac12}^{\text{new},+,-r}(8)$, 
then there must exist a newform $f_2 \in S_{2k}^{\text{new},\epsilon_2}(2)$
whose eigenvalues correspond to those of $\phi$.
However, this would imply that the eigenvalues of $f_1$ and $f_2$ coincide, which contradicts the strong multiplicity one theorem.
Thus, we obtain the desired inclusion.

According to \cite[Thm. 1 (1)]{UY}, it is known that
 \begin{equation*}
   \dim S_{k+\frac12}^{+}(8)
   = \dim S_{k+\frac12}^{\text{new},+}(8) + 2 \dim S_{k+\frac12}^+(4)
   = \dim S_{k+\frac12}^{\text{new},+}(8) + 2 \dim S_{2k}(1),
 \end{equation*}  
 where the last equality follows from~\cite{Ko}.
 In view of (\ref{id:S_+_r}), we obtain the following decomposition for the dimension:
 \begin{equation*}
   \dim S_{k+\frac12}^{+}(8)
    =
     \begin{cases}
       \dim S_{k+\frac12}^{+,-1}(8) +  \dim S_{k+\frac12}^{+,-5}(8),
       & \text{if } k \text{ is odd}, \\
       \dim S_{k+\frac12}^{+,-3}(8) +  \dim S_{k+\frac12}^{+,-7}(8),
       & \text{if } k \text{ is even}.
     \end{cases}
 \end{equation*}
 
First, we consider the cases $r=1$ and $r=5$. 
We have the following relation for the dimensions:
 \begin{align*}
   \left(\dim S_{k+\frac12}^{(1),-1} + \dim S_{k+\frac12}^{(0),-1} \right)
   & +  \left(  \dim S_{k+\frac12}^{(1),-5} + \dim S_{k+\frac12}^{(0),-5} \right) \\
   &=
      \dim S_{k+\frac12}^{+}(8) \\ 
   &=
   \dim  S_{k+\frac12}^{\text{new},+}(8) + 2 \dim S_{2k}(1).
 \end{align*}
 
 Recall from Proposition~\ref{prop:J_old} that
 $\dim S_{k+\frac12}^{(0),-1} = \dim S_{k+\frac12}^{(0),-5} = \dim S_{2k}(1)$.
 
 Since we have the inclusion $S_{k+\frac12}^{(0),-1} \oplus S_{k+\frac12}^{(0),-5} \subset
  S_{k+\frac12}^{\text{old},+,-1}(8) \oplus S_{k+\frac12}^{\text{old},+,-5}(8) = 
  S_{k+\frac12}^{\text{old},+}(8)$,
  and since $\dim S_{k+\frac12}^{\text{old},+}(8) = 2 \dim S_{k+\frac12}^+(4) = 2 \dim S_{2k}(1)$,
  the equality of dimensions implies that
  \begin{equation*}
  S_{k+\frac12}^{(0),-1} = S_{k+\frac12}^{\text{old},+,-1}(8)
  \quad \text{and} \quad 
  S_{k+\frac12}^{(0),-5} = S_{k+\frac12}^{\text{old},+,-5}(8).
  \end{equation*}
  Consequently, we ontain
  \begin{equation*} 
  S_{k+\frac12}^{(1),-1} = S_{k+\frac12}^{\text{new},+,-1}(8) \quad
  \text{and} \quad 
  S_{k+\frac12}^{(1),-5} = S_{k+\frac12}^{\text{new},+,-5}(8).
  \end{equation*}
  
  The cases for $r = 3$ and $r = 7$ follows from an analogous argument. 
\end{proof}

We are now in a position to complete the proof of Theorem~\ref{thm:even_weight_case}.
\begin{proof}[Proof of Theorem~\ref{thm:even_weight_case}]
Theorem~\ref{thm:even_weight_case} follows immediately from
Corollary~\ref{cor:J_S_half}, which establishes the isomorphism 
\begin{equation*}
J_{k+\frac{r+1}{2},D_r}^{\text{cusp, new}} \cong S_{k+\frac12}^{\text{new},+,-r}(8),
\end{equation*}
and Theorem~\ref{thm:UY} which provides the isomorphism
\begin{equation*}
 S_{k+\frac12}^{\text{new},+,-r}(8) \ \cong \ S_{2k}^{\text{new}, \epsilon_2}(2).
\end{equation*}
Combining these two isomorphisms as Hecke modules, we obtain the assertion.
\end{proof}

\section{Proof of Theorems~\ref{thm:main} and \ref{thm:k_1}}\label{sec:proof_main_k_1}
In this section, we provide the proofs of the main results of this paper, which establish the isomorphisms between
Jacobi forms of index $D_r$ and elliptic modular forms of level $2$.

\subsection{Proof of Theorem~\ref{thm:main}}
We first provide the proof of Theorem~\ref{thm:J_new}
\begin{proof}[Proof of Theorem~\ref{thm:J_new}]
Theorem~\ref{thm:J_new} follows from Theorem~\ref{thm:odd_weight_case} 
for the case where the weight $k + \frac{r+1}{2}$ is odd, which provides the isomorphisms
\begin{equation*}
  J_{k+\frac{r+1}{2},D_r}^{\text{cusp}} \cong \eta^{24-3r} M_{k-12+\frac{3r+1}{2}}(1)  \cong  S_{2k}^{\text{new}, \epsilon_2}(2) 
\end{equation*}
and from Theorem~\ref{thm:even_weight_case} for the case where the weight $k+\frac{r+1}{2}$ is even, which gives
\begin{equation*}
  J_{k+\frac{r+1}{2},D_r}^{\text{cusp, new}} \cong S_{k+\frac12}^{\text{new},+,-r}(8)  \cong  S_{2k}^{\text{new}, \epsilon_2}(2).
\end{equation*}
\end{proof}
We note that if $k + \frac{r+1}{2} \equiv 1 \!\! \mod 2$, it follows from Theorem~\ref{thm:J_decom}
that
$J_{k+\frac{r+1}{2},D_r}^{\text{old}} = \{ 0 \}$ and thus
$J_{k+\frac{r+1}{2},D_r}^{\text{cusp}} = J_{k+\frac{r+1}{2},D_r}^{\text{cusp, new}} $.

We are now ready to prove Theorem~\ref{thm:main}. 

\begin{proof}[Proof of Theorem~\ref{thm:main}]
Theorem~\ref{thm:main} follows from the decomposition established in Theorem~\ref{thm:J_decom} 
\begin{equation*}
 J_{k+\frac{r+1}{2},D_r} = J_{k+\frac{r+1}{2},D_r}^{\text{new}} \oplus J_{k+\frac{r+1}{2},D_r}^{\text{old}} 
\cong J_{k+\frac{r+1}{2},D_r}^{\text{new}} \oplus M_{2k}^{\epsilon_1}(1)  
\end{equation*}
together with the isomorphism from Theorem~\ref{thm:J_new}: 
\begin{equation*}
 J_{k+\frac{r+1}{2},D_r}^{\text{new}} \cong S_{2k}^{\epsilon_2}(2).
\end{equation*}
Combining these results yields the assertion.
\end{proof}

\subsection{Proof of Theorem~\ref{thm:k_1}}\label{ss:proof_thm_k_1}
In this subsection, we provide the proof of Theorem~\ref{thm:k_1}, which is an analogue of Theorem~\ref{thm:J_new} for the case $k=1$.
\begin{proof}[Proof of Theorem~\ref{thm:k_1}]
First, we establish the following identities: 
\begin{equation}\label{id:dim_k_1_minus}
  \dim J_{2,D_1} = \dim J_{3,D_3} = \dim M_{3/2}^{+,-1}(8) = \dim \eta^{15} M_{-6}(1) = \dim M_{2}^-(2) = 0.
\end{equation}
The vanishing of $J_{2,D_1}$ and $J_{3,D_3}$ was shown in \cite{BS} (see also \cite[Thm. 3.29]{Mocanu}).
By Proposition~\ref{prop:J_M_even}, we have $\dim M_{3/2}^{+,-1}(8) = \dim J_{2,D_1} = 0$.
The identities $\dim (\eta^{15} M_{-6}(1)) = \dim  M_{-6}(1) = 0$ are trivial.
Furthermore, one can easily verify that $\dim M_2(2) = \dim M_2^{+}(2) = 1 $ and $\dim M_2^{-}(2) = 0$.
Thus, (\ref{id:dim_k_1_minus}) holds.

To prove the remaining assertion in Theorem~\ref{thm:k_1}, 
we calculate the Hecke eigenvalues for $\eta^3 \in \eta^3 M_0(1)$ and$E_{4,D_5} \in J_{4,D_5}$.

It is well-known that $\eta^3$ has the Fourier expansion $\eta^3(\tau) =  \sum_{n=1}^\infty \left( \frac{-4}{n} \right) n q^{n^2/8}$.
A direct computation then yields
\begin{equation*}
  \eta^3| \tilde{T}(p^2) = (1+p) \eta^3
\end{equation*}
for any odd prime $p$,
where $\tilde{T}(p^2)$ is the twisted  Hecke operator defined in $\S$\ref{ss:Hecke_op}. 
Since the Eisenstein series $E_2^{(2)}$ of weight $2$ and level $2$ also has the Hecke eigenvalue $1+p$ 
for any odd prime $p$, we obtain an isomorphism of Hecke modules $\C \eta^3 \cong \C E_2^{(2)}$.

The identity $\dim J_{4,D_5} = 1$ was shown by \cite{BS} (see~\cite[Thm. 3.29]{Mocanu}).
Since $E_{4,D_5}$ belongs to $J_{4,D_5}$, it is necessarily a Hecke eigenform.
The constant term of $E_{4,D_5}$ is non-zero, thus its eigenvalue can be determined
by computing the constant term of $E_{4,D_5}|T^J(p)$ (see $\S$\ref{ss:Hecke_op} for the definition of $T^J(p)$). 
A straightforward calculation then shows that 
\begin{equation*}
 E_{4,D_5} |T^J(p) = (1+p) E_{4,D_5}
\end{equation*}
for any odd prime $p$.
Consequently, we have  $\C E_{4,D_5} \cong \C E_2^{(2)}$ as Hecke modules.

By Proposition~\ref{prop:J_M_odd},
we obtain the identities 
\begin{equation*}
\dim J_{5,D_7}^{\text{cusp}} = \dim (\eta^3 M_0(1)) = \dim M_{2}^{+}(2) = 1
\end{equation*}
and the isomorphisms 
\begin{equation*}
  \C \psi_{5,D_7} \cong \C \eta^3 \cong \C E_2^{(2)}
\end{equation*}
as Hecke modules, where $\psi_{5,D_7}$ is a non-zero element of $J_{5,D_7}^{\text{cusp}}$.

The identity $\dim J_{4,D_5} = 1$ was established by \cite{BS} (see also \cite[Thm.~3.29]{Mocanu}).
By Proposition~\ref{prop:J_M_even}, we have 
an isomorphism $M_{3/2}^{+,-5}(8) \cong J_{4,D_5}$ as Hecke modules.
Since $\dim J_{4,D_5}^{\text{old}} = \dim M_2^{-}(1) = 0$,
it follows that $J_{4,D_5} = J_{4,D_5}^{\text{new}}$, 
and hence $M_{3/2}^{+,-5}(8) \cong J_{4,D_5}^{\text{new}}$.
Recalling that $E_{3/2}^{(8)} \in M_{3/2}^{+,-5}(8)$ and $E_{4,D_5} \in J_{4,D_5}^{\text{new}}$, 
we obtain $\C E_{3/2}^{(8)} \cong \C E_{4,D_5}$ as Hecke modules.

Furthermore, noting that $J_{5,D_7}^{\text{cusp}} = J_{5,D_7}^{\text{cusp, new}}$ by Theorem~\ref{thm:J_decom}, we conclude that
\begin{equation*}
  \dim J_{4,D_5}^{\text{new}} = \dim J_{5,D_7}^{\text{cusp, new}} 
  = \dim M_{3/2}^{+,-5}(8)   
  = \dim (\eta^3 M_0(1)) = \dim M_{2}^{+}(2) = 1
\end{equation*}
and
\begin{equation*}
   \C E_{4,D_5} \cong  \C \psi_{5,D_7} \cong \C E_{3/2}^{(8)} \cong  \C \eta^3 \cong \C E_2^{(2)}
\end{equation*}
as Hecke modules.
\end{proof}

\section{Fourier coefficients of Jacobi--Eisenstein series}\label{sec:Fourier_coeff_Jacobi_Eisenstein}

In this section, we provide an explicit formula for the Fourier coefficients of the Jacobi--Eisenstein series 
$E_{k+\frac{r+1}{2},D_r} \in J_{k+\frac{r+1}{2},D_r}$
for $k \geq 2$ (see Definition~\ref{def:Jacobi_Eisenstein} and Theorem~\ref{thm:Jacobi_Eisenstein}),
and we prove Theorem~\ref{thm:E_3_2_8}.

\subsection{Cohen-type Eisenstein series of level $8$}\label{ss:cohen_type_level_8}
Let $\mathscr{H}_k$ ($k \in \mathbb{N}$) be the Cohen--Eisenstein series of level $4$ and weight $k+\frac12$ 
defined by
\begin{equation*}
  \mathscr{H}_k(\tau) := \zeta(1-2k) + \sum_{N>0} H(k,N) q^N
\end{equation*}
(see \cite[pp.~272--273]{cohen} for the definition of $H(k,N)$).
If $k \geq 2$, then $ \mathscr{H}_k \in M_{k+\frac12}^{+}(4)$.

Suppose $N$ is written as $(-1)^k N = D f^2$, where $D$ is a fundamental discriminant and $f \in \mathbb{N}$.
Then $H(k,N)$ satisfies
\begin{equation}\label{id:H_N_Dff}
  H(k,N) = H(k,|D|) \sum_{d|f} \mu(d) \left( \frac{D}{d} \right) d^{k-1} \sigma_{2k-1}(f/d)
\end{equation}
and $H(k,|D|) = L(1-k,\left(\frac{D}{\cdot}\right))$, where $L(s,\chi)$ is the Dirichlet L-function.

Let $r \in \{1,3,5,7 \}$. We assume that $(-1)^k \equiv -r \mod 4$.
Specifically, 
if $k$ is odd, then $r \in \{1,5\}$,  and if $k$ is even, then $r \in \{3,7\}$.
Note that $k + \frac{r+1}{2} \in 2\Z$.

We define the Cohen-type Eisenstein series of level $8$ by
\begin{equation*}
  \mathscr{H}^*_{r,k}(\tau) := \zeta(1-2k) + \sum_{N > 0} H^*_r(k,N) q^N ,
\end{equation*}
where we set
\begin{equation*}
  H^*_r(k,N) :=  H(k,N) + \frac{H(k,N)  - H(k,4N)}{2^{k-1} \left( \left( \frac{8}{r}\right)  + 2^k \right)} .
\end{equation*}

By definition, we have
\begin{equation*}
2^{k-1} \left( \left( \frac{8}{r}\right)  + 2^k \right) \mathscr{H}_{r,k}^* =  \left( 1 + \left( \frac{8}{r}\right) 2^{k-1} + 2^{2k-1} \right) \mathscr{H}_k - \mathscr{H}_k|U_k(4),
\end{equation*}
where the operator $U_k(4)$ is defined in~$\S$\ref{ss:half_3_2_and_2}.
Since $\mathscr{H}_k|U_k(4) \in M_{k+\frac12}^{+}(8)$ for $k \geq 2$ follows from~\cite[Prop.~6]{UY},
it follows that $\mathscr{H}_{r,k}^* \in M_{k+\frac12}^{+}(8)$ for $k \geq 2$.

The vector space $M_{k+\frac12}^{+,-r}(8)$ was defined in~$\S$\ref{ss:half_int_theta}.

\begin{prop}\label{prop:H_r_k_in}
Assume $k+\frac{r+1}{2}$ is even with $k \geq 2$.
Then we have
\begin{equation*}
  M_{k+\frac12}^{+,-r}(8) = \C \mathscr{H}^*_{r,k} \oplus S_{k+\frac12}^{+,-r}(8).
\end{equation*}
In particular, $\mathscr{H}_{r,k}^* \in M_{k+\frac12}^{+,-r}(8)$.
\end{prop}
\begin{proof}
Suppose $N \equiv 4 - r  \! \! \mod 8$.
Since $((-1)^k N)^2 \equiv r^2 - 8  \! \! \mod 16 $, we obtain 
$\left( \frac{8}{(-1)^k N} \right) = (-1)^{(((-1)^k N)^2 -1)/8} = (-1)^{(r^2-9)/8} = - \left( \frac{8}{r} \right)$.
By virtue of the property (\ref{id:H_N_Dff}), we have
\begin{equation}\label{id:H_4N_H_N}
  H(k,4N) = \left(1 + 2^{2k-1} - \left( \frac{8}{(-1)^k N}\right) 2^{k-1} \right) H(k,N).
\end{equation}
It then follows that $H^*_r(k,N) = 0 $ for $N \equiv 4 - r \!\! \mod 8$.
Combined with the fact that $\mathscr{H}_{r,k}^* \in M_{k+\frac12}^{+}(8)$, 
this implies $\mathscr{H}_{r,k}^* \in M_{k+\frac12}^{+,-r}(8)$.

Furthermore, by Proposition~\ref{prop:J_M_even} and Lemma~\ref{lem:dim_Eisenstein}, we have
\begin{equation*}
 \dim M_{k+\frac12}^{+,-r}(8) - \dim S_{k+\frac12}^{+,-r}(8) =  \dim J_{k+\frac{r+1}{2},D_r} - \dim J_{k+\frac{r+1}{2},D_r}^{\text{cusp}}
  \leq 1.
\end{equation*}
Since $\mathscr{H}_{r,k}^*  \notin  S_{k+\frac12}^{+,-r}(8)$
(as its constant term is non-zero),
the assertion follows.
\end{proof}

\begin{lemma}\label{lem:H*_formula}
 Let $N$ be a positive-integer such that $N \equiv 0, 4, -r  \!\! \mod 8$.
 There exists $(n', r') \in \Z \times D_r^\sharp$ such that $N = 8 (n' - \beta(r'))$.
 Write $(-1)^k N = D f^2$, where $D$ is a fundamental discriminant and $f \in \mathbb{N}$.

\begin{itemize}
 \item[(i)]
 If $f$ is odd, then
 \begin{equation}\label{id:e_r_k_odd}
   H^*_r(k,N) = 
     \frac{ \left( 1 + \left( \frac{8}{r} \right) \left( \frac{8}{D} \right) \right) H(k,N)}{\left( 1 + \left( \frac{8}{r}\right) 2^k \right) }.
 \end{equation}
 \item[(ii)]
 If $f$ is even, then
 \begin{equation}\label{id:e_r_k_even}
   H^*_r(k,N) = 
     \frac{H(k,N) + \left( \frac{8}{r}\right) 2^k H(k,\frac{N}{4}) }{\left( 1 + \left( \frac{8}{r}\right) 2^k \right) }.
 \end{equation}
 \end{itemize}
 Note that if $N \equiv -r \!\! \mod 8$, then $f$ is odd and 
 $\left( \frac{8}{D} \right) = \left( \frac{8}{(-1)^k N} \right) = \left( \frac{8}{r} \right) $.
 Furthermore, if $N \equiv 0 \!\! \mod 4$, then $D \equiv 0 \!\! \mod 4$ or  $f$ is even.
\end{lemma}
\begin{proof}
The formula (\ref{id:e_r_k_odd}) follows from 
the identity (\ref{id:H_4N_H_N}) established in the proof of Proposition~\ref{prop:H_r_k_in}, 
combined with the assumption $f \equiv 1 \!\! \mod 2$.

The formula (\ref{id:e_r_k_even}) follows from 
the identity 
\begin{equation*}
 \left( 1 +  2^{2k-1} \right) H(k,N)- H(k,4N) 
 =  2^{2k-1} H\!\left(k,\frac{N}{4}\right),
\end{equation*}
which holds under the assumption $f \equiv 0 \!\! \mod 2$.
\end{proof}

\subsection{Jacobi--Eisenstein series}

\begin{df}\label{def:Jacobi_Eisenstein}
For $k \geq 2$ and $r \in \{1,3,5,7 \}$ such that $k+\frac{r+1}{2}$ is even, we define
the Jacobi--Eisenstein series of weight $k+\frac{r+1}{2}$ and index $D_r$ by
\begin{equation*}
  E_{k+\frac{r+1}{2},D_r}(\tau,z) := \sum_{A \in SL(2,\Z)/\Gamma_\infty} \sum_{\lambda \in D_r} 1 |_{k,D_r} (\lambda,0) |_{k,D_r} A,
\end{equation*}
where $\Gamma_\infty := \left\{ \pm \left( \begin{smallmatrix} 1 & n \\ 0 & 1 \end{smallmatrix} \right) \, : \, n \in \Z \right\} $.
\end{df}
For the convergence of $E_{k+\frac{r+1}{2},D_r}$, we refer the reader to \cite[Prop. 3.3.6]{ali} or \cite[Thm. 2.6]{Mocanu}.

Note that the cases $(r,k) = (5,1)$ and $(r,k) = (7,0)$ for $E_{4,D_5}$ and $E_{4,D_7}$ are not included in Definition~\ref{def:Jacobi_Eisenstein}.
These functions, $E_{4,D_5}$ and $E_{4,D_7}$, are instead constructed using theta functions (see, for example, \cite[$\S$3.3.1]{Mocanu}; 
for the Fourier coefficients of $E_{4,D_5}$, see also Theorem~\ref{thm:Jacobi_Eisenstein_k_1}).

We now provide an explicit formula for the Fourier coefficients of
Jacobi--Eisenstein series $E_{k+\frac{r+1}{2},D_r}$ of index $D_r$ for $k \geq 2$.
Regarding the notation $\beta(r')$, we refer the reader to $\S$\ref{sec:decom}.

\begin{theorem}\label{thm:Jacobi_Eisenstein}
Assume that the weight $k + \frac{r+1}{2}$ is an even integer with $k \geq 2$.
The Jacobi--Eisenstein series $E_{k+\frac{r+1}{2},D_r} \in J_{k+\frac{r+1}{2},D_r}$ has the Fourier expansion:
 \begin{equation*}
   E_{k+\frac{r+1}{2},D_r}(\tau,z)
   =
   \sum_{n' \in \Z} \sum_{r' \in D_r^\sharp}
   e_{r,k}(n',r') e(n'\tau + \beta(r',z)),
 \end{equation*}
 where we put $N = 8 (n' - \beta(r'))$ and
 \begin{equation}\label{id:e_r_k}
   e_{r,k}(n',r') = \begin{cases}
     1, & \mbox{ if } N = 0, \\
     \displaystyle{ \frac{H(k,N) + \left( \frac{8}{r}\right) 2^k H(k,\frac{N}{4})}{ \zeta(1-2k) \left( 1 + \left( \frac{8}{r}\right) 2^k \right) } } , & \text{ if } N  > 0 
     \text{ and } N \equiv 0, 4, -r \mod 8, \\
     0 , & \text{otherwise} .
   \end{cases}
 \end{equation}
 Note that $H(k,N) = 0$ if  $N \notin \Z_{\geq 0}$ or $(-1)^k N {\not \equiv} 0,1 \!\! \mod 4$.
\end{theorem}
\begin{proof}
The function $E_{k+\frac{r+1}{2},D_r}$ (resp. $\mathscr{H}_{r,k}^*$) is an eigenform
for the Hecke operator $T^J(p)$ (resp. $T(p^2)$) for any odd prime $p$
with the eigenvalue $1+p^{2k-1}$ (see, e.g., \cite[Thm.3.3.18]{ali}).
By virtue of Theorems~\ref{thm:J_decom} and \ref{thm:J_new},
both $E_{k+\frac{r+1}{2},D_r}$ and $\mathscr{H}_{r,k}^*$ correspond to $E_{2k}$,
where $E_{2k}$ denotes the usual Eisenstein series in $M_{2k}(1)$ with the constant term $1$.
Therefore, $E_{k+\frac{r+1}{2},D_r}$ and $\mathscr{H}_{r,k}^*$ correspond to each other
via the isomorphism in Proposition~\ref{prop:J_M_even}.
This implies that $\mathcal{J}_{r,k}^{\text{even}}(E_{k+\frac{r+1}{2},D_r}) = c\mathscr{H}_{r,k}^*$
for some $c \in \C$.
By comparing the constant terms, we find $c = \zeta(1-2k)^{-1}$.
The specific formula for $e_{r,k}(n',r')$ follows from the definition of the map
$\mathcal{J}_{r,k}^{\text{even}}$ (see (\ref{id:J_r_k_even_def})).
It is
 \begin{equation*}
   e_{r,k}(n',r') = \begin{cases}
     1, & \mbox{ if } N = 0, \\
     \displaystyle{ \frac{H^*_r(k, N)}{ \zeta(1-2k)} } , & \text{ if } N  > 0 
     \text{ and } N \equiv 0 \mod 4, \\
     \displaystyle{ \frac{H^*_r(k,  N)}{2 \zeta(1-2k)}} , & \text{ if } N  > 0 
     \text{ and } N \equiv -r \mod 8,
   \end{cases}
 \end{equation*}
  where $N = 8 (n' - \beta(r'))$.
  By combining this and Lemma~\ref{lem:H*_formula}, the formula (\ref{id:e_r_k}) follows.
\end{proof}

Recall that $\dim J_{4,D_5} = 1$.
By Theorem~\ref{thm:k_1}, 
there exists $E_{4,D_5} \in J_{4,D_5}$ that corresponds to $E_2^{(2)}$.
Such a function was defined in~\cite{Mocanu}, 
where an explicit formula for its Fourier coefficients is provided \cite[pp.~76--77]{Mocanu}.
We quote this result below.
Using the identity
\begin{equation*}
  \#\left\{ (x_1,x_2,x_3) \in \Z^3 \, : \, 
  -2D =  
  \sum_{i=1}^3( x_i^2 + x_i)
  + \frac34 \right\}
  =
  r_3(-8D)
\end{equation*}
for $-8D \in \mathbb{N}$ with $-8D \equiv 3 \mod 8$, 
we have the following theorem.
\begin{theorem}[\cite{BS,Mocanu}]\label{thm:Jacobi_Eisenstein_k_1}
Let $E_{4,D_5}(\tau,z) = \sum_{n',r'} e_{5,1}(n',r') e(n'\tau + \beta(r',z))$
be the Fourier expansion of $E_{4,D_5} \in J_{4,D_5}$.
Then,
\begin{equation*}
  e_{5,1}(n',r') = \begin{cases}
    r_3(N), & \text{if $N   \equiv 0 \!\! \mod 4$}, \\
    \frac12 r_3(N), & \text{if $N  \equiv 3 \mod 8$}, \\
    0 & \text{otherwise},
  \end{cases}
\end{equation*}
where $N = 8(n' - \beta(r'))$.
\end{theorem}
We remark that there is a misprint in the formula for $C_{4,n}(D,r)$ in \cite[p.~77]{Mocanu}.
Due to the parity condition $\sum_{i=1}^8 x_i \in 2\Z$ defining the $E_8$ lattice,
the term 
\[
  \#\left\{ x \in \Z^{8-n} \, : \, 
  -2D =  x_1^2 + x_1 + \cdots +  x_{8-n}^2 + x_{8-n}  
  + \frac{8-n}{4} \right\}
\]
in \cite[p.~77]{Mocanu} must be halved.

We will show the identity $r_3(N) = 12 (H(1,4N) - 2 H(1,N))$ in $\S$\ref{ss:E_3_2_8}.
By virtue of this identity,
and noting that $\zeta(-1) = -\frac{1}{12}$ and $H^*_5(1,N) = 2H(1,N) - H(1,4N) = - \frac{1}{12} r_3(N)$, 
it follows from Theorem~\ref{thm:Jacobi_Eisenstein_k_1}
that Theorem~\ref{thm:Jacobi_Eisenstein} remains valid for the case $(r,k) = (5,1)$.

\subsection{Zagier-type Eisenstein series of level $8$}\label{ss:E_3_2_8}
We now consider the case $k=1$ for the Cohen-type Eisenstein series.
Let
\begin{equation*}
  \mathscr{H}_1(\tau) := -\frac{1}{12} + \sum_{N>0} H(N) q^N,
\end{equation*}
where we set $H(N) := H(1,N)$.
We define the Zagier--Eisenstein series $\mathscr{F}$ of weight $3/2$ and level $4$ by
\begin{equation*}
  \mathscr{F}(\tau) :=
   \mathscr{H}_1(\tau) + \frac{1}{16 \pi \sqrt{v}} \sum_{n \in \Z} \alpha(n^2 y) q^{-n^2},
\end{equation*}
where $v = \mbox{Im}(\tau)$ and $\alpha(t) := \int_1^{\infty} e^{-4 \pi u t} u^{-3/2} \, du$.
A straightforward calculation shows that
\begin{equation*}
  \mathscr{F}|U_1(4) - 2 \mathscr{F}
  =
  \mathscr{H}_1|U_1(4) - 2 \mathscr{H}_1
  \ = \
  \frac{1}{12} + \sum_{\begin{smallmatrix} N > 0 \\ N \equiv 0, 3 \!\! \mod 4  \end{smallmatrix}} (H(4N) - 2 H(N)) q^N.
\end{equation*}
We then define the level $8$ Zagier-type Eisenstein series 
$  \mathscr{H}_1^{*}(\tau) :=  \mathscr{H}_{5,1}^{*}(\tau)$ by
\begin{equation*}
  \mathscr{H}_1^{*}(\tau) 
   :=
    - ( \mathscr{F}|U_1(4)   -  2 \mathscr{F})
   = 
  -\frac{1}{12} + 
  \!\!\!\!\!
   \sum_{\begin{smallmatrix} N > 0 \\ N \equiv 0, 3 \!\! \mod 4  \end{smallmatrix}} 
   \!\!\!\!\!
   (2 H(N) -  H(4N)) q^N.
\end{equation*}

We now proceed to the proof of Theorem~\ref{thm:E_3_2_8}.
\begin{proof}[Proof of Theorem~\ref{thm:E_3_2_8}]
First, we establish the identity 
\begin{equation}\label{id:three_squares_H}
  r_3(N) = 12 (H(4N) - 2 H(N))
\end{equation}
for any  $N \in \mathbb{N} \cup \{ 0 \}$, 
where we set $H(N) = 0$
if  $-N \equiv 2, 3 \!\! \mod 4$.

The identity $r_3(4N) = r_3(N)$ is a direct consequence of the definition of $r_3$.
For $N$ such that $-N \equiv 0,1 \!\! \mod 4$,
the property~(\ref{id:H_N_Dff}) implies $H(16N) - 2 H(4N) = H(4N) - 2 H(N)$.
Furthermore,  a straightforward calculation yields
\begin{equation*}
  H(4N) - 2 H(N) = 
  \begin{cases}
  H(N), & \text{if $N \equiv 0 \!\! \mod 4$ and $ \frac{N}{4}\, \equiv \, 1,2  \!\! \mod 4$}, \\
  2H(N), & \text{if $N \equiv 3 \!\! \mod 8$}, \\
  0, & \text{if $N \equiv 7 \!\! \mod 8$}.
  \end{cases}
\end{equation*}
As explained in~\cite[p.~274]{cohen}, the right-hand side of this identity 
equals $\frac{1}{12} r_3(N)$ for $N$ satisfying these conditions.
Therefore, the identity (\ref{id:three_squares_H}) holds 
for all $N$ such that $-N \equiv 0,1 \!\! \mod 4$, .

If $-N \equiv 2,3  \! \! \mod 4$, then $H(4N) - 2 H(N) = H(4N) = \frac{1}{12} r_3(4N) = \frac{1}{12} r_3(N)$.
Thus, the identity (\ref{id:three_squares_H}) holds for all $N \in \mathbb{N}$.
For $N = 0$, we have $H(4N) - 2 H(N) = - H(0) = \frac{1}{12} = \frac{1}{12} r_3(0)$.
Consequently, the identity (\ref{id:three_squares_H}) holds for all $N \in \mathbb{N} \cup \{0\}$.

We then obtain
\begin{align*}
  (\theta^3|U_1(4))(\tau) &= 
  \sum_{\begin{smallmatrix} N \geq 0 \\ N \equiv 0, 3 \!\! \mod 4 \end{smallmatrix}}
  r_3(N) q^N
  \ = \
  12 \sum_{\begin{smallmatrix} N \geq 0 \\ N \equiv 0, 3 \!\! \mod 4 \end{smallmatrix}}
  (H(4N) - 2 H(N)) q^N \\
  & = 
   -12 \mathscr{H}_1^*(\tau).
\end{align*}
Note that the constant term of $\theta^3|U_1(4)$ is $r_3(0) = 1$.

By Proposition~\ref{prop:J_M_even} and Theorem~\ref{thm:Jacobi_Eisenstein_k_1},
we have
\begin{equation*}
 \theta^3|U_1(4)  = \sum_{N \equiv 0, 3, 4 \!\! \mod 8} r_3(N) q^N = \mathcal{J}_{5,1}^{\text{even}}(E_{4,D_5})  \in M_{3/2}^{+,-5}(8).
\end{equation*}
Since $\dim M_{3/2}^{+,-5}(8) = \dim J_{4,D_5} = 1$
and the constant terms of $\theta^3|U_1(4)$ and $E_{3/2}^{(8)}$ are both equal to $1$,
we conclude $-12 \mathscr{H}_1^*(\tau)   = \theta^3|U_1(4) = E_{3/2}^{(8)} \in M_{3/2}^{+,-5}(8)$.
This completes the proof of Theorem~\ref{thm:E_3_2_8}.

We remark that, since $\theta^3 \in M_{3/2}(4)$,  the fact $\theta^3|U_1(4) \in M_{3/2}^{+}(8)$ also follows from~\cite[Prop. 6]{UY}.
Furthermore, since $r_3(N) = 0$ for $N \equiv 7 \mod 8$, it also follows that $\theta^3|U_1(4) \in M_{3/2}^{+,-5}(8)$.
\end{proof}

\section{Maps from Jacobi forms to elliptic modular forms}\label{sec:S_d_0}
In this section we prove Theorems~\ref{thm:map_J_even} and \ref{thm:map_J_odd},
which provide explicit formulas for the maps from Jacobi forms of index $D_r$ to elliptic modular forms.

Theorems~\ref{thm:map_J_even} and \ref{thm:map_J_odd}  follow essentially from the facts that
the Hecke eigenvalues of $\phi \in J_{k+\frac{r+1}{2},D_r}$ coincide with the Hecke eigenvalues of $f \in M_{2k}^{\text{new}}(N)$ ($N=1,2$),
provided that $\phi$ corresponds to $f$ under Theorem~\ref{thm:J_decom}, \ref{thm:J_new} or \ref{thm:k_1}.
Furthermore, the Hecke eigenvalues of $f$ coincide, up to a constant, with its Fourier coefficients.

\begin{proof}[Proof of Theorem~\ref{thm:map_J_even}]
Assume that $\phi \in J_{k+\frac{r+1}{2},D_r}$ is a Hecke eigenform.
Let $g =\mathcal{J}_{r,k}^{\text{even}}(\phi)  \in M_{k+\frac12}^{+,-r}(8)$ be the corresponding modular form
with the Fourier expansion $g = \sum A(m) q^m $.

For any odd prime $p$ and any discriminant $m$, it follows from (\ref{id:A^**}) that
  \begin{equation}\label{id:A_Hecke_p}
    \lambda_p A(m) =     
    A(p^2 m) + \left( \frac{(-1)^k m }{p}\right) p^{k-1} A(m) + p^{2k-1} A(m/p^2),
  \end{equation}
  where $\lambda_p$ denotes the Hecke eigenvalue of $\phi$ with respect to $T^J(p)$.
  If we assume  $p^2 {\not|} m$, then this identity implies
  \begin{align*}
    &
    \sum_{u=0}^\infty \sum_{l=0}^u p^{l(k-1)} \left( \frac{(-1)^k m}{p^l} \right) A(p^{2(u-l)} m) X^u \\
    &=
    (1- p^{k-1} \left( \frac{(-1)^k m}{p}\right)X)^{-1} \sum_{t=0}^\infty A(p^{2t}m) X^t \\
    &=
    (1- p^{k-1} \left( \frac{(-1)^k m}{p}\right)X)^{-1} \frac{1- p^{k-1} \left( \frac{(-1)^k m}{p}\right)X}{1-\lambda_p X + p^{2k-1} X^2} A(m)\\
    &=
     \frac{A(m)}{1-\lambda_p X + p^{2k-1} X^2}.
  \end{align*}
  Since $\lambda_p$ is the Hecke eigenvalue of $\phi$,  
  it is also the Hecke eigenvalue of a modular form  
  $f \in M_{2k}^{\text{new},\epsilon_2}(2) \oplus M_{2k}^{\epsilon_1}(1)$ 
  corresponding to $\phi$ via the isomorphism in Theorem~\ref{thm:main}. 
  Without loss of generality, let $f = \sum a_f(n) q^n$ be normalized such that $a_f(1) = 1$. 
  Then we have
  \begin{equation*}
    \frac{A(m)}{1-\lambda_p X + p^{2k-1} X^2} = A(m) \sum_{l=0}^\infty a_f(p^l) X^l.
  \end{equation*}
  Comparing the coefficients of $X^u$, we obtain
  \begin{equation}\label{id:A_sum_p}
    \sum_{l=0}^u p^{l(k-1)} \left( \frac{(-1)^k m}{p^l} \right) A(p^{2(u-l)} m)
    = A(m) a_f(p^u).
  \end{equation}
  Using the multiplicativity of $a_f(n)$, and assuming that $m$ is not divisible by $p^2$ for any odd prime $p$,
  we conclude that
  \begin{equation}\label{id:A_sum_odd}
    \sum_{d | n} d^{k-1} \left( \frac{(-1)^k m}{d} \right) A(n^2 d^{-2} m) 
    = A(m) a_f(n)
  \end{equation}
  for any odd integer $n$.
  
  It remains to verify (\ref{id:A_sum_odd}) for any even integer $n$.
  Since $S_{d_0}$ is a linear map, it is suffices to consider the following three cases:
  \begin{itemize}
  \item[(i)] \quad  $\phi \in J_{k+\frac{r+1}{2},D_r}^{\text{cusp, new}}$,
  \item[(ii)] \quad $\phi \in J_{k+\frac{r+1}{2},D_r}^{\text{cusp, old}}$,
  \item[(iii)] \quad $\phi$ is a Jacobi--Eisenstein series $E_{k+\frac{r+1}{2},D_r}$.
  \end{itemize}
  
  First, we consider case (i), where $\phi \in J_{k+\frac{r+1}{2}, D_r}^{\text{cusp, new}}$,   
  and show that $S_{d_0}(\phi) \in S_{2k}^{\text{new},\epsilon_2}(2)$.
  In this case, the corresponding modular form $f$ belongs to $S_{2k}^{\text{new},\epsilon_2}(2)$.
  Let $g = \mathcal{J}_{r,k}^{\text{even}}(\phi)$. 
  Then $g |U_k(4) = \lambda_2 g$, where 
  the eigenvalue $\lambda_2$ coincides with $a_f(2)$ by Theorem~\ref{thm:UY}.
  By the definition of the $U_k(4)$ operator, it follows that
  $A(4^l m) = \lambda_2 A(4^{l-1} m) = \lambda_2^l A(m) = a_f(2)^l A(m)$
  for any integer $l \geq 1$ and any $m$ satisfying $(-1)^k m \equiv 0,1 \!\! \mod 4$.
  
  By assumption, $(-1)^k d_0 $ is a fundamental discriminant with $d_0 \equiv 0 \mod 4$.
  It follows from the definition of the Kronecker symbol that $\left( \frac{(-1)^k d_0}{2} \right) = \left( \frac{8}{(-1)^k d_0} \right) =  0$.
  Therefore, for any odd integer $n$ and any natural number $l$, we obtain
  \begin{align*}
    a_f(2^l n) A(d_0) &= a_f(2)^l \sum_{d | n} d^{k-1} \left( \frac{(-1)^k d_0}{d} \right) A(n^2 d^{-2} d_0) \\
    &= \sum_{d | 2^l n} d^{k-1}  \left( \frac{(-1)^k d_0}{d} \right) A((2^l n)^2 d^{-2} d_0),
  \end{align*}
  where the second equality holds because the terms for even $d$ in the sum vanish.
  This implies $S_{d_0}(\phi) = A(d_0) f \in S_{2k}^{\text{new},\epsilon_2}(2)$,
  which completes the proof for the case $\phi \in J_{k+\frac{r+1}{2},D_r}^{\text{cusp, new}}$.
  
  Next, we consider case (ii), where $\phi \in J_{k+\frac{r+1}{2},D_r}^{\text{cusp, old}}$, and
  show that $S_{d_0}(\phi) \in S_{2k}^{\epsilon_1}(1)$.
  
  In this case, the corresponding modular form $f$ belongs to $S_{2k}^{\epsilon_1}(1)$.
  The form $\phi$ arises as a Fourier--Jacobi coefficient of the Ikeda lift of $f$ (see $\S$\ref{sec:Ikeda_lift}).
  Let $h = \sum_{n} b(n) q^n \in S_{k+\frac12}^{+}(4)$ be a Hecke eigenform corresponding to $f$
  via the Shimura correspondence,
  where the sum is taken over $n$ such that $n \equiv 0, (-1)^k \!\! \mod 4$.
  Let $\lambda_2$ be the Hecke eigenvalue of $h$ with respect to the Hecke operator $T(2^2)$.
  We define $h^* = \sum_{n \equiv 0, (-1)^k \!\! \mod 4} b^*(n) q^n$ with coefficients
  \begin{equation}\label{id:b^*_b}
    b^*(n) := b(4n) - b(n) \left\{ \left( \frac{8}{r} \right) 2^{k-1} + \lambda_2 \right\}.
  \end{equation}
  Since $h^*$ is a linear combination of $h$ and $h|U_k(4)$, it follows from \cite{UY} that
  $h^* \in S_{k+\frac12}^{+}(8)$.
  Furthermore, as $h$ is a Hecke eigenform, we have the relation
  \begin{equation}\label{id:b_2}
    \lambda_2 b(n) = b(4n) + \left( \frac{(-1)^k n}{2} \right) 2^{k-1} b(n) + 2^{2k-1} b(n/4).
  \end{equation}
  If $n \equiv 4 - r \mod 8$, then $\left( \frac{(-1)^k n}{2} \right) = \left( \frac{(-1)^{\frac{r+1}{2}} (4-r)}{2} \right) = - \left( \frac{8}{r}\right)$ and $b(n/4) = 0$.
  Consequently,  $b^*(n) = 0$ for $n$ such that $n \equiv 4 - r \mod 8$,
  which implies that $h^*$ belongs to $S_{k+\frac12}^{+,-r}(8)$.
  
  Using the identities  (\ref{id:b^*_b}) and (\ref{id:b_2}), a straightforward calculation yields
  \begin{equation*}
    2 b^*(4n) - 2^{k-1} \left( \frac{8}{r} \right) b^*(n) = \lambda_2 b^*(n)
  \end{equation*}
  for $n \equiv -r \mod 8$, and
  \begin{equation}\label{id:b^*_Hecke_2}
    b^*(4n) + 2^{2k-1}  b^*(n/4) = \lambda_2 b^*(n)
  \end{equation}
  for $n \equiv 0 \mod 4$.
  
  The Hecke eigenvalue $\lambda_2$ of $\phi$ at $p=2$ coincides with that of $f$; 
  hence, we have $\lambda_2 = a_f(2)$.
  
  Recalling the assumption that $(-1)^k d_0$ is a fundamental discriminant with $d_0 \equiv 0 \mod 4$, 
  it follows that $\left( \frac{(-1)^k d_0}{2} \right) = 0$.
  Since the identity (\ref{id:b^*_Hecke_2}) is analogous to (\ref{id:A_Hecke_p}),
  a similar argument to the one used for (\ref{id:A_sum_p}) yields
  \begin{equation}\label{id:b^*_sum_2}
       \sum_{l=0}^u 2^{l(k-1)} \left( \frac{(-1)^k d_0}{2^l} \right) b^*(2^{2(u-l)} d_0)
    = b^*(d_0) a_f(2^u),
  \end{equation}
  where $\left( \frac{(-1)^k d_0}{2^l} \right)$ equals $1$ if $l = 0$ and $0$ if $l > 0$.
  
  The functions $h^*$, $h$ and $\phi$ have the same Hecke eigenvalues for all odd primes $p$.
  There exists $\psi \in J_{k+\frac{r+1}{2},D_r}^{\text{cusp}}$
  such that $\mathcal{J}_{r,k}^{\text{even}}(\psi) = h^*$.
  Then, $\phi$ and $\psi$ have the same Hecke eigenvalues for all odd primes $p$.
  By the strong multiplicity one theorem for $S_{2k}(2)$
  and Theorem~\ref{thm:main}, it follows that $\psi$ coincides with $\phi$ up to a constant multiple.
  Consequently, $g = \mathcal{J}_{r,k}^{\text{even}}(\phi)$ and $h^* (= \mathcal{J}_{r,k}^{\text{even}}(\psi))$ also 
  coincide up to a constant multiple.
  Thus, there exists a constant $c$ such that $A(m) = c\, b^*(m)$ for all $m \in \mathbb{N}$.
  It follows that the identity (\ref{id:b^*_sum_2}) remains valid if $b^*(m)$ is replaced by $A(m)$.
  Following an argument similar to the case $\phi \in J_{k+\frac{r+1}{2},D_r}^{\text{cusp, new}}$,
  we conclude that the identity (\ref{id:A_sum_odd}) holds for any $n \in \mathbb{N}$ and any fundamental discriminant $m = (-1)^k d_0$
  such that $d_0 \equiv 0 \!\! \mod 4$.
  Therefore, we obtain $S_{d_0}(\phi) = A(d_0) f \in S_{2k}^{\epsilon_1}(1)$ for $\phi \in J_{k+\frac{r+1}{2},D_r}^{\text{cusp, old}}$.
  
  Finally, we consider case (iii), where $\phi = E_{k+\frac{r+1}{2},D_r}$. 
  We write its Fourier expansion as
  \begin{equation*}
    E_{k+\frac{r+1}{2},D_r}(\tau,z) = \sum_{n' \in \Z} \sum_{r' \in D_r^\sharp} e_{r,k}(n',r') e(n' \tau + \beta(r',z)),
  \end{equation*}
  where the explicit formula for $e_{r,k}(n',r')$ was established in Theorem~\ref{thm:Jacobi_Eisenstein}.

  If $(r,k) \neq (5,1)$, then $E_{k+\frac{r+1}{2},D_r}$ corresponds to the Eisenstein series $E_{2k} \in M_{2k}(1)$, 
  which follows from Proposition~\ref{prop:J_old}.
  In the case $(r,k) = (5,1)$, the series $E_{4,D_5}$ corresponds to $E_2^{(2)} \in M_2^{\text{new}}(2)$ 
  by Theorem~\ref{thm:k_1}.
  For any pair $(r,k)$, an argument similar to the one used above yields the identity (\ref{id:A_sum_odd}) for 
  $g = \mathcal{J}_{r,k}^{\text{even}}(E_{k+\frac{r+1}{2}, D_r}) = \sum A(m) q^m$, 
  for any $n \in \mathbb{N}$ and any fundamental discriminant $(-1)^k d_0$ such that $d_0 \equiv 0 \!\! \mod 4$.
  
  We first consider the case $(r,k) \neq (5,1)$; 
  the case $(r,k) = (5,1)$ can be treated similarly.

  Let $\mathbb{G}_{2k} = \frac{\zeta(1-2k)}{2} E_{2k}$.
  We take the Fourier expansion of $\mathbb{G}_{2k}$ as
  \begin{equation*}
    \mathbb{G}_{2k}(\tau) = \frac{\zeta(1-2k)}{2} +  \sum_{n \geq 1} \sigma_{2k-1}(n) q^n,
  \end{equation*} 
  where $\sigma_m(n) = \sum_{d | n} d^m$.
  Then, the Fourier coefficient of $S_{d_0}(E_{k+\frac{r+1}{2},D_r})$ at $q^n$ is $A(d_0) \sigma_{2k-1}(n)$
  for any $n \geq 1$, which follows from the identity (\ref{id:A_sum_odd}).
  
  To establish the identity $S_{d_0}(E_{k+\frac{r+1}{2}, D_r}) = A(d_0) \mathbb{G}_{2k}$, 
  it suffices to show that the constant term on both sides coincide. 
  
  By the definition of $S_{d_0}$, 
  the constant term of $S_{d_0}(E_{k+\frac{r+1}{2}, D_r})$ is 
  $\frac{A(0)}{2 (1 + \left( \frac{8}{r} \right) 2^k) } L(1-k,\left( \frac{(-1)^k d_0}{ * }\right))$. 
   On the other hand, the constant term of $A(d_0) \mathbb{G}_{2k}$ is
   $A(d_0) \frac{\zeta(1-2k)}{2}$.
   Thus, it remains to verify that
   \begin{equation}\label{id:A_constant}
     \frac{A(d_0)}{A(0)}
     =
     \frac{L(1-k,\left( \frac{(-1)^k d_0}{ * }\right))}{(1 + \left( \frac{8}{r} \right) 2^k)\zeta(1-2k)}.
   \end{equation}
   Recall that $A(0) = e_{r,k}(0,0)$ and $A(d_0) = e_{r,k}(n',r')$ with $d_0 = 8(n' - \beta(r'))$.
   Using Theorem~\ref{thm:Jacobi_Eisenstein} and Lemma~\ref{lem:H*_formula}, 
   together with the assumption on $d_0$,
   we obtain $A(0) = 1$ and $A(d_0) =  \frac{H(k,(-1)^k d_0)}{(1 + \left( \frac{8}{r} \right) 2^k)\zeta(1-2k)}$.
   Since $H(k,(-1)^k d_0) = L(1-k,\left( \frac{(-1)^k d_0}{ \cdot }\right))$, the identity (\ref{id:A_constant}) follows immediately.
   Therefore, we conclude that $S_{d_0}(E_{k+\frac{r+1}{2}, D_r}) = A(d_0) \mathbb{G}_{2k} \in M_{2k}(1)$,
   which proves case (iii).
    
    Since $S_{d_0}$ is a linear map, 
    for any $\phi \in J_{k+\frac{r+1}{2}, D_r} = 
      J_{k+\frac{r+1}{2}, D_r}^{\text{cusp, new}} \oplus J_{k+\frac{r+1}{2}, D_r}^{\text{cusp, old}} \oplus \C E_{k+\frac{r+1}{2},D_r}$,
     we obtain $S_{d_0}(\phi) = A(d_0) f$
    with $f \in M_{2k}^{\text{new},\epsilon_2}(2) \oplus M_{2k}^{\epsilon_1}(1)$.
   This completes the proof of Theorem~\ref{thm:map_J_even}.
\end{proof}

Finally, we provide the proof of Theorem~\ref{thm:map_J_odd}, which addresses the case where the weight
of the Jacobi form is odd.
\begin{proof}[Proof of Theorem~\ref{thm:map_J_odd}]
The proof proceeds similarly to that of Theorem~\ref{thm:map_J_even}. 
By virtue of the identity~(\ref{id:A^*_Hecke}), 
for any odd prime $p$ and any $m \equiv -r \mod 8$, 
we have
  \begin{equation*}
    \lambda_p A(m) = \left( \frac{-4}{p} \right)
    \left\{ 
    A(p^2 m) + \left( \frac{(-1)^k m }{p}\right) p^{k-1} A(m) + p^{2k-1} A(m/p^2) 
    \right\},
  \end{equation*}
where $\lambda_p$ denotes
  the Hecke eigenvalue of $\phi$ with respect to $T^J(p)$.

Let $f \in S_{2k}^{\text{new},\epsilon_2}(2)$ be a Hecke eigenform corresponding to 
$\phi \in J_{k+\frac{r+1}{2},D_r}^{\text{cusp, new}}$
via Theorem~\ref{thm:J_new}.
Let $f = \sum a_f(n) q^n$ be the Fourier expansion of $f$, normalized such that $a_f(1) = 1$.
  
  By an argument analogous to the derivation of (\ref{id:A_sum_odd}), 
  if $m$ is not divisible by $p^2$ for any odd prime $p$, it follows that
    \begin{equation*}
    \left( \frac{-4}{n_1} \right) \sum_{d | n_1} d^{k-1} \left( \frac{(-1)^k m}{d} \right) A(n_1^2 d^{-2} m) 
    =  A(m) a_f(n_1)
  \end{equation*}
  for any odd integer $n_1$.
  
  On the other hand, since $f \in S_{2k}^{\text{new},\epsilon_2}(2)$, it satisfies
  $f|_{2k}\left( \begin{smallmatrix} 0 & -1 \\ 2 & 0\end{smallmatrix} \right) = \epsilon_2 i^{-2k} f$.
  It is known that 
  $f|_{2k}\left( \begin{smallmatrix} 0 & -1 \\ 2 & 0\end{smallmatrix} \right) = - 2^{1-k} a_f(2) f$ (see ~\cite[Introduction]{zhang}), 
  which yields 
  \begin{equation*}
  a_f(2) = - \epsilon_2 i^{-2k} 2^{k-1} 
   =  \left( \frac{-8}{r} \right) (-1)^k 2^{k-1}
   =  \left( \frac{-8}{r}\right) (-1)^{\frac{r-1}{2}} 2^{k-1} 
   =  \left( \frac{8}{r} \right) 2^{k-1}.
  \end{equation*}
  Here, we used the assumption $k + \frac{r+1}{2} \equiv 1 \mod 2$.
  Since the arithmetic function $a_f(n)$ is multiplicative and $a_f(2^{e_2}) = a_f(2)^{e_2}$ holds for $f \in S_{2k}^{\text{new}}(2)$, 
  we obtain
  \begin{align*}
    \left( \left( \frac{8}{r} \right) 2^{k-1} \right)^{e_2}
      \left( \frac{-4}{n_1} \right) \sum_{d | n_1} d^{k-1} \left( \frac{(-1)^k d_0}{d} \right) A(n^2 d^{-2} d_0) 
    &=  A(d_0) a(2^{e_2}) a(n_1) \\
    & =  A(d_0) a(2^{e_2} n_1)
  \end{align*}
  for any fundamental discriminant $(-1)^{k-1} d_0$ such that $d_0 \equiv -r \mod 8$ and $d_0 > 0$.
  
  This implies $S_{d_0}(\phi) = A(d_0) f \in S_{2k}^{\text{new},\epsilon_2}(2)$ for any Hecke eigenform $\phi \in J_{k+\frac{r+1}{2},D_r}^{\text{cusp,  new}}$.
  By the linearity of the map $S_{d_0}$, 
  we conclude the proof of Theorem~\ref{thm:map_J_odd}.
\end{proof}

\vspace{1cm}

\noindent
Department of Mathematics, Joetsu University of Education,\\
1 Yamayashikimachi, Joetsu, Niigata 943-8512, JAPAN\\
e-mail hayasida@juen.ac.jp


\begin{thebibliography}{99}
 \bibitem[Aj 15]{ali}
   A.~Ajouz: Hecke operators on Jacobi forms of lattice index and the relation to elliptic modular forms,
   Ph.D. thesis, University of Siegen. \\
   URL: https://www.researchgate.net/publication/285601134
 \bibitem[Ar 98]{ara}
   T.~Arakawa: K\"ocher-Maass Dirichlet series corresponding to Jacobi forms and Cohen Eisenstein series
   {\it Comment. Math. Univ. St. Paul.} {\bf 47} (1998), no. 1, 93--122.
  \bibitem[Boe 83]{Bo}
   S.~Boecherer: \"Uber die Fourier--Jacobi-Entwicklung Siegelscher Eisensteinreihen, 
   {\it Math.\ Z.}\ {\bf 183} (1983), 21--46.
 \bibitem[Boy 11]{boy}
   H.~Boylan:
   Jacobi forms, quadratic modules and Weil representations over number Fields.
   PhD thesis, (2011). URL:  http://dokumentix.ub.uni-siegen.de/opus/volltexte/2012/597/.
 \bibitem[B-S 23]{BS}
  H.~Boylan and N.-P.~Skoruppa: 
  Jacobi forms of lattice index I. Basic theory. Preprint (2023) arXiv:2309.04738.
 \bibitem[B-K 00]{BK}
 S.~Breulmann and M.~Kuss:
  On a conjecture of Duke-Imamo\=glu.
  {\it Proc. Amer. Math. Soc.} \ {\bf 128} (2000), no. 6, 1595--1604.
 \bibitem[Br 06]{bringmann}
   K.~Bringmann:
   Lifting maps from a vector space of Jacobi cusp forms to a subspace of elliptic modular forms,
   {\it Math. Z.} \ {\bf 253} (2006), no. 4, 735--752.
 \bibitem[Co 75]{cohen}
  H.~Cohen: Sums involving the values at negative integers of L-functions of quadratic characters,
  {\it Math. Ann.} \ {\bf 217}, (1975), 271--285.
 \bibitem[Co 93]{cohen:computation}
  H.~Cohen: A course in computational algebraic number theory,
  {\it Grad. Texts in Math.}, \ {\bf 138} Springer-Verlag, (1993)
  {\it Math. Ann.} \ {\bf 217}, (1975), 271--285.  
\bibitem[Ha 11]{FJlift}
 S.~Hayashida~: Fourier--Jacobi expansion and the Ikeda lift,
    {\it Abh. Math. Semin. Univ. Hambg.} {\bf 81} (2011), 1--17.
\bibitem[Ik 01]{Ik}
  T.~Ikeda : On the lifting of elliptic cusp forms to Siegel cusp forms of degree 2n, 
         {\it Ann.\ of Math.\ (2)}\ {\bf 154}, no.3 (2001), 641--681.
\bibitem[Ko 80]{Ko}
  W.~Kohnen : Modular forms of half integral weight on $\Gamma_{0}(4)$, 
	 {\it Math,\ Ann.}\ {\bf 248} (1980), 249--266.
\bibitem[Mi 89]{Miyake}
  T.~Miyake:  Modular Forms, Springer-Verlag, Berlin, (1989).
\bibitem[Mo 19a]{Mocanu_JN}
  A.~Mocanu: Poincar\'e and Eisenstein series for Jacobi forms of lattice index.
  {\it J. Number Theory}\ {\bf  204} (2019), 296--333.
\bibitem[Mo 19b]{Mocanu}
  A.~Mocanu: On Jacobi forms of lattice index, Ph.D. thesis, University of Nottingham (2019),  158 pages.
\bibitem[Mu 89]{Murase}
  A.~Murase: L-functions attached to Jacobi forms of degree n. I. The basic identity.,
  J. Reine Angew. Math. {\bf 401} (1989), 122--156.
\bibitem[S-Z 88]{SZ}
  N.-P.~Skoruppa and D.~Zagier: Jacobi forms and a certain space of modular forms,
  Invent. Math. {\bf 94}, (1988), 113--146.
 \bibitem[U-Y 10]{UY}
   M.~Ueda and S.~Yamana:
   On newforms for Kohnen plus spaces.
   {\it Math. Z.} \ {\bf 264} (2010), no. 1, 1--13.
 \bibitem[W 19]{williams}
  K.S.~Williams:
  Some formulas of Liouville in the spirit of Gauss.
  {\it Math. Mag.} \ {\bf 92} (2019), no. 2, 83--90.   
 \bibitem[Ya 14]{YY}
  Y.~Yang: Modular forms of half-integral weights on SL(2,$\Z$), Nagoya Math. J. {\bf 215} (2014), 1--66.
 \bibitem[Za 75]{zagier}
  D.~Zagier: 
  Nombres de classes et formes modulaires de poids 3/2,
  {\it C. R. Acad. Sci. Paris S\'{e}r. A-B 281} (1975), A883--A886.
 \bibitem[Zh 13]{zhang}
  Y.~Zhang: Eigenvalue of Fricke involution on new forms of level $4$ and of trivial character,
  Preprint (2013) arXiv:1308.3417v1.
 \bibitem[Zi 89]{Zi} 
  C.~Ziegler: Jacobi forms of higher degree, {\it Abh.\ Math.\ Sem.\ Univ.\  
  Hamburg.}\ {\bf 59} (1989), 191--224.  
\end{thebibliography}
\end{document}